\documentclass{amsart}

\usepackage{amsmath}
\usepackage{amssymb,amsmath}
\usepackage{comment}
\usepackage{graphicx}
\usepackage{subfigure}
\usepackage{ upgreek }
\usepackage{mathtools} 
\usepackage{tikz}
\usepackage{url}
\usepackage{amscd}
\usepackage[all]{xy}
\usetikzlibrary{matrix}

\newcommand{\bQ}{{\mathbb Q}}
\newcommand{\bR}{{\mathbb R}}
\newcommand{\bZ}{{\mathbb Z}}
\newcommand{\bN}{{\mathbb N}}

\newcommand{\bK}{{\mathbb K}}

\newcommand{\cF}{\mathcal F}

\newcommand{\cA}{\mathcal A}

\newcommand{\cC}{\mathcal C}

\newcommand{\cP}{\mathcal P}

\newcommand{\cL}{\mathcal L}

\newcommand{\cH}{\mathcal H}

\DeclareMathOperator{\val}{val}

\DeclareMathOperator{\id}{id}
\DeclareMathOperator{\im}{im}

\DeclareMathOperator{\gr}{gr}

\newtheorem{conj}{Conjecture}

\newtheorem{tm}{Theorem}[section]

\newtheorem{cy}[tm]{Corollary}
\newtheorem{lm}[tm]{Lemma}
\newtheorem{prp}[tm]{Proposition}
\newtheorem{rem}[tm]{Remark}
\newtheorem{df}[tm]{Definition}
\newtheorem{ques}{Question}

\newtheorem{ex}[tm]{Example}

\newtheorem{thmA}{Theorem}

\newtheorem{thmB}{Theorem}

\newtheorem{mainthm}{Theorem}

\newtheorem*{thmAsummary}{Theorem A (summary)}
\newtheorem*{thmBsummary}{Theorem B (summary)}
\newtheorem*{thmCsummary}{Theorem C (summary)}

\begin{document}
\title{Boundary Depth and Deformations of Symplectic Cohomology}
\author{Yoel Groman}
\address{ Yoel Groman,
Hebrew University of Jerusalem,
Mathematics Department}
\begin{abstract}
    We study the relation between the ambient symplectic cohomology with supports associated to a Liouville domain \(D\subset M\), which depends on the embedding and is defined over the Novikov ring, and the intrinsic symplectic cohomology of \(D\), which depends only on its local Liouville geometry. We first prove that the leading-order reduction of the ambient theory with respect to relative symplectic area is intrinsic. Under a quantitative hypothesis expressed in terms of the boundary depth of the intrinsic complex, this leading-order comparison becomes a genuine deformation model: the ambient theory is obtained as a controlled deformation of intrinsic symplectic cohomology, compatibly with restriction and the basic Floer operations. When the boundary depth is finite, we study the size of this deformation through an invariant \(\tau\), and prove that it is concave under variation of the Liouville primitive and monotone under exact inclusions of Liouville subdomains with compatible primitives for which the Viterbo restriction map is injective. These results provide the closed-string input to a local-to-global approach to Floer theory. We outline a conjectural strategy for proving homological mirror symmetry in SYZ settings by combining computable intrinsic local models with Fukaya categories with supports and descent.
    \end{abstract}

\maketitle

\section{Introduction}
\label{SecIntroduction}

\subsection{A conceptual overview}
\label{subsec:conceptual}
Floer-theoretic invariants are defined by counting solutions of nonlinear PDEs, but direct counts are rarely an effective means of computation. In exact symplectic geometry, by contrast, one has access to a much richer collection of geometric, topological, and algebraic tools. In Weinstein and Lefschetz settings in particular one can additionally exploit handle decompositions, surgery, and explicit models such as the relation between cotangent bundles and loop spaces~\cite{CieliebakEliashberg,SeidelBook,AbbondandoloSchwarz2006}. This raises a natural question: can such computable intrinsic Floer theories be used as local models inside general symplectic manifolds?

Floer theory with supports is the natural framework for such a local-to-global strategy~\cite{Varolgunes,GromanVarolgunes2021, AbouzaidGromanVarolgunes2026FukayaSupports}. Let us focus on the closed-string variant: to a compact subset $K\subset M$ one associates the group
\[
SH^*_M(K),
\]
the \emph{symplectic cohomology of $M$ with support on $K$}. This assignment is contravariantly functorial under inclusions of compact subsets. Its principal advantage is descent. If $\{D_i\}$ is an
involutive cover of $M$, meaning that it can be described using
Poisson-commuting defining functions, the Mayer--Vietoris theorem for
symplectic cohomology with supports~\cite{Varolgunes} relates the groups $SH^*_M(D_i)$
and their intersections to the global theory. Thus, if the supported invariants of the pieces could be recovered from their intrinsic Floer theories, one would obtain a mechanism for turning local computations into global ones. When $M$ is closed, for example,
\[
SH^*_M(M)\simeq QH^*(M;\Lambda).
\]
The phrase ``with support in \(D\)'' should not, however, be confused with intrinsic locality. The complex
\[
SC^*_M(D)
\]
is defined using Hamiltonian Floer theory in the ambient manifold \(M\), and its Floer trajectories are not confined to \(D\). Consequently, the invariant \(SH^*_M(D)\) generally depends on the embedding \(D\hookrightarrow M\), and is not determined by the symplectic geometry of \(D\) alone.

By contrast, if \(D\) is a Liouville domain and \(\theta\) is a Liouville primitive defined near \(D\), there is an intrinsic complex
\[
SC^*_\theta(D)=SC^*_{\widehat D,\theta}(D),
\]
defined entirely in the completion \(\widehat D\) and over the ground ring. More generally, every compact subset \(K\subset D\) has an intrinsic supported complex
\[
SC^*_{\widehat D,\theta}(K).
\]

This leads to the basic question of the paper:
\begin{ques}
To what extent is \(SC^*_M(K)\) determined by \(SC^*_{\widehat D,\theta}(K)\), and what additional data are needed to recover the ambient theory?
\end{ques}

The answer has two levels. First, we prove that at leading order in relative symplectic area the ambient Floer complex is intrinsic: it carries a filtration whose leading-order reduction is the intrinsic Floer complex of the Liouville domain. Thus the intrinsic theory is already visible inside the ambient one. But this leading-order identification is only a formal approximation. By itself, the filtration gives no convergence statement for the higher-order terms, since positivity with respect to the relative valuation does not imply smallness in the ambient Floer norm.

The second level is quantitative. The key observation is that Liouville shrinking separates the intrinsic and ambient scales: keeping the reference domain \(D\) fixed, one shrinks the supported subdomain \(K=tD\) toward the skeleton, so that the intrinsic boundary depth decreases while the ambient relative-area gap associated with \(D\) remains fixed. Thus, when the boundary depth is finite, one can shrink until \(\beta(K)<\hbar(M,D,\theta)\) and homological perturbation becomes convergent. The ambient theory is then identified with a controlled deformation of the intrinsic theory of the shrunken domain \(t\cdot D\).

In this sense, the missing ingredient in the naive local-to-global strategy is an embedding-dependent deformation of the intrinsic local theory:
\[
\left\{\text{intrinsic local theory}\right\}
\;+\;
\left\{\text{ambient deformation}\right\}
\;+\;
\left\{\text{descent}\right\}
\quad\rightsquigarrow\quad
\left\{\text{global theory}\right\}.
\]
The present paper constructs and quantitatively controls this deformation, and studies the size of its first nonzero term. A more geometric description of the deformation data is left conjectural and is discussed below from the perspective of neck stretching.

The local-to-global approach is particularly natural in SYZ-type situations.
There, a Lagrangian torus fibration over polytopes in the base supplies
especially tractable local models, while the deformation and gluing
data carried by the ambient theory have a geometric interpretation on
the mirror side.  We return later to the
corresponding open-string picture, where one expects an analogous role
for Fukaya categories.

\subsection{Main results}

Throughout, $M$ is either closed or geometrically of finite type (see \S\ref{sec:floer-package}), and
$D\subset M$ is a compact Liouville domain with smooth boundary. We fix
a Liouville primitive $\theta$ on a neighborhood of $D$. Let $R$ be an
integral domain, let $\Lambda=\Lambda_R$, and write
\[
\Lambda_{\geq 0}=\{a\in\Lambda:\val(a)\geq 0\}.
\]
 When classical transversality is unavailable, we work over a characteristic-zero field and assume the existence of a Floer-local theory of virtual counts in the sense of Definition~\ref{dfLocalVirtualCounts}. See Appendix~\ref{SecVirtualChains} for a discussion of the locality hypothesis and its relation to the construction of \cite{Abouzaid2022}.

\paragraph{\textbf{Formal locality.}}At leading order in relative symplectic area, the ambient Floer theory is intrinsic. Theorem A makes this precise by constructing a relative Novikov lattice whose reduction recovers the intrinsic theory.
\begin{thmAsummary}
There is a constant $\hbar=\hbar(M,D,\theta)>0$ and, for every compact $K\subset D$, a valuation $\operatorname{val}_{M,\theta}$ on $SC_M^\ast(K)$, \emph{the relative valuation}, whose nonnegative part \[ SC^{\ast,\geq 0}_{M,\theta}(K)\subset SC_M^\ast(K) \] is a closed $\Lambda_{\geq 0}$-subcomplex, the \emph{relative Novikov lattice}, such that the following hold: \begin{enumerate}

        \item \emph{Reduction.}   There is a quasi-isomorphism \[ SC^{\ast,\geq 0}_{M,\theta}(K)/ T^\hbar SC^{\ast,\geq 0}_{M,\theta}(K) \simeq SC^\ast_{\widehat D,\theta}(K;R) \otimes_R\Lambda_{[0,\hbar)} . \] 
        \item \emph{Recovery.} After extending scalars to $\Lambda$, the relative Novikov lattice is norm-dense in $SC_M^\ast(K)$. In particular, its completion with respect to the ambient norm recovers $SC^*_M(K)$.
        \item \emph{Compatibility with structure.} The relative Novikov lattices and their reductions modulo $T^\hbar$ are natural, up to chain homotopy, under restriction of supports. The pair-of-pants product and BV operator admit representatives preserving the relative valuation, and their reductions agree with the corresponding intrinsic operations.
        \end{enumerate}

\end{thmAsummary}

The precise statement of Theorem~\ref{tmFormalLocality} is formulated and proved in \S\ref{ssFormalLocality}.

Taking $K=D$, the $T^\hbar$-adic filtration of the relative Novikov lattice gives a
one-sided locality spectral sequence
\begin{equation}
\label{eqLocalitySpectralSequence}
E_1^{p,q}
=
SH^{p+q}_\theta(D;R)
\otimes_R
\Lambda_{[p\hbar,(p+1)\hbar)},
\qquad p\geq 0.
\end{equation}
Its higher differentials record successive obstructions to locality.
Because completion may produce a non-Hausdorff or nonexhaustive
filtration on homology, this spectral sequence is not, by itself, a
reconstruction theorem for the completed ambient invariant; see the
analysis of Example~\ref{exCP1} at the end of \S\ref{SubSecLocalSS}.

\paragraph{\textbf{Boundary depth.}}
The missing quantitative invariant is the Floer-theoretic \emph{boundary depth} introduced by Usher~\cite{Usher2011BoundaryDepth}, adapted here to the intrinsic complex. We write
\[
\beta(K)
=
\beta\bigl(SC^*_{\widehat D,\theta}(K;\bQ)\bigr),
\]
and in particular $\beta(D)=\beta(SC^*_\theta(D;\bQ))$.
The ambient correction has operator norm strictly smaller than
$e^{-\hbar}$, where $\hbar=\hbar(M,D,\theta)$ is the fixed geometric constant,
whereas a right inverse to the intrinsic differential has size controlled by $e^{\beta}$.
Thus, if the boundary depth is finite, shrinking the supported domain \(K=tD\) toward the skeleton
makes $\beta(K)$ arbitrarily small while the ambient gap $\hbar(M,D,\theta)$ associated with the fixed reference domain stays put, so that once
$\beta(K)<\hbar$ the intrinsic differential dominates and the deformation
becomes strictly contractive. This is nicely
illustrated by Example~\ref{exCP1}.
Precise definitions,
properties, and further examples pertaining to boundary depth are given in
Section~\ref{SecBoundaryDepth}.

\paragraph{\textbf{Quantitative deformation theory.}}
The boundary-depth condition allows us to upgrade formal locality of Theorem A to a genuine
deformation model. For Theorem~B and Theorem~C, we specialize to coefficients
in~$\bQ$.

\begin{thmBsummary}
    Let \(K\subset D\) satisfy
    \[
    \beta(K)<\hbar(M,D,\theta).
    \]
    Then the formal comparison of Theorem~A can be promoted to a quantitatively controlled deformation model for the ambient Floer complex.
    
    \begin{enumerate}
    \item \emph{Deformation model.}
    There is a differential
    \[
    d_{\mathrm{def}}
    :
    SH^*_{\widehat D,\theta}(K;\mathbb Q)\widehat\otimes\Lambda
    \longrightarrow
    SH^{*+1}_{\widehat D,\theta}(K;\mathbb Q)\widehat\otimes\Lambda
    \]
    and a deformation retraction
    \[
    SC^*_M(K)
    \longrightarrow
    \left(
    SH^*_{\widehat D,\theta}(K;\mathbb Q)\widehat\otimes\Lambda,
    d_{\mathrm{def}}
    \right),
    \]
    with
    \[
    \operatorname{val}_{M,\theta}(d_{\mathrm{def}})\geq \hbar,
    \qquad
    |d_{\mathrm{def}}|_\infty<e^{-\hbar}.
    \]
    
    \item \emph{Leading order.}
    The deformation retraction induces an isometry on cohomology, and its leading-order reduction agrees with the intrinsic comparison of Theorem~A. Thus \(d_{\mathrm{def}}\) records the higher-order correction in the transferred model.
    
    \item \emph{Compatibility with structure.}
    The deformation model is natural, up to chain homotopy, under restriction of supports. The ambient BV operator transfers to a quantitatively small deformation of the intrinsic BV operator, and the ambient pair-of-pants product induces a bounded product on the cohomology of the deformed complex. When \(d_{\mathrm{def}}=0\), this gives a deformation of the intrinsic product itself.
    \end{enumerate}
    \end{thmBsummary}

The precise statement of Theorem \ref{tmQuantitativeDeformation} is formulated and proved in \S\ref{SubSecProofThmA}. 
\begin{rem}
    \begin{enumerate}
        \item
Although $\hbar$ is defined through the ambient Floer problem,
Lemma~\ref{lmOutsideContributions} shows that it admits a positive lower
bound determined only by the fixed local geometry near $D$,
independently of the geometry of $M\setminus D_\delta$.
\item For closed $M$, a discussion of the PSS map in connection with deformation theory is given in Lemma~\ref{lmPSSOrderZero} and Corollary~\ref{cyPSSRestrictionInterface}.
\item The operations considered in Theorems A and B form only a small truncation of the structure carried by symplectic cochains. As shown in \cite{AbouzaidGromanVarolgunes24} symplectic cochains with supports carry a framed-$E_2$ structure, functorially with respect to restriction of supports. We expect that after choosing a cofibrant resolution
\[
\mathcal P_\infty\longrightarrow C_*(E_2^{\mathrm{fr}}),
\]
it should be possible, after additional technical work, to upgrade Theorems A and B to statements about $\mathcal P_\infty$ algebras. 
    \end{enumerate}
\end{rem}
\paragraph{\textbf{The size of the deformation.}}
Suppose first that
\[
\beta(D)<\hbar(M,D,\theta).
\]
Then Theorem B, applied with \(K=D\), produces a deformation differential
\[
d_{\mathrm{def}}
\]
on \(SH^*_\theta(D;\mathbb Q)\widehat\otimes\Lambda\). We measure the the deformation by the real number
\[
\tau(D,\theta,i,M)
=
\operatorname{val}_{M,\theta}(d_{\mathrm{def}}).
\]
Thus larger \(\tau\) corresponds to a smaller deformation. We set \(\tau=+\infty\) when \(d_{\rm def}=0\).

More generally, assume only that \(\beta(D)<\infty\). Choose a sufficiently small Liouville rescaling \(K=tD\subset D\), keeping \(D\) fixed as the reference domain. Then \(\beta(K)<\hbar(M,D,\theta)\), so Theorem B produces a deformation differential \(d_{\mathrm{def}}^{\,t}\) for the supported theory of \(K\). Proposition 7.1 shows that
\[
\operatorname{val}_{M,\theta}(d_{\mathrm{def}}^{\,t})
\]
is independent of the sufficiently small choice of \(t\), and agrees with the valuation of the first nonzero higher differential in the locality spectral sequence for \(D\). We use this common value to define \(\tau(D,\theta,i,M)\) in the finite-boundary-depth case.

\begin{thmCsummary}
Fix a symplectic embedding $i\colon D\hookrightarrow M$ and suppose $\beta(D)<\infty$. Then the invariant 
\[ 
\tau(D,\theta,i,M)\in \mathbb R\cup\{+\infty\} 
\] 
has the following properties. 
\begin{enumerate} 
    \item \emph{Invariance.} It is independent of the choices going into the construction other than, possibly, a choice at infinity in the noncompact case. 
    \item \emph{Dependence on the relative class.} For fixed $D$, $i$, and $M$, the invariant depends on the symplectic data only through the relative cohomology class \[ [\omega,\theta]\in H^2(M,D;\mathbb R). \] 
    \item \emph{Concavity.} Let $U\subset H^1(D;\mathbb R)$ be a connected convex open set consisting of classes $[\eta]$ for which $\theta+\eta$ is again a Liouville primitive. Then \[ [\eta]\longmapsto \tau(D,\theta+\eta,i,M) \] is either identically $+\infty$, or finite everywhere and concave on $U$. 
    \item \emph{Monotonicity.} If $D_1\subset D_2$ is an exact inclusion of Liouville domains with compatible primitives and finite boundary depth, and if the intrinsic Viterbo restriction map \[ SH^\ast(D_2;\bQ)\longrightarrow SH^\ast(D_1;\bQ) \] is injective, then \[ \tau(D_1)\leq \tau(D_2). \] 
    \item \emph{Liouville rescaling.} The invariant is unchanged under Liouville rescaling: \[ \tau(t\cdot D)=\tau(D) \] whenever the rescaled domain is defined.
\end{enumerate}
\end{thmCsummary}

A complete formulation and proof is given in \S\ref{SubSecTheoremC}.

\begin{rem}
The invariant \(\tau\) measures the first unary obstruction to locality, and \(\tau=+\infty\) precisely when the locality spectral sequence collapses at \(E_1\). This should be distinguished from the stronger locality theorem for complete embeddings proved in \cite{GromanVarolgunes2021}: in that setting the entire  ambient deformation of the local model can be trivialized by the choice of a complete embedding. 
\end{rem}

\subsection{Locality and SYZ geometry}
\label{subsec:propagation}
Consider a Lagrangian torus fibration
\[
\pi:M\longrightarrow B
\]
possibly with singularities. Theorem C gives rise in this setting to two geometric mechanisms for establishing locality:
propagation along families of Lagrangian tori, and constraints coming from large star shape.

\paragraph{\textbf{Propagation.}}
Over the regular locus \(B_0\), action--angle coordinates identify neighborhoods of nearby fibers with a fixed Liouville model, while the corresponding Liouville primitives vary through \(H^1(T^n;\mathbb R)\). The $\tau$ invariant of neighborhoods of the fibers is given by a function on $B_0$. The concavity dichotomy of Theorem~C therefore implies that collapse of the locality spectral sequence propagates along paths in \(B_0\). In particular, if \(B_0\) is connected, collapse for one regular fiber implies collapse for every regular fiber. The boundary-depth hypothesis is automatic here, since standard convex cotangent-torus neighborhoods have $\beta=0$; see Corollary~\ref{cyConvexTorusBetaZero}.

The same comparison applies to singular local models. Suppose a Liouville domain $D_{\mathrm{sing}}\subset M$ has finite boundary depth and contains an exact Liouville subdomain $D_{\mathrm{reg}}$, with compatible primitive, which is a neighborhood of an exact Lagrangian torus Hamiltonian isotopic to a regular fiber. Suppose moreover that the Viterbo restriction
\[
SH^*(D_{\mathrm{sing}};\bQ)
\longrightarrow
SH^*(D_{\mathrm{reg}};\bQ)
\]
is injective. Then locality for $D_{\mathrm{reg}}$ forces locality for $D_{\mathrm{sing}}$. The finite-boundary-depth condition can be verified directly in various standard local models and is expected more generally for the singular models of interest in SYZ mirror symmetry.
\paragraph{\textbf{Star shape.}}
Following Shelukhin--Tonkonog--Vianna~\cite{ShelukhinEtAl}, the \emph{star shape} records the flux classes realized by straight symplectic isotopies of the embedding of~$D$. When the star shape is large, possible ambient corrections must survive many changes of relative area and are therefore strongly constrained in relative homology.
Proposition~\ref{tmTauConstant} makes this precise: if the star shape contains sufficiently many unbounded directions, then the first nonzero higher differential, if one exists, must preserve the
\[
\Gamma=H_1(D;\mathbb Z)/\operatorname{Tor}
\]
grading. Concretely, this holds if the star shape contains either a line whose direction has trivial annihilator in \(\Gamma\), or rays whose nonnegative span is all of \(H^1(D;\mathbb R)\). If \(M\) is exact, these constraints force the locality spectral sequence to degenerate at \(E_1\).

We expect the exactness assumption in the last statement to be unnecessary in the case of Lagrangian tori: conjecturally, a sufficiently large star shape forces locality outright. The expected explanation comes from the neck-stretching description of the ambient deformation together with its \(L_\infty\)-structure, discussed in \S1.5. We do not pursue this strengthening here.

\subsection{Proof overview}

The conceptual overview of \S1.1 distinguishes two levels of comparison between
ambient and intrinsic Floer theory: a formal leading-order comparison and a
quantitative deformation statement. Both arise from the same Floer-theoretic
construction, but from two different filtrations, or scales.

Fix a Liouville domain \(D\subset M\) and a compact subset \(K\subset D\).
We compute \(SH^*_M(K)\) using specially chosen \emph{S-shaped} Floer data.
Near \(\partial D\) the Hamiltonians have a small fixed radial slope and the
almost complex structures are of contact type, while outside a neighborhood of
\(D\) the Hamiltonians are dynamically small; the cofinal growth needed to
compute the supported theory takes place on the interior.  The resulting collar
acts as a barrier separating the intrinsic and exterior parts of the Floer
problem.

The first scale is the \emph{relative}, or topological, filtration.  A Floer
solution \(u\) determines a relative class in \(H_2(M,D)\), and the change of
relative valuation along \(u\) is its relative symplectic area
\[
E_{M,\theta}(u)=\langle[\omega,\theta],[u]\rangle.
\]
The central geometric input is a uniform trichotomy for S-shaped Floer
solutions.  The contact-type collar and barrier equation give an integrated
maximum-principle estimate, while the smallness of the exterior Hamiltonian
gives geometric energy quantization.  As a result, after choosing
\(\hbar>0\) below the relevant geometric thresholds, every contribution either
has relative energy at least \(\hbar\), is entirely local, or has relative
energy zero with exterior output.  The same estimates hold for the continuation
data and for the Floer operations used in the paper.

This trichotomy produces the relative Novikov lattice.  Modulo \(T^\hbar\),
only the relative-order-zero part remains.  At order zero the exterior
generators form a contractible subcomplex, while the quotient consists exactly
of the local generators and local Floer trajectories.  The integrated maximum
principle identifies this quotient with the corresponding intrinsic Floer
complex in \(\widehat D\).  This gives the reduction statement of Theorem~A;
the same S-shaped continuation and operation data give its naturality and its
compatibility with the product and BV operator.

The second scale is the \emph{analytic} filtration coming from the ordinary
Floer norm, or equivalently from geometric energy.  Relative positivity alone
does not imply convergence in this norm, which is why Theorem~A by itself is
only a formal locality statement.  The additional geometric input is a uniform
exterior energy gap: every compactified exterior component with nonzero relative energy has
\[
E_{\mathrm{geo}}>e_0
\]
for some \(e_0>0\) determined by the fixed local geometry near \(D\).
The proof combines a graph-packing estimate with a topological argument showing
that a relative class with nonzero relative energy is nontrivial and must contain a uniformly wide slice.  We choose
\(\hbar\) below the resulting geometric thresholds, so that the
positive-relative-order part of the Floer differential is not merely filtered
but strictly small in the analytic norm.

Boundary depth is what allows this analytic smallness to be used.  Finite
boundary depth gives a deformation retract of the intrinsic complex onto its
cohomology whose homotopy has norm controlled, up to an arbitrarily small
loss, by \(e^{\beta}\).  When
\[
\beta(K)<\hbar(M,D,\theta),
\]
the positive-order ambient correction composed with this homotopy is therefore
strictly contractive.  The homological perturbation lemma then converges and
transfers the ambient differential to a deformation differential
\(d_{\mathrm{def}}\) on intrinsic symplectic cohomology.  The same argument,
using the quantitative estimates for continuation maps and operations, gives
the naturality and structural statements of Theorem~B.  Keeping \(D\) fixed as
the reference domain while shrinking \(K=tD\) toward the skeleton makes the
intrinsic boundary depth arbitrarily small without changing the relevant
ambient gap; this is what extends the construction from the small-boundary-depth
regime to arbitrary finite boundary depth.

Finally, Theorem~C is obtained by keeping track of the first nonzero
relative-order term in this construction.  Simultaneous S-shaped Floer data
allow the primitive and the Liouville domain to vary while retaining the same
underlying Floer configurations.  Their relative energies vary affinely with
the relative cohomology class, which leads to the concavity statement for
\(\tau\).  The same simultaneous-data construction gives comparison under
changes of domain and hence monotonicity, while Liouville rescaling follows
from the scaling behavior of the intrinsic models.  Thus the formal,
quantitative, and deformation-theoretic parts of the paper all arise from the
same S-shaped Floer model, viewed through its two complementary filtrations.

\subsection{Neck stretching and deformation}
\label{subsec:geometric-meaning}
Theorem~B captures the unary part of the ambient deformation. We expect the full deformation, including its higher operations, to be governed by a Maurer--Cartan element in the $L_\infty$ structure underlying symplectic cochains. The discussion in this subsection is within the purview of a collaboration with K.~Siegel.

Let
\[
\mathfrak g(D)
\]
denote the complete filtered shifted \(L_\infty\) algebra carried by
\(SC^*_\theta(D)\widehat\otimes\Lambda\). Similarly, let
\[
\mathfrak g_M(D)
\]
denote the complete filtered shifted \(L_\infty\)-algebra carried by the ambient
complex \(SC^*_M(D;\Lambda)\). The framed $E_2$ structure on intrinsic
symplectic cochains is expected to induce, by an equivariant or
gravity construction, a complete filtered shifted \(L_\infty\) structure
\[
\mathfrak g^+_{S^1}(D)
\]
on the positive \(S^1\)-equivariant symplectic cochains.

We further expect the marking map to admit a filtered \(L_\infty\)
enhancement
\[
\rho:\mathfrak g^+_{S^1}(D)\longrightarrow \mathfrak g(D),
\]
whose linear term agrees with the Gysin, or marking, map wherever
the latter is defined on the positive Floer complex. When
\(\partial D\) has no Reeb orbits contractible in \(D\), this target
causes no additional difficulty, since the positive and full Floer
complexes agree in the free-homotopy classes containing the relevant
Reeb orbits. In general, the existence of \(\rho\) should be understood
as an additional conjectural assertion.

An \(L_\infty\) enhancement \(\rho\) induces
\[
\rho_*:MC(\mathfrak g^+_{S^1}(D))
   \longrightarrow MC(\mathfrak g(D)).
\]
All Maurer--Cartan elements are required to have positive Novikov
valuation.

\begin{conj}[SFT deformation conjecture]
The exterior SFT theory of \(M\setminus D\) determines a
Cieliebak--Latschev Maurer--Cartan element~\cite{CieliebakLatschev}
\[
x_{\mathrm{CL}}\in MC(\mathfrak g^+_{S^1}(D)),
\]
well defined up to gauge equivalence, such that there is a filtered
\(L_\infty\) quasi-isomorphism
\[
\mathfrak g_M(D)\simeq_{L_\infty}
\mathfrak g(D)^{\rho_*(x_{\mathrm{CL}})}.
\]
\end{conj}

The assertions above are expected to arise from two complementary neck-stretching procedures.
The first, \emph{neck-stretching of the target}, produces the
Cieliebak--Latschev element from exterior punctured curves. The further assertions involving the map $\rho$ are expected to follow from a domain-dependent neck-stretching construction which separates the
equivariant exterior component from the marked Floer component. The
higher Taylor coefficients of \(\rho\) encode the boundary strata in
which several exterior components interact before being attached to the
Floer component.

Domain-dependent neck stretching also suggests the following conjecture, which is independent of SFT machinery.

\begin{conj}[Divisor deformation conjecture]
Suppose $V\subset M$ is a symplectic normal-crossings divisor representing the
symplectic form, in the sense required for the corresponding divisor
compactification, and suppose $D\subset M\setminus V$. Then there is a
Maurer--Cartan element
\[
x\in MC(\mathfrak g(D)),
\]
well defined up to gauge equivalence, and a filtered $L_\infty$
quasi-isomorphism
\[
\mathfrak g_M(D)\simeq_{L_\infty}\mathfrak g(D)_x.
\]
\end{conj}

This is closely related to the work of \cite{Borman2024}. It would be
interesting to determine whether their method can be adapted to prove this
statement in the above generality.

\begin{rem}[Trivializations and wall-crossing]
        \begin{enumerate}
        \item Suppose first that there are no contractible Reeb orbits. Then, in either picture, the Maurer--Cartan element is expected to encode planes in \(M\setminus D\) asymptotic to Reeb orbits on \(\partial D\). In the general case, the SFT picture is more naturally described in terms of buildings: the exterior level consists of genus-zero curves in \(M\setminus D\) with several outputs asymptotic to Reeb orbits on \(\partial D\), all but one of which are capped by planes in \(D\).
        
        \item When the resulting Maurer--Cartan element is gauge equivalent to zero, a choice of gauge trivialization yields an \(L_\infty\) quasi-isomorphism with the undeformed local model
        \[
        \mathfrak g(D).
        \]
        Such a trivialization need not be canonical. Different choices may therefore give rise to nontrivial automorphisms of the local model, which should be interpreted as wall-crossing transformations. The same picture is expected at the homotopy-BV level.
        
        \item The locality theorem for complete embeddings \cite{GromanVarolgunes2021} provides a concrete model for this phenomenon. In the Maurer--Cartan description above, one expects a complete embedding to determine a gauge trivialization of the corresponding Maurer--Cartan element. The work \cite{GromanVarolgunes2022} shows concretely that different complete embeddings can give different identifications with the local model, and that the transition maps between these identifications are the relevant wall-crossing transformations.
     \item The conjectural \(L_\infty\) picture also suggests why a sufficiently large star shape should force full locality for torus models. Here the intrinsic \(L_\infty\)-algebra is formal, so only the binary bracket survives. Since the Maurer--Cartan element is supported on noncontractible boundary orbits, every nonzero term carries nontrivial relative boundary class. By Proposition \ref{tmTauConstant} a sufficiently large star shape would force any leading term to have trivial boundary class, hence no such term can occur. This suggests that the Maurer--Cartan element is gauge equivalent to zero.
    \end{enumerate}

\end{rem}

\subsection{The open-string perspective}

The preceding results are motivated in part by an open-string local-to-global
picture.  For closed \(M\), in joint work with Abouzaid and
Varolgunes~\cite{AbouzaidGromanVarolgunes2026FukayaSupports}, we
construct Fukaya categories with supports, together with restriction functors and
Mayer--Vietoris descent for weakly involutive covers.  The objects are \emph{global}
Lagrangian branes in \(M\), but the morphism complexes depend on the support \(K\).
The construction of~\cite{AbouzaidGromanVarolgunes2026FukayaSupports} presently treats tautologically
unobstructed objects.  We expect a virtual-chain extension in which supported
Fukaya categories are defined in the usual curved setting and carry compatible
open--closed and closed--open maps.  When \(D\) is a Liouville domain, we expect
an open-string analogue of the theory developed here, comparing the ambient
morphism complexes of Lagrangian branes with the intrinsic Floer complexes of
their intersections with \(D\), whenever these intersections define intrinsic
Lagrangian branes.

This leads to the following conjecture, the subject of joint work with Abouzaid
and Varolgunes.  We say that a full subcategory
\[
\mathcal B\subset \mathcal W(D)
\]
satisfies the \emph{quantitative generation criterion} if its open--closed map
\[
OC_{\mathcal B}\colon CC_*(\mathcal B)\longrightarrow SC^*(D)
\]
is a quasi-isomorphism and the acyclic complex
\[
\operatorname{Cone}(OC_{\mathcal B})
\]
has finite boundary depth with respect to its intrinsic action filtration.

Now let \(\mathcal B^M\) be the full subcategory spanned by a collection of
unobstructed Lagrangian branes in \(M\) whose intersections with \(D\) are
invariant under the Liouville flow.  Their intersections with \(D\), completed
along \(\partial D\), define a full intrinsic subcategory
\[
\mathcal B^D\subset\mathcal W(D).
\]
Under Liouville rescaling, the action filtration on the intrinsic Floer complexes
is multiplied by the scaling factor.  We expect the same scaling for the
Hochschild complex and the open--closed map, and hence for the boundary depth of
\(\operatorname{Cone}(OC_{\mathcal B^D})\).

\begin{conj}\label{conjLocToAmbGen}[Local-to-ambient generation]
If \(\mathcal B^D\) satisfies the quantitative generation criterion, then, for
all sufficiently small \(t>0\), \(\mathcal B^M\) generates the Fukaya category
of \(M\) with support in \(tD\).
\end{conj}

The intrinsic generation hypothesis is particularly simple over the regular
locus of a Lagrangian torus fibration which admits a Lagrangian section.  For a small convex polytope \(P\) in the
regular base, the corresponding Liouville piece is a cotangent-torus model
\[
       T^n\times P\subset T^*T^n,
\]
and an intrinsic Lagrangian section is a cotangent-fibre type object.
Generation in the completed cotangent model is the standard theorem of
Abouzaid~\cite{Abouzaid-cotangent-fibre}.  The quantitative strengthening required
by our criterion is expected to follow directly from the same explicit cotangent-bundle
model, although we do not include a proof here.  Thus, aside from this technical
refinement, the genuinely new step over the regular locus is
Conjecture~\ref{conjLocToAmbGen}.  For singular charts, the same strategy leads
to an additional intrinsic local question: does the section satisfy the quantitative
generation criterion for the corresponding singularity model?  Even when the answer
is negative, or when there is no section altogether, we expect a broader version of
the local-to-global picture above to remain applicable.  A full discussion is outside
the scope of the present work.

Combined with the descent theorem of~\cite{AbouzaidGromanVarolgunes2026FukayaSupports}, this suggests a
local-to-global approach to homological mirror symmetry.  One computes the
intrinsic Fukaya categories of the regular and singular pieces of a Lagrangian
torus fibration, promotes local generators to generators of the ambient
supported categories, and then reconstructs the global Fukaya category by
descent.  In favorable situations the global problem is thereby reduced to a
finite collection of explicit local Floer-theoretic calculations.

This differs from Seidel's deformation-theoretic approach to Fukaya categories
relative to a divisor~\cite{Seidel-deformations}.  In Seidel's approach, one starts with a Liouville domain \(D\) obtained from a projective variety \(M\) by removing an ample divisor \(V\). Seidel proposes to understand the corresponding part of the Fukaya theory of \(M\) as a deformation of the Fukaya theory of the exact divisor complement \(D\). In the supports picture the Lagrangians remain global objects of \(M\);
changing the support localizes their morphism theory, and restriction maps relate the resulting local categories.  Intrinsic Fukaya categories therefore serve as computable models for ambient supported categories rather than as categories whose objects must themselves lie in a fixed divisor complement.  This feature is what makes the framework naturally suited to a \v{C}ech-type local-to-global argument.

\begin{rem}
A separate geometric input in this program is the construction of suitable
involutive covers.  Toric degenerations of Calabi--Yau hypersurfaces and complete
intersections are expected to yield Lagrangian torus fibrations with a
small number of local singularity types, though this has not yet been
established in the required generality. Alternatively, one may start from
an integral affine manifold with singularities arising from a
non-Archimedean SYZ fibration~\cite{Nicaise2019} and seek a corresponding
symplectic model.  In this direction, the relation between tropical geometry
and unobstructed Lagrangian lifts provides one mechanism connecting the
symplectic and non-Archimedean pictures; see~\cite{hicks2022}.
\end{rem}

\subsection{Organization of the paper}
The paper is organized as follows.  Section~2 reviews the basic definitions
of symplectic cohomology.  Section~3 introduces S-shaped Hamiltonians and
constructs the relative filtration.  Section~4 proves Theorem~A on formal locality
and introduces the locality spectral sequence.  Section~5 studies boundary depth for Liouville domains.
Section~6 develops the required homological perturbation theory and proves
Theorem~B.  Section~7 studies the \(\tau\) invariant, flux, and
Theorem~C.  Appendix~\ref{SecEstimates} contains the packing and topological
estimates used in Proposition~\ref{PrpRelHbar}, while
Appendix~\ref{SecVirtualChains} formulates the locality axiom for virtual
counts and its consequences for Hamiltonian Floer multimorphisms.

\subsection{Acknowledgments}
The author was supported by the ISF (grant no.~2445/20 and no.~3605/24).

The author benefited greatly from conversations with Kyler Siegel and with Nick Sheridan concerning Maurer--Cartan theory and deformations. Long-standing collaborations with U.~Varolgunes on relative symplectic cohomology and with U.~Varolgunes and M.~Abouzaid on the Fukaya category with supports have contributed to this work in ways too numerous to spell out in detail.

The author used OpenAI’s ChatGPT during the revision of this manuscript as an aid in proof auditing, identifying possible gaps and inconsistencies, and improving exposition. All mathematical claims and arguments were independently evaluated by the author, who takes full responsibility for their correctness.
% Overview of Symplectic Cohomology section
\section{Basics}

\subsection{Notation and conventions}\label{sec:notation-conventions}
Fix an integral domain $R$ and denote by
\[
\Lambda=\Lambda_{R}
\]
the Novikov ring over $R$, consisting of finite or countable sums
\[
\sum_i a_iT^{\lambda_i}
\]
with every displayed coefficient $a_i\in R$ nonzero, exponents $\lambda_i\in\bR$ strictly increasing, and, in the infinite case, $\lambda_i\to+\infty$. Write
\[
\Lambda_{\geq 0}:=\Bigl\{\sum_i a_iT^{\lambda_i}\in\Lambda: \lambda_i\geq 0\Bigr\}
\]
for the nonnegative Novikov ring.
When $R$ is a field, $\Lambda_R$ is a field.

For an element $x=\sum_i a_iT^{\lambda_i}\in\Lambda$ we write
\[
\val(x)=\min\{\lambda_i\mid a_i\neq 0\},
\]
with the convention
\[
\val(0)=+\infty,
\]
and
\[
|x|=e^{-\val(x)}.
\]
More generally, for a module $M$ over $\Lambda$, a \textit{non-Archimedean valuation} is a map $\val:M\to\bR\cup\{+\infty\}$ satisfying
\[
\val(x+y)\geq \min\{\val(x),\val(y)\},
\qquad
\val(\lambda x)=\val(\lambda)+\val(x),
\]
together with the separation axiom
\[
\val(x)=+\infty\quad\Longleftrightarrow\quad x=0.
\]
The associated non-Archimedean norm is
\[
|x|=e^{-\val(x)}.
\]
We use the notation $\Lambda_{[a,b)}$ for the $\Lambda_{\geq0}$-module of elements with valuation $\geq a$ modulo those with valuation $\geq b$. When $a=0$, $\Lambda_{[0,b)}$ is a unital ring.

\begin{rem}

    We use half-open notation \([a,b)\) with respect to the filtration parameter under consideration. In this paper the parameter is the Novikov valuation, so
    \[
    \Lambda_{[a,b)}=F^{\ge a}\Lambda/F^{\ge b}\Lambda.
    \]
    In \cite{Groman2023SYZLocal} the filtration is parametrized instead by the logarithmic norm
    \(\log|x|=-\operatorname{val}(x)\). Under this orientation-reversing change of parameter, a window denoted \([a,b)\) there corresponds in valuation coordinates to \((-b,-a]\). We retain the respective conventions so that half-open filtration windows are written \([a,b)\) in both settings.
\end{rem}

Some of the main results require $R$ to be a field $\bK$ of characteristic zero, though many intermediate results hold for arbitrary ground integral domains $R$.

\begin{rem}\label{remGrading}
    Both ambient and intrinsic invariants are a priori $\bZ/2\bZ$-graded. When $2c_1(M)=0$ we lift this to a $\bZ$-grading by fixing a grading structure on $M$ and use its restriction to $D$ for the intrinsic theory. This may involve a non-trivial choice when $D$ is not simply connected. 

For Weinstein neighborhoods of Lagrangians $L$, a natural grading structure exists, and it may be extended to $M$ exactly when its relative Maslov class
$$
\mu_{M,L}: \pi_2(M,L)\to\bZ
$$
vanishes. The choice of extension is generally immaterial when we consider Floer theory with support on one compact set $D$. 

In the rest of the paper we leave the grading ($\bZ$ with compatible grading data, or $\bZ/2\bZ$) unspecified.

\end{rem}

\subsection{The Hamiltonian Floer package}\label{sec:floer-package}
Our Floer cohomology conventions follow \cite{GromanVarolgunes2021}.
Throughout, $M$ is either closed or geometrically of finite type. We recall the latter notion. A function $f$ on a geometrically bounded symplectic manifold is called
\emph{admissible} if it is proper, bounded below, and
$\|X_f\|_{g_J}$ and $\|\nabla X_f\|_{g_J}$ are uniformly bounded with
respect to a geometrically bounded almost complex structure $J$.
A symplectic manifold is \emph{geometrically of finite type} if it admits a
compatible geometrically bounded almost complex structure $J_0$ and an
admissible exhaustion $f_0$ whose critical points all lie in a compact set;
cf.\ \cite{GromanVarolgunes2021}.
In the noncompact case we fix once and for all such a pair $(J_0,f_0)$.
By \cite[Lemma~3.1]{GromanVarolgunes2021}, for all sufficiently small
$\epsilon>0$ a Hamiltonian equal to $\epsilon f_0+C$ outside a compact set
has no nonconstant $1$-periodic orbits sufficiently far out and is
dissipative (providing
necessary $C^0$ estimates). The same applies to sufficiently small allowed perturbations. For the notion  dissipative Floer data see an overview in
\cite[Appendix A]{GromanVarolgunes2021}. The model example for a manifold of geometrically finite type is the completion of a Liouville domain, with $f_0$ radial and $J_0$ conical at infinity.

\emph{From now on every Hamiltonian which occurs as a Floer object in this paper is assumed
nondegenerate and to have finitely many $1$-periodic orbits}
\footnote{This concerns Hamiltonians which define Floer complexes; we do not
require every intermediate Hamiltonian in a continuation path or
higher-parameter homotopy to be nondegenerate or to have finitely many
periodic orbits.}.

We first treat the $c_1(M)=0$ case with classical transversality.

\subsubsection{The classical transversality package}\label{sec:classical_transversality}
Assume $c_1(M)=0$ with a fixed trivialization of the complexified canonical bundle of $M$. For a Hamiltonian $H$, denote by $X_H$ the vector field satisfying $\omega(X_H,\cdot)=dH$. The trivialization associates an orientation line $o_{\gamma}$ and Conley-Zehnder index $i_{cz}(\gamma)$ to each $1$-periodic orbit $\gamma$. We also consider the semi-positive case, in which we interpret $i_{cz}(\gamma)\in\bZ/2\bZ$.

Fix a time-dependent $\omega$-compatible almost complex structure $J$ with $(H,J)$ regular for Floer cohomology (see \cite[Definition 3.13]{AbouzaidGromanVarolgunes24}). Consider Floer trajectories $u:\bR\times S^1\to M$ solving
\begin{equation}\label{eqTransInvFloer}
\partial_su+J(\partial_tu-X_H)=0.
\end{equation}
The end at $-\infty$ is the \emph{input}, the end at $+\infty$ the \emph{output}. The orientation of the $X_H$ orbit at the input is opposite to boundary orientation; at the output they coincide.

The underlying $\Lambda$-module of the Floer complex $CF^*(H,J)$ is the
direct sum
\begin{equation}
CF^*(H,J)=\bigoplus o_{\gamma}\otimes\Lambda[i_{cz}(\gamma)]
\end{equation}
where the sum is over all $1$-periodic orbits and $i_{cz}(\gamma)$ denotes the Conley-Zehnder index.

A Floer trajectory   $u$ with input $\gamma_1$ and output $\gamma_2$ defines a map $d_u:o_{\gamma_1}\to o_{\gamma_2}$. The Floer differential $d:CF^*(H,J)\to CF^*(H,J)$ is then defined by
\begin{equation}
d:=\sum_u T^{E_{top}(u)}d_u,
\end{equation}
where
$$
E_{top}(u):=\int u^*(\omega+d(Hdt)).
$$

The sum is over all finite energy rigid solutions. It converges in the $T$-adic topology on $CF^*(H,J)$. To see this, write $E_{geo}(u)=\int \|du-X_H\|^2$. For translation-invariant Floer data one has the exact identity
\begin{equation}\label{eqEtopEgeo}
E_{\mathrm{top}}(u)=E_{\mathrm{geo}}(u).
\end{equation}
Under the stated regularity and $c_1=0$/semipositivity assumptions, together with the dissipative $C^0$ estimates in the noncompact case \cite{Groman}, the zero- and one-dimensional moduli spaces used for the differential, continuation maps, and their boundary identities have the usual Gromov--Floer compactifications.

For regular dissipative data $(H_0,J_0)\leq (H_1,J_1)$, a dissipative monotone path $s\mapsto (H_s,J_s)$ with $(H_s,J_s)=(H_0,J_0)$ for $s\leq 0$ and $(H_s,J_s)=(H_1,J_1)$ for $s\geq 1$ induces a continuation map $\kappa:CF^*(H_0,J_0)\to CF^*(H_1,J_1)$ via
$$
\kappa=\sum T^{E_{top}(u)}\kappa_u,
$$
for $\kappa_u:o_{\gamma_1}\to o_{\gamma_2}$ the isomorphism of orientation lines induced by the solution $u$ with asymptotics on $\gamma_1,\gamma_2$.

For $s$-dependent monotone continuation data one has $E_{\mathrm{top}}(u)\ge E_{\mathrm{geo}}(u)$. Ordinary monotonicity $\partial_sH_s\geq 0$ therefore yields nonnegative Novikov exponents, so continuation preserves the $\Lambda_{\geq0}$-lattice. If $\partial_sH_s$ is merely bounded below, continuation is still defined after allowing the full Novikov ring $\Lambda$.

Given admissible monotone paths $(H_s^0, J_s^0)$ and $(H_s^1, J_s^1)$ from $(H_0,J_0)$ to $(H_1,J_1)$ and an admissible two-parameter interpolation $(H_{s,t}, J_{s,t})$, chosen dissipative in the noncompact setting and regular (or interpreted through the virtual package), we obtain a chain homotopy
$$
K: CF^*(H_0,J_0) \to CF^{*-1}(H_1,J_1)
$$
of degree $-1$ between the continuation maps $\kappa_0,\kappa_1$ via
$$
K = \sum T^{E_{top}(v)} K_v
$$
where the sum is over rigid solutions $v$ to the 2-parameter Floer equation, and $K_v: o_{\gamma_1} \to o_{\gamma_2}$ are induced orientation line maps. Then $\kappa_1 - \kappa_0 = d \circ K + K \circ d$.

\subsubsection{The general setting: virtual chains}
\label{subsec:virtual_chains_general}

We now drop the regularity assumptions used in
\S\ref{sec:classical_transversality}. In the general virtual-chain
setting the coefficient ring is a field $\bK$ of characteristic zero.
In the noncompact case all Floer data are assumed dissipative, as in
\cite[Remark~10.2]{Abouzaid2022}.

We use the framework of \cite{Abouzaid2022}. A theory of virtual
counts in the sense of \cite[Definition~1.5]{Abouzaid2022}, together
with the Hamiltonian Floer multifunctor
\[
\mathbb X:\mathcal F\longrightarrow\mathrm{Flow}
\]
of \cite[Theorem~10.3]{Abouzaid2022}, produces a Floer multifunctor
with values in chain complexes over $\bK$. 

For a Hamiltonian $H$, let $\mathcal P_H$ denote its set of
nondegenerate $1$-periodic orbits and let $\mathbb X_H$ denote the
associated Kuranishi flow category. For $p,p'\in\mathcal P_H$ and
energy label $\lambda$, the theory of virtual counts gives an
operation
\[
m_{\lambda,p,p'}:o_p\longrightarrow o_{p'}.
\]
The Floer differential is
\[
d=\sum_{p,p'}\sum_{\lambda}T^\lambda m_{\lambda,p,p'}.
\]
The energy label is the total topological energy of the stable Floer
configuration, including the energies of attached sphere bubbles; see
\cite[\S11.2]{Abouzaid2022}. The colimit of the footprints of the
Kuranishi presentation associated to fixed $p,p'$ and $\lambda$ is the
corresponding compactified moduli space of stable Floer configurations
\cite[Proposition~11.4]{Abouzaid2022}.

The same formalism supplies continuation maps, continuation
homotopies, cubically parameterized families, and multi-input Floer
operations, compatibly with faces, gluing, and multicomposition
\cite[Theorem~10.3 and Proposition~12.9]{Abouzaid2022}. Compact
parameter cycles are represented by finite cubical chains; in
particular the $S^1$-family defining the BV operator is represented by
a finite cubical $1$-cycle. For monotone continuation data the
construction factors through the nonnegative Novikov theory; after
extension of scalars we use the corresponding maps over $\Lambda$.

Throughout the virtual-chain setting we fix a Floer-local theory of
virtual counts in the sense of Definition~\ref{dfLocalVirtualCounts}.

\subsection{Indexing categories and telescopes}

\subsubsection{Indexing categories}
Symplectic cohomology in its various flavors is defined as a colimit of Hamiltonian Floer cohomologies. For the chain-level constructions below we use the following $2$-truncated indexing structure. This notion will be reused repeatedly throughout the paper for different contexts.

\begin{df}\label{dfIndexingCategory}
An \emph{indexing category} $\cH$ consists of:
\begin{itemize}
\item A set of objects $\cH_0$ of \emph{$\cH$-admissible Floer data}, consisting of pairs $(H,J)$ of Hamiltonians and almost complex structures.

\item For each pair of objects $(H_0,J_0), (H_1,J_1) \in \cH_0$, a set $\cH_1((H_0,J_0), (H_1,J_1))$ of \emph{$\cH$-admissible paths}, consisting of smooth families $s \mapsto (H_s, J_s)$ for $s \in \bR$ with compactly supported derivative, such that $(H_s, J_s) = (H_0, J_0)$ for $s \ll 0$ and $(H_s, J_s) = (H_1, J_1)$ for $s \gg 0$, together with their finite broken concatenations.

\item For each pair of paths $\gamma_1, \gamma_2 \in \cH_1((H_0,J_0), (H_1,J_1))$, a set $\cH_2(\gamma_1, \gamma_2)$ of \emph{$\cH$-admissible two-parameter homotopies} interpolating between $\gamma_1$ and $\gamma_2$, including the standard gluing homotopies that compare a broken path with a smoothing.
\end{itemize}
\end{df}

\begin{rem}\label{remIndexingCategoryTruncation}

    An indexing category can be viewed as the $2$-skeleton of a cubically enriched category. Given such an indexing category,  either consisting of regular Floer data or equipped with a compatible theory of virtual counts, the Hamiltonian Floer package defines a $2$-truncated enriched functor to the category of chain complexes over the Novikov ring: objects are sent to Floer complexes, admissible paths to continuation maps, and admissible two-parameter homotopies to chain homotopies.

\end{rem}

\begin{df}\label{dfAccelerationDatum}
Given an indexing category $\cH$, an \emph{acceleration datum $A$} in $\cH$ is a family of pairs $\tau \mapsto (H_{\tau}, J_{\tau})$ of Floer data together with a smooth function $f: \bR \to \bR$ such that $f(t) \equiv 0$ for $t \ll 0$ and $f(t) \equiv 1$ for $t \gg 1$. These are required to satisfy:
\begin{itemize}
\item For each $\tau$, the pair $(H_{\tau}, J_{\tau})$ is $\cH$-admissible.
\item For each $i \in \bN$, the continuation path $g_i: s \mapsto (H_{i+f(s)}, J_{i+f(s)})$ is $\cH$-admissible
\end{itemize}

The acceleration datum is said to be \emph{regular} if:

\begin{itemize}
\item For each $i \in \bN$, the pair $(H_i, J_i)$ is regular for the definition of Floer cohomology.
\item For each $i \in \bN$, the continuation path $g_i: s \mapsto (H_{i+f(s)}, J_{i+f(s)})$ is regular for the definition of continuation maps.
\end{itemize}

In the virtual chains setting, we only impose that for each $i \in \bN$, the $1$-periodic orbits of $H_i$ are non-degenerate.
\end{df}

\subsubsection{The telescope construction}\label{subsecTelescope}

An acceleration datum $A$ in $\cH$ produces a $1$-ray of chain complexes over $\Lambda$: $$\mathcal{C}(A):= CF^*(H_1)\xrightarrow{\kappa_1} CF^*(H_2)\xrightarrow{\kappa_2}\ldots. $$
When the paths are monotone, the maps preserve a distinguished $\Lambda_{\geq0}$-lattice. We omit $J$ when clear from context. Define the telescope $tel(\cC(A))=\bigoplus_{i=1}^\infty\ CF^*(H_i)[q]$ with $q$ a degree $-1$ variable satisfying $q^2=0$.  Abbreviating $C_i=CF^*(H_i)$, the differential is
\begin{align}\label{teles}
\xymatrix{
C_1\ar@{>}@(ul,ur)^{d }  &C_2\ar@{>}@(ul,ur)^{d} &C_3\ar@{>}@(ul,ur)^{d}\\
C_1[1]\ar@{>}@(dl,dr)_{-d} \ar[u]^{\text{id}}\ar[ur]^{\kappa_1} &C_2[1]\ar@{>}@(dl,dr)_{-d} \ar[u]^{\text{id}}\ar[ur]^{\kappa_2}&\ldots\ar[u]^{\text{id}}_{\ldots} }
\end{align}
given on generators $a\in C_i$ by
\begin{equation}\label{eqTelDiff}
\delta a:=da,
\qquad
\delta qa:=q\,da+(-1)^{\deg(a)}(\kappa_i(a)-a).
\end{equation}

\subsubsection{Morphisms between telescopes}

A morphism between acceleration data $A^1,A^2$ in $\cH$ requires:
\begin{itemize}
\item A $1$-morphism $(H^1_i, J^1_i)\to(H^2_i, J^2_i)$ in $\cH$ for each $i \in \bN$.
\item A $2$-morphism in $\cH$ providing a homotopy between the composition of  $(H^1_i, J^1_i)\to(H^1_{i+1}, J^1_{i+1})$ with  $(H^1_{i+1}, J^1_{i+1})\to(H^2_{i+1}, J^2_{i+1})$, and the composition of $(H^1_i,
J^1_i)\to(H^2_i, J^2_i)$ with  $(H^2_i, J^2_i)\to(H^2_{i+1}, J^2_{i+1})$.
\end{itemize}

This yields a diagram between the $1$-rays $\cC(A^1),\cC(A^2)$:
\begin{align}\label{eqRayMorphism}
\xymatrix{
\mathcal{C}_1\ar[r]^{\iota_1}\ar[d]_{f_1}\ar[dr]_{h_1}& \mathcal{C}_2\ar[d]_{f_2}\ar[r]^{\iota_2}\ar[dr]_{h_2} &\mathcal{C}_3\ar[r]^{\iota_3}\ar[d]_{f_3}\ar[dr]_{h_3}& \ldots \\ \mathcal{C}_{1}'\ar[r]_{\iota_1'} &\mathcal{C}_2'\ar[r]_{\iota_2'}&\mathcal{C}_3'\ar[r]_{\iota_3'}&\ldots}.
\end{align}
Here $h_i$ is a homotopy between $f_{i+1}\circ\iota_i$ and $\iota'_i\circ f_i$. It is shown in \cite{Varolgunes} that a morphism of $1$ rays induces a morphism on the associated telescopes.

\subsection{Symplectic cohomology of K relative to M}\label{SecSHReview}

\subsubsection{The monotone indexing category}

\begin{df}\label{dfMonotoneIndexing}
    The \emph{monotone indexing category} $\cH_{monotone}$ is defined by:
    \begin{itemize}
    \item Objects: pairs $(H,J)$ which are dissipative and with $H$ non-degenerate.
    \item Paths: composable sequences of smooth dissipative paths $s \mapsto (H_s, J_s)$ for $s \in \bR$ with compactly supported derivative and satisfying the monotonicity condition
       $$\frac{\partial H_s}{\partial s}\geq 0.$$
    \item Homotopies:
     smooth $1$-parameter families of dissipative data interpolating between paths, including the standard gluing homotopies.
\end{itemize}

\end{df}

\begin{df}\label{dfAccDatKsubM}
 An \emph{acceleration datum $A=(H_{\tau},J_{\tau},f)$  for $K\subset M$} is an $\cH_{monotone}$-admissible acceleration datum such that for each $x\in M$ we have $$\lim_{\tau\to\infty }H_{\tau}(x)=\begin{cases}0,&\quad x\in K\\\infty,&\quad x\not\in K\end{cases}.$$
 We additionally require $H_\tau<0$ on $K$ for all $\tau$.
\end{df}

For the monotone indexing category, continuation maps preserve the $T$-adic filtration and therefore the $\Lambda_{\geq0}$-lattice inside the underlying $\Lambda$-module. Assign to $CF^*(H_i,J_i)[q]$ the valuation with $\val(\mathrm{orbit})=0$, $\val(T)=1$, $\val(q)=0$, and associated norm $|\cdot|=e^{-\val(\cdot)}$. Define
$$SC_M^*(K,A;\Lambda):=\widehat{tel}(\mathcal{C}(A)),$$
the completed normed complex over $\Lambda$. Monotonicity ensures that the telescope carries a distinguished $\Lambda_{\geq0}$-lattice, but the ambient relative complex itself is a completed $\Lambda$-module. The relative symplectic cohomology of $K$ is
\begin{equation}
    SH^*_M(K):=H^*(SC^*_M(K)).
\end{equation}

In keeping with \S2.1, define valuation submodules
\[
F^{\ge a}CF^*(H,J)
=
\{x:\operatorname{val}(x)\ge a\}
\]
and
\[
CF^*_{[a,b)}(H,J)
=
F^{\ge a}CF^*(H,J)/F^{\ge b}CF^*(H,J).
\]
The same convention is used for the telescope and for $\Lambda_{[a,b)}$. Since monotone continuation maps preserve the valuation filtration, we obtain the truncated complex $SC^*_{M,[a,b)}(K)$. Define
\begin{equation}\label{eqTruncatedColimit}
    SH^*_{M,[a,b)}(K):=\varinjlim_{(H,J)}HF^*_{[a,b)}(H,J)
\end{equation}
where the colimit is over dissipative $(H,J)$ with $H<0$ on $K$.

\begin{lm}\label{lmTruncatedTelescope}
$H^*(SC^*_{M,[a,b)}(K))=SH^*_{M,[a,b)}(K)$.
\end{lm}

\begin{proof}
Modulo $F^{\ge b}$, completion adds no new elements: a Cauchy tail in the $T$-adic topology eventually has valuation at least $b$, hence vanishes in the $[a,b)$-quotient. Therefore the completed truncated telescope is the ordinary telescope of the truncated Floer complexes. The ordinary mapping telescope computes the filtered colimit, so
\[
H^*(SC^*_{M,[a,b)}(K))
\cong
\varinjlim HF^*_{[a,b)}(H,J).
\]
When $M$ is closed the chosen acceleration ray is cofinal in the indexing system defining \eqref{eqTruncatedColimit}; hence the cohomology of its ordinary truncated telescope computes the colimit appearing there. In the noncompact case this conclusion follows from continuity under monotone approximation and the existence of common lower refinements; see \cite[Theorem~8.9 and Lemma~8.14]{Groman}
\end{proof}

\emph{Henceforth we shall omit $\Lambda$ from the notation and write $SC^*_M(K,A)=SC^*_M(K,A;\Lambda)$ unless there is a chance of confusion.}

\subsubsection{Restriction maps}

Acceleration data $A_1, A_2$ in the monotone indexing category yield rays $\cC(A_1),\cC(A_2)$. Morphism data between $A^1,A^2$ gives a ray morphism
\begin{align}
\xymatrix{
\mathcal{C}_1\ar[r]^{\iota_1}\ar[d]_{f_1}\ar[dr]_{h_1}& \mathcal{C}_2\ar[d]_{f_2}\ar[r]^{\iota_2}\ar[dr]_{h_2} &\mathcal{C}_3\ar[r]^{\iota_3}\ar[d]_{f_3}\ar[dr]_{h_3}& \ldots \\ \mathcal{C}_{1}'\ar[r]_{\iota_1'} &\mathcal{C}_2'\ar[r]_{\iota_2'}&\mathcal{C}_3'\ar[r]_{\iota_3'}&\ldots}.
\end{align}
This induces a continuous, norm-contracting map between completed telescopes.

For $K_1\subset K_2$ with acceleration data $A^1,A^2$, write $A^2\leq A^1$ if $H^2_{\tau}\leq H^1_{\tau}$ for all $\tau$. Construct a monotone family interpolating from $A^2$ to $A^1$ with $H^s_{\tau}$ monotone in $s$ for each $\tau$. Imposing regularity (and dissipativity if non-compact), we obtain a ray morphism: $f_i$ from continuation maps (fix $\tau=i$, vary $s$), and $h_i$ from homotopies. The resulting map is the \emph{restriction map} from $K_2$ to $K_1$. Homotopies between monotone families yield chain homotopies between restriction maps (see \cite{Varolgunes}). Thus:

\begin{prp}\label{lmWeightedWellDefined}
    For $K_1 \subset K_2 \subset M$ compact and acceleration data $A^1, A^2$ for $K_1, K_2$ with $A^2 \leq A^1$:
    \begin{enumerate}
    \item There exists a dissipative monotone homotopy $A^2\to A^1$.
    \item Any two such are connected by a dissipative monotone homotopy.
    \item A choice of monotone homotopy $A^2\to A^1$ induces the \emph{restriction map} $SC^*_M(A^2)\to SC^*_M(A^1)$.
    \item Any two such give chain homotopic maps, hence the same map on homology.
    \item When $K_1=K_2$, the restriction map is an isomorphism on homology.
    \item For any $A^1, A^2$ for $K \subset M$, there exists $A^3$ with $A^3 \leq A^1, A^2$.
\end{enumerate}
Thus different acceleration-data models for $SC^*_M(K)$ are connected by comparison maps and zigzags inducing quasi-isomorphisms, and the induced restriction morphisms are well defined in the homotopy/derived category (equivalently, up to the appropriate quasi-isomorphism and chain homotopy).
\end{prp}

\begin{proof}
    Items (1)-(2): \cite[Appendix A]{AbouzaidGromanVarolgunes24}.
    Items (3)-(4): \cite{Varolgunes}.
    Item (5): \cite[Proposition~6.5]{GromanVarolgunes2021}.
    Item (6): \cite[Lemma 8.13]{Groman}.
\end{proof}

\begin{rem}\label{remNoncompactIndependence}
    In the noncompact case, independence of approximating behavior at infinity follows from \cite[\S8]{Groman}: Theorem~8.9 gives continuity of action-truncated Floer cohomology under monotone approximation, and Lemma~8.13 supplies common lower refinements of approximating sequences. Consequently the resulting truncated/reduced Floer cohomology groups are naturally independent of the approximating behavior at infinity.
\end{rem}

\subsection{Exact symplectic manifolds and relative SH over the ground ring}

For exact $M$, Hamiltonian Floer cochains may be defined over the standing ground integral domain $R$, equipped with the trivial valuation, instead of over $\Lambda_R$. This trivial valuation is the one used for action-norm completions below and in Lemma~\ref{lmRelativeViterboComparison}. Define the \emph{unweighted} complex
\begin{equation}
CF^*_{uw}(H,J;R)=\bigoplus o_{\gamma}\otimes R[i_{cz}(\gamma)]
\end{equation}
with unweighted differential $d:=\sum_u d_u$ and unweighted continuation maps $\kappa=\sum \kappa_u$ for $(H_0,J_0)\leq (H_1,J_1)$.  

Well-definedness in the exact case follows from topological energy depending only on input and output. Fix a primitive $\theta$ of $\omega$ and define the action of a $1$-periodic orbit $\gamma$ of $H$ as
\begin{equation}
\mathcal{A}_H(\gamma)=-\int_{\gamma}\theta-\int_{t\in S^1}H(t,\gamma(t))dt.
\end{equation}
Then $E_{top}(u)=\cA_H(\gamma_0)-\cA_H(\gamma_1)$ for input $\gamma_0$ and output $\gamma_1$. A similar formula holds for continuation maps.

\subsubsection{Relative SH over the ground ring}

For exact $Q$ and an acceleration datum $A$ for $K\subset Q$ (Definition \ref{dfAccDatKsubM}), consider the unweighted $1$-ray  $$\cC_{uw}(A):= CF^*_{uw}(H_1;R)\xrightarrow{\kappa_1} CF^*_{uw}(H_2;R)\xrightarrow{\kappa_2}\ldots. $$ Choosing a primitive $\theta$ of the symplectic form, each complex is normed by $|\gamma|=e^{\cA_{H,\theta}(\gamma)}$, using the trivial valuation on $R$. Since $E_{\mathrm{top}}(u)=\cA_H(\gamma_{\mathrm{in}})-\cA_H(\gamma_{\mathrm{out}})$, the differential and monotone continuation maps are norm non-increasing, inducing a norm on the telescope.
\begin{df}\label{dfUnweightedSHofK}
The unweighted relative cochains of $K$ with respect to $\theta$ are
\begin{equation}
    SC^*_{Q,\theta}(K,A;R):=\widehat{tel}(\mathcal{C}_{uw}(A)),
\end{equation}
the completion of the unweighted telescope with respect to the $\theta$-defined action norm.
\end{df}

Analogously to Proposition \ref{lmWeightedWellDefined}:

\begin{prp}\label{lmUnweightedWellDefined}
    For $K_1 \subset K_2 \subset Q$ compact with acceleration data $A^1, A^2$ for $K_1, K_2$ and $A^2 \leq A^1$:
    \begin{enumerate}
    \item A monotone homotopy $A^2\to A^1$ induces the \emph{restriction map} $SC^*_{Q,\theta}(A^2)\to SC^*_{Q,\theta}(A^1)$.
    \item Any two such are chain homotopic, hence induce the same map on homology.
    \item When $K_1=K_2$, the restriction map is an isomorphism on homology.
    \end{enumerate}
Thus $SC^*_{Q,\theta}(K)$ and the restriction morphism $SC^*_{Q,\theta}(K_2)\to SC^*_{Q,\theta}(K_1)$ are well defined in the homotopy/derived category up to the appropriate quasi-isomorphism and chain homotopy.
\end{prp}

The weighted and unweighted versions relate as follows:

\begin{lm}\label{lmUnweightedSHcomparison}
    For $K\subset Q$ and acceleration datum $A$, there is an isomorphism
    \begin{equation}\label{eqUnweightedSHcomparison}
        SC^*_{Q,\theta}(K,A;R)\widehat{\otimes}_{R}\Lambda_R\to SC^*_Q(K,A;\Lambda_R)
    \end{equation}
    commuting up to homotopy with maps from ordered pairs $A^2\leq A^1$.
\end{lm}

\begin{proof}
    At generator level: $\gamma\otimes T^\lambda\mapsto T^{\lambda-\cA_{H_i,\theta}(\gamma)}\gamma$ for $\gamma$ a $1$-periodic orbit of $H_i$. It is clear this map commutes with continuation maps and so gives rise to a morphism of $1$ rays (with all $h_i=0$). Moreover, all the $f_i$'s are isomorphisms of complexes. With the unweighted norm $|\gamma|=e^{\cA_{H,\theta}(\gamma)}$ and the standard $T$-adic norm on the weighted side, the comparison is an isometry:
    \[
    |\gamma\otimes T^\lambda|
    =
    e^{\cA_{H,\theta}(\gamma)-\lambda}
    =
    \bigl|T^{\lambda-\cA_{H,\theta}(\gamma)}\gamma\bigr|.
    \]
\end{proof}

\subsubsection{Liouville domains and Viterbo $SH$}\label{subsec:viterbo-relative}

For a Liouville domain $(D,\omega=d\theta)$ (compact exact symplectic manifold with smooth boundary, primitive $\theta$), the $\omega$-dual $Z$ of $\theta$ ($\iota_Z\omega=\theta$) points outward at $\partial D$. The negative flow of $Z$ defines a function $\rho$ on a collar of $\partial D$ measuring (negative) flow time from $\partial D$. Attach the symplectization $\partial D\times [0,\infty)$ to get the completion $\widehat{D}=D\cup_{\partial D}\partial D\times [0,\infty)$, a geometrically bounded manifold. The function $\rho$ extends as the projection to $[0,\infty)$.

\emph{Viterbo symplectic cohomology} is defined using linear Hamiltonians and cylindrical contact type almost complex structures. A proper $H:\widehat{D}\to \bR$ is \emph{linear at infinity} if away from a compact set $H=ae^{\rho}+b$ for constants $a,b$ and radial coordinate $\rho$. Recall that the Reeb period spectrum has measure zero. In particular, there are arbitrarily large slopes outside the spectrum \cite{CFHW96}.

\begin{df}\label{dfViterboIndexing}
    The \emph{Viterbo indexing category} $\cH_{Viterbo}$ is defined by:
    \begin{itemize}
    \item Objects: pairs $(H,J)$ where $H$ is linear at infinity with slope not in the period spectrum of $\partial D$, and $J$ is cylindrical and of contact type at infinity
    \item Paths: smooth paths $s\mapsto (H_s,J_s)$ with compactly supported derivative and monotone increasing slope at infinity
    \item Homotopies: smooth $2$-parameter families $(s,\tau)\mapsto(H_{s,\tau},J_{s,\tau})$ interpolating between paths for which $\partial_s H_{s,\tau}$ is uniformly bounded below, uniformly in the auxiliary parameter $\tau$, and for which the constituent continuation data satisfy the same cylindrical/dissipative asymptotic conditions
    \end{itemize}
\end{df}

For $A=(H_\tau,J_\tau,f)$ in $\cH_{Viterbo}$ with $slope(H_{\tau})\to\infty$, construct $tel(\cC_{uw}(A))$ \emph{without completion}, yielding $SC^*_{Viterbo}(D;R)$. Different Viterbo acceleration rays are related by cofinal comparisons inducing quasi-isomorphisms, so $SC^*_{Viterbo}(D;R)$ is well defined up to quasi-isomorphism.

Fix the exact ambient $Q=\widehat D$ and a compact subset $K\subset Q$.
By the universal mapping property of relative symplectic cohomology, any
cofinal sequence of Floer data which is linear at infinity and negative on
$K$ induces a canonical map
\[
\iota_K:
SH^*_{\mathrm{Viterbo}}(D;R)
\longrightarrow
SH^*_{Q,\theta}(K;R).
\]
The induced map is independent of the chosen cofinal sequence. Moreover,
for every inclusion $K_1\subset K_2\subset Q$ one has
\[
\operatorname{res}_{K_2,K_1}\circ \iota_{K_2}
=
\iota_{K_1}.
\]
See~\cite[Equation~(41) and Lemma~4.22]{GromanVarolgunes2022}.

\begin{lm}\label{lmRelativeViterboComparison}
    For $Q=\widehat{D}$, $K=D$, and the Liouville primitive $\theta$, the canonical map
    \[
    \iota_D:
    SH^*_{\mathrm{Viterbo}}(D;R)
    \longrightarrow
    SH^*_{\widehat{D},\theta}(D;R)
    \]
    is an isomorphism after forgetting the norm. Moreover, for a suitable Viterbo-type acceleration datum, $\iota_D$ is represented by an isomorphism from the algebraic telescope to its action-norm completion.
\end{lm}
\begin{proof}
Choose a Viterbo-type acceleration datum $A^{\mathrm{conv}}$ which is simultaneously admissible for the relative theory of $D\subset\widehat{D}$ in the sense of Definition~\ref{dfAccDatKsubM}. Require
\[
H_\tau|_D<0,\qquad H_\tau(x)\to0\quad(x\in D),
\]
with increasing Viterbo slope at infinity, and choose the ray pointwise monotone.

For a radial Hamiltonian $H=h(r)$, the action convention of this paper gives
\[
A_H(\gamma)=rh'(r)-h(r);
\]
see \cite[Proposition~5.3]{GromanVarolgunes2022}. That reference uses the opposite sign for its action/filtration convention, so its nonpositive bounded-action statement becomes a uniform lower bound here. Choose $A^{\mathrm{conv}}$ so that generator actions are uniformly bounded below.

Because $R$ carries the trivial valuation, every nonzero telescope basis component then has norm uniformly bounded away from zero. Hence a norm-convergent sum cannot have infinitely many nonzero telescope components, and the canonical inclusion
\[
\operatorname{tel} C_{\mathrm{uw}}(A^{\mathrm{conv}})
\longrightarrow
\operatorname{tel}^{\,c} C_{\mathrm{uw}}(A^{\mathrm{conv}})
\]
is an isomorphism. This inclusion is a chain representative of the already-defined canonical map $\iota_D$. Therefore $\iota_D$ is an isomorphism on cohomology; cf.\ \cite[Theorem~4.23]{GromanVarolgunes2022}. Another chosen relative chain model computes the same relative symplectic cohomology by Propositions~\ref{lmWeightedWellDefined} and~\ref{lmUnweightedWellDefined}.
\end{proof}

\begin{rem}\label{remGrowthRateSensitivity}
    The map $\iota_K$ exists and is natural for every compact $K\subset\widehat{D}$. It is not generally an isomorphism for arbitrary $K$. For $K=D$ we obtain an isomorphism over the trivially valued ground integral domain $R$ by considering an acceleration datum with all actions bounded from below. For a general relative acceleration datum the action-norm completion over $R$ may contain norm-convergent infinite sums across telescope stages.
\end{rem}

%\subsection{Summary of the various notions of symplectic cohomology}

%\begin{itemize}
 %   \item \textbf{Relative symplectic cohomology} $SH^*_M(K;\Lambda)$: defined over $\Lambda$ for arbitrary compact $K\subset M$, carries a norm, depends on $K$ and $M$.
  %  \item \textbf{Relative SH over the ground ring} $SH^*_{M,\theta}(K;\bK)$: for exact $M$ and primitive $\theta$, defined over $\bK$ but carries a non-Archimedean norm.
  %  \item \textbf{Viterbo's symplectic scohomology} $SH^*_{Viterbo}(D)$: defined over $\bK$ for Liouville domains, depends only on $\widehat{D}$, no norm.
%\end{itemize}

% Relative energy section 
\section{S-shaped Hamiltonians and positivity of relative energy}\label{SecInMaxPrinc}

\subsection{Relative energy for Floer data which is nearly constant outside}
We use the standing conventions of \S\ref{sec:floer-package}. In the
noncompact case, we fix asymptotic data  $J_0,f_0$ witnessing $M$ being of finite type. We consider a compact sub-manifold  $D\subset M$ with smooth boundary, and a smooth $1$-form $\theta$ defined on an open neighborhood  $V(D)$ of $D$. We assume  that  $\omega|_{V(D)}=d\theta$ and that  the $\omega$-dual $Z$ of $\theta$, defined by $\iota_Z\omega=\theta$, points outward at $\partial D$.

$\partial D$ possesses a collar neighborhood $C(D)$ equipped with a proper function $\rho:C(D)\to [-\delta,\delta]$ for some $\delta>0$ such that $\rho(x)$ measures the time it takes to flow from $\partial D$ to $x$ under $Z$. In other words identify $C(D)$ with a subset of  $\partial D\times\bR$ so that $\rho$ is the $\bR$ coordinate we have $Z=\frac{\partial}{\partial{\rho}}$. Note that $\rho$ has no critical points on $C(D)$. We assume $D\cup C(D)$ is the closure of $V(D)$ and in particular  that $\partial V(D)$ is a level set of $\rho$. Thus $V(D)$ is entirely determined by the domain $D$, the primitive $\theta$, and the constant $\delta$. For $t\in [0,\delta]$ we let $D_t:=D\cup\rho^{-1}([0,t])$. 

\begin{df}
A time dependent almost complex structure $J_t$ is called \emph{convex near $\partial D$} if  
\begin{equation}\label{eqConvexity}
d\rho\circ J_t=-e^{f_t}\theta,
\end{equation}
for some time dependent smooth function on $C(D)$. We will take $f_t\equiv 0$ for most of the discussion.
\end{df}

Suppose $H:M\to\bR$ is a smooth function and \emph{has sufficiently small derivatives on $M\setminus V(D)$ so as to have no non-trivial periodic orbits there}. In the non-compact case we assume further that, outside a compact set,
\[
H=\epsilon f_0+C
\]
for $\epsilon>0$ sufficiently small and a spatial constant $C$, where $f_0$ is the fixed admissible exhaustion from \S2. Denote by $\cP(H)$ the set of $1$-periodic orbits of $H$. We define an action functional $\mathcal{A}_H:\mathcal{P}(H)\to\bR$ as follows. For a local orbit $\gamma\subset V(D)$ set
\begin{equation}\label{eq:ActionDef}
\mathcal{A}_H(\gamma)
:=
-\int_{\gamma}\theta
-\int_{0}^{1}H_t(\gamma(t))\,dt.
\end{equation}
For an exterior periodic orbit, which by hypothesis is a constant
critical orbit $p$, set
\begin{equation}
\mathcal{A}_H(p)
:=
-\int_{0}^{1}H_t(p)\,dt.
\end{equation}
Equivalently, if $\lambda$ is any global smooth extension of $\theta$,
both cases are given by $-\int_{\gamma}\lambda-\int H\,dt$, because
the integral of any one-form around a constant orbit vanishes. Note that
unless $\omega$ admits a global primitive the action functional is
\emph{not} defined for arbitrary Hamiltonians.

We consider a pair $(H,J)$ as a function of the cylinder $\bR\times S^1$ where $H$ is as above, and either $H$ is independent of $s$, or $H$ has a partial derivative in the $s$ direction greater than or equal to zero. We will then be considering solutions $u$ to Floer's equation
\begin{equation}\label{eqFloer}
\frac{\partial u}{\partial s} + J(u)\frac{\partial u}{\partial t} = J(u)X_H(u).
\end{equation}

Any solution to Floer's equation \eqref{eqFloer} for such $H$ gives rise to a relative cycle in $H_2(M,V(D);\bZ)$. More precisely we consider the compactifications $\overline{u}$ where punctures mapping to a critical point in $M\setminus V(D)$ are point compactified.
On the other hand, the pair $(\omega,\theta)$ gives rise to a relative cocycle in $H^2(M,V(D);\bR)$. Since $V(D)$ is obtained from $D$ by adjoining the chosen Liouville collar, the inclusion $D\to V(D)$ is a deformation equivalence. We henceforth use the canonical identification
\[
H_2(M,D;\bZ)\cong H_2(M,V(D);\bZ)
\]
(and likewise in relative cohomology) without further comment. For a relative homology class $A$ write
$$E_{M,\theta}(A):=\langle A,[\omega,\theta]\rangle.$$  We slightly abuse notation and write
$$E_{M,\theta}(u):=E_{M,\theta}([\overline{u}]).$$
Note that for a Floer solution $u$ as above 
\begin{equation}
E_{M,\theta}(u)=\int u^*\omega-\int_{\partial \overline{u}}u^*\theta.
\end{equation}

More generally, let $\Sigma$ be a Riemann surface with one output and $n\geq 1$ inputs, equipped with a closed one-form $\alpha$. For a cylindrical end $e$ define its \emph{weight}
\[
w_e:=\int_{S^1_e}\alpha.
\]
Our convention is  $w_e=1$ on each input and $w_0=n$ on the output (so $n=2$ for the pair of pants and $n=1$ for the BV operator). An operation datum is a map $(H,J):\Sigma\to\mathcal{J}\times C^{\infty}(M)$ which is $s$-independent near the ends. Write $H_z$ for the Hamiltonian at $z\in\Sigma$ and call
\[
H_e^{\mathrm{eff}}:=w_e H_e
\]
the \emph{effective Hamiltonian} at the end $e$. We assume 
\begin{equation}
d_{\Sigma}H_z\wedge\alpha\geq 0
\end{equation}
Floer's equation for such a datum is 
\begin{equation}\label{eqOperationFloer}
(du-X_H\otimes\alpha)^{0,1}=0
\end{equation}
Solutions determine relative classes in $H_2(M,V(D);\bZ)$ by compactifying exterior
critical-point ends, and we write $E_{M,\theta}(u):=E_{M,\theta}([\overline{u}])$ as before.

In this paper we use only the following Floer domains. The twice-punctured sphere, identified with the cylinder $\mathbb R\times S^1$, is used for the Floer differential and continuation maps. The pair-of-pants product is defined using the thrice-punctured sphere with two input ends and one output end. Finally, the BV operator is defined by the standard $S^1$-family of Floer data on the twice-punctured sphere, in which the asymptotic marker at the cylindrical end rotates. We also use compact parameter families of these data when constructing chain homotopies. We refer to \cite{AbouzaidGromanVarolgunes24} for the standard construction of the product, BV operator, and their parametrized Floer data.

A compactified exterior component is obtained by cutting at a regular level
$V_s$ in a fixed small collar of $V$, chosen inside the domain of $\theta$,
and adding the limiting point at each exterior constant puncture. We use the
collar isotopy to identify
\[
H_2(M,V_s;\mathbb Z)\simeq H_2(M,V;\mathbb Z)
\]
and likewise in relative cohomology.

\begin{prp}\label{PrpRelHbar}
Fix $V=D_\delta$ and a bounded $C^0$-ball $B$ of compatible
geometrically bounded almost complex structures. Then there are constants
$e_0(B)>0$ and $\epsilon(B)>0$ such that the following holds.

Let $u$ be a Floer solution on the cylinder, or on one of the
fixed punctured surfaces appearing in the product and BV
constructions, with almost complex structure taking values in $B$.
Suppose that for some function $f:\Sigma\to\mathbb R$ depending only
on the domain,
\[
\begin{aligned}
\|X_H\|_{C^0(\Sigma\times(M\setminus V))} &< \epsilon(B),\\
\|d_\Sigma(H-f)\|_{C^0(\Sigma\times(M\setminus V))} &< \epsilon(B),\\
d_\Sigma(H-f)\wedge\alpha &\geq 0
    \qquad\text{on }M\setminus V.
\end{aligned}
\]
Then every compactified connected component $v$ of the exterior portion
with $E_{M,\theta}(v)\neq0$ satisfies
\[
E^{\mathrm{geo}}(v)>e_0(B).
\]
In particular, if $E_{M,\theta}(u)\neq0$, then $E^{\mathrm{geo}}(u)>e_0(B)$.
\end{prp}

We use two auxiliary facts.

\begin{prp}[Graph packing with a boundary buffer]
    \label{prpSmallDiamEst}
    Fix a bounded $C^0$-set $B$ of compatible almost complex structures
    and a constant $q>0$. Restrict attention to the cylinder and to the
    fixed finite collection of punctured domains used for the product and
    BV constructions in this paper, as well as exterior subdomains thereof.
    Given $d_*,b_*>0$, there exist constants $c_B,C_B>0$ and
    \[
    \tau_0=\tau_0(d_*,b_*)>0,
    \]
    depending also on the fixed geometric data, $B$, and $q$, such that
    the following holds.

    Let $N\subset M$ be a codimension-zero region with boundary and let $S$
    be one of the possibly punctured domains just listed. Let
    \[
    u:S\longrightarrow M
    \]
    solve Floer's equation for a possibly domain-dependent datum $(H,J)$ with
    $J\in B$. Assume that
    \[
    u(\partial S)\subset M\setminus\operatorname{int}N,
    \]
    where $\partial S$ denotes the genuine finite boundary of $S$; punctures
    are not part of $\partial S$. Write $\widehat H:=H-f$ for a function
    $f:S\to\mathbb R$ depending only on the domain (allowing $f=0$).
    Let $\gamma\subset S$ be a path compactly contained in $S$, and set
    \[
    d:=d_M\bigl(u(\gamma(0)),u(\gamma(1))\bigr),
    \qquad
    b:=d_M\bigl(u(\gamma),M\setminus\operatorname{int}N\bigr),
    \]
    where target distances are measured in the fixed reference metric of
    Appendix~\ref{SecEstimates}. Suppose that
    \[
    d\geq d_*,
    \qquad
    b\geq b_*,
    \qquad
    0<\tau\leq\tau_0,
    \]
    and that
    \[
    \|X_H\|_{C^0(S\times N)}\le q\tau,
    \qquad
    \|d_S\widehat H\|_{C^0(S\times N)}\le q\tau^2,
    \qquad
    d_S\widehat H\wedge\alpha\ge0
    \quad\text{on }S\times N.
    \]
    Then
    \[
    E^{\mathrm{geo}}(u)
    \geq
    c_B\,d\,\tau
    -
    C_B\tau^2\bigl(\ell(\gamma)+1\bigr),
    \]
    where the geometric energy is taken over the (possibly punctured)
    domain $S$.
    \end{prp}

    \begin{proof}
    Set $\widehat H=H-f$. Since $f$ depends only on the domain,
    $X_{\widehat H}=X_H$.

    Let $\widetilde u(z)=(z,u(z))$ be the graph of $u$. By Lemma~\ref{lmProdCEquivalence},
    under
    \[
    \|X_H\|_{C^0(S\times N)}\le q\tau,
    \qquad
    \|d_S\widehat H\|_{C^0(S\times N)}\le q\tau^2,
    \]
    its graph metric is uniformly equivalent to
    \(\tau^2g_S\oplus g_J\).
    For the reference product metric
    \[
    h_\tau=\tau^2g_S\oplus g_0,
    \]
    the fixed domain geometries and bounded geometry of $g_0$ give curvature
    $O(\tau^{-2})$ and injectivity radius $\gtrsim\tau$. Thus the comparison
    clause of Lemma~\ref{lmMonEst++} applies uniformly on balls of radius $O(\tau)$.
    Hence Lemma~\ref{lmMonEst++} and the fixed reference metric
    give constants $a,c>0$ and $\tau_*>0$, depending only on the prescribed
    data, such that for $0<\tau\le\tau_*$ every graph ball of radius
    $a\tau$ centered on $\widetilde u(S)$ and disjoint from
    $\widetilde u(\partial S)$ has graph energy at least $c\tau^2$.

    Choose
    \[
    \tau_0\le\tau_*
    \]
    so small that, whenever $\tau\le\tau_0$, the assumptions
    $d\ge d_*$ and $b\ge b_*$ ensure the following. After decreasing $a>0$ by a
    uniform factor, metric comparison implies that every such graph ball
    centered on $\widetilde u(\gamma)$ is contained in
    $S\times\operatorname{int}N$. Indeed, any path from its center to
    $S\times(M\setminus\operatorname{int}N)$ has target displacement at least
    $b$, and hence graph length bounded below by a uniform multiple of
    $b$; for $\tau$ sufficiently small relative to $b_*$ this exceeds
    $a\tau$. Since
    \[
    u(\partial S)\subset M\setminus\operatorname{int}N,
    \]
    these balls are also disjoint from $\widetilde u(\partial S)$.
    Metric comparison then allows one to pack
    at least
    \[
    c_1\,\frac d\tau
    \]
    pairwise disjoint such balls along $\widetilde u(\gamma)$. Their total
    graph energy is therefore at least
    \[
    c_2d\tau.
    \]

    The projections of these balls to $S$ lie in a fixed-width
    neighborhood of $\gamma$. Since the domain geometries are fixed, this
    neighborhood has area at most
    \[
    C\bigl(\ell(\gamma)+1\bigr).
    \]
    Using the graph-energy identity of Appendix~\ref{SecEstimates} and
    \(\|d_S\widehat H\|_{C^0(S\times N)}\le q\tau^2\), the same total graph energy is
    at most
    \[
    E^{\mathrm{geo}}(u;S)
    +
    C'\tau^2\bigl(\ell(\gamma)+1\bigr).
    \]
    Thus
    \[
    E^{\mathrm{geo}}(u;S)
    \ge
    c_Bd\tau-C_B\tau^2\bigl(\ell(\gamma)+1\bigr).
    \]
    The argument is local near $\gamma$ and is unchanged on the punctured
    domains considered here.
    \end{proof}

    \begin{lm}[A nontrivial relative class has a wide slice]
    \label{lmSmallSlicesRelative}
    Let $V\subset M$ be a compact codimension-zero submanifold
    with smooth boundary. There exists $\epsilon_*>0$, depending
    only on $V$ and the chosen Riemannian geometry, with the following
    property.

    Let $S$ be a compact oriented surface and let
    \[
    v:(S,\partial S)\longrightarrow (M,V)
    \]
    be continuous. Let $\sigma:S\to\mathbb R$ be a Morse function for
    which there is a constant $L_\sigma<\infty$ such that every connected
    component of every regular level set of $\sigma$ has domain length
    at most $L_\sigma$.

    If $[v]\neq0$ in $H_2(M,V;\mathbb Z)$, then for some regular
    value $a$ of $\sigma$ there is a connected component $C$ of
    \[
    \sigma^{-1}(a)\cap
    v^{-1}\bigl(M\setminus N_{\epsilon_*}(V)\bigr)
    \]
    such that
    \[
    \operatorname{diam}v(C)\geq\epsilon_*.
    \]
    Consequently $C$ contains an arc $\gamma$ satisfying
    \[
    d\bigl(v(\gamma(0)),v(\gamma(1))\bigr)\geq\epsilon_*,
    \qquad
    d\bigl(v(\gamma),V\bigr)\geq\epsilon_*,
    \qquad
    \ell(\gamma)\leq L_\sigma.
    \]
    \end{lm}

    \begin{proof}
    Choose $\epsilon_*>0$ so that $N_{4\epsilon_*}(V)$ deformation retracts
    onto $V$ and $4\epsilon_*$ is below a uniform convexity radius.

    Suppose that every regular level component lying outside
    $N_{\epsilon_*}(V)$ has image diameter $<\epsilon_*$. Choose a compact
    subsurface
    \[
    v^{-1}\bigl(M\setminus N_{2\epsilon_*}(V)\bigr)
    \subset\operatorname{int}(W)
    \subset
    W
    \subset
    v^{-1}\bigl(M\setminus N_{\epsilon_*}(V)\bigr)
    \]
    and collapse the connected level components of $\sigma|_W$ to its
    Reeb graph $G$. The diameter assumption and convexity give a homotopy,
    with boundary in $N_{4\epsilon_*}(V)$, from $v|_W$ to a map factoring
    through $G$. Since the complement of $W$ maps into
    $N_{2\epsilon_*}(V)$, the image of $[v]$ in
    \[
    H_2\bigl(M,N_{4\epsilon_*}(V);\mathbb Z\bigr)
    \]
    factors through
    \[
    H_2\bigl(G,\pi(\partial W);\mathbb Z\bigr)=0.
    \]
    As $N_{4\epsilon_*}(V)$ deformation retracts onto $V$, this contradicts
    $[v]\ne0$.

    Thus a component $C$ as in the statement exists. Two points of $C$
    whose images are separated by at least $\epsilon_*$ determine the
    required subarc $\gamma$, and $\ell(\gamma)\le L_\sigma$.
    \end{proof}

\begin{proof}[Proof of Proposition \ref{PrpRelHbar}]
Fix Morse level functions on the finitely many Floer domains, standard on
their cylindrical ends, whose regular level components have a common length
bound $L_0$. Shrinking the fixed collar family if necessary, choose
$\epsilon_*>0$ from Lemma~\ref{lmSmallSlicesRelative} uniformly for all its levels $V_s$.

Let $v$ be a compactified exterior component with
$E_{M,\theta}(v)\ne0$. If it is cut along $V_s$, then
\[
E_{M,\theta}(v)=\langle[\omega,\theta],[v]\rangle
\]
implies
\[
[v]\ne0\in H_2(M,V_s;\mathbb Z).
\]
Lemma~\ref{lmSmallSlicesRelative} gives an arc $\gamma$ with
\[
d(v(\gamma(0)),v(\gamma(1)))\ge\epsilon_*,
\qquad
d(v(\gamma),V_s)\ge\epsilon_*,
\qquad
\ell(\gamma)\le L_0.
\]
Take
\[
N=M\setminus\operatorname{int}(V_s)
\]
as the controlled region in Proposition~\ref{prpSmallDiamEst}. The finite boundary of the
exterior component maps to $\partial V_s\subset
M\setminus\operatorname{int}N$, while Lemma~\ref{lmSmallSlicesRelative} gives
\[
d_M\bigl(v(\gamma),M\setminus\operatorname{int}N\bigr)
=
d_M\bigl(v(\gamma),V_s\bigr)
\ge\epsilon_*.
\]
Thus Proposition~\ref{prpSmallDiamEst} applies with
\[
d_*=b_*=\epsilon_*.
\]

Choose $0<\tau\le\tau_0(\epsilon_*,\epsilon_*)$ so that
\[
C_B\tau(L_0+1)\le\frac12c_B\epsilon_*,
\]
set
\[
e_0(B)=\frac14c_B\epsilon_*\tau,
\]
and choose
\[
0<\epsilon(B)<\min\{\tau,\tau^2\}.
\]
Then the hypotheses of Proposition~\ref{prpSmallDiamEst} hold and
\[
E^{\mathrm{geo}}(v)
\ge
\frac12c_B\epsilon_*\tau
>
e_0(B).
\]

Finally, Stokes' theorem on the exact collar gives
\[
E_{M,\theta}(u)=\sum_vE_{M,\theta}(v).
\]
Thus $E_{M,\theta}(u)\ne0$ gives an exterior component of nonzero
relative energy, and hence
\[
E^{\mathrm{geo}}(u)\ge E^{\mathrm{geo}}(v)>e_0(B).
\]
\end{proof}

We call a bounded $C^0$-ball $B$ as in Proposition~\ref{PrpRelHbar}
\emph{$e_0$-admissible.}

\subsection{S-shaped data}
\begin{df}\label{dfSShaped}
Fix $\delta>0$ small enough that the embedding of $D$ in $M$ extends to a symplectic embedding of $D_{\delta}$. Fix a bounded $C^0$-ball $B\subset\mathcal{J}$ of compatible geometrically bounded almost complex structures, and let $e_0(B)>0$ and $\epsilon(B)>0$ be the constants of Proposition \ref{PrpRelHbar} applied with $V=D_{\delta}$. Fix parameters
\[
\sigma>0,\qquad \eta_{\mathrm{end}}>0
\]
small enough that $\sigma$ is less than the smallest Reeb period of $\theta$ on $\partial D$ and
\[
\sigma+\eta_{\mathrm{end}}<e_0(B).
\]

\paragraph{S-shaped Floer objects.}
An object of the indexing category $\cH(D,\theta,\delta,\sigma,\eta_{\mathrm{end}})$ is a pair $(H,J)$ for which there exists an inner level $\delta'\in(\delta/2,\delta)$ such that, writing
\[
A:=\rho^{-1}([\delta',\delta])=D_{\delta}\setminus\operatorname{int}(D_{\delta'})
\]
for the associated \emph{cutting annulus}, the following hold.
\begin{enumerate}
\item On $A$ we have $H=h(e^{\rho})$ for a smooth function $h:(e^{\delta'},e^{\delta})\to\bR$ with
\begin{equation}\label{eqhPrimeDelta}
e^{\rho}h'(e^{\rho})=\sigma.
\end{equation}
\item $J$ is of contact type on $A$, i.e.\ \eqref{eqConvexity} holds throughout $A$.
\item Outside $D_{\delta}$ there are no nonconstant $1$-periodic orbits of $H$; every exterior orbit is a critical point.
\item  On $M\setminus D_{\delta}$ the Hamiltonian $H$ is autonomous and
\[
\|X_{H}\|_{C^{0}(M\setminus D_{\delta})}<\epsilon(B),
\]
while $J$ takes values in $B$.
\item For every exterior critical point $p$,
\begin{equation}\label{eqEndOscillation}
H|_{\rho=\delta}
\leq H(p)
\leq H|_{\rho=\delta'}+\eta_{\mathrm{end}}.
\end{equation}
\end{enumerate}
In the noncompact case, $J$ coincides with $J_0$ outside a compact set and $H$ coincides outside a compact set with $\epsilon f_0+\mathrm{Const}$ for an admissible exhaustion $f_0$ and $\epsilon$ small enough for dissipativity.

\paragraph{Admissible S-shaped continuation data.}
An admissible continuation path is a path $(H_s,J_s)$, constant for $|s|$ sufficiently large. Its endpoints are objects of $\cH(D,\theta,\delta,\sigma,\eta_{\mathrm{end}})$, and which satisfies
\begin{equation}\label{eqContOrdinaryMono}
\partial_sH_s\geq 0.
\end{equation}
 The path is required to satisfy:
\begin{enumerate}
\item There is one common $\delta'\in(\delta/2,\delta)$ for the whole path.
\item On $A=D_{\delta}\setminus\operatorname{int}(D_{\delta'})$ every $H_s$ is of the form
\[
H_s=h_s(e^{\rho}),\qquad e^{\rho}h_s'(e^{\rho})=\sigma,
\]
and $J_s$ is of contact type on $A$ and takes values in $B$. Writing
\[
c(s):=H_s|_{\rho=\delta},\qquad H_s^e:=H_s-c(s),
\]
each exterior slice $H_s|_{M\setminus D_\delta}$ is autonomous in Floer time, and one has on $M\setminus D_\delta$
\begin{equation}\label{eqContExteriorSmall}
\|X_{H_s}\|_{C^0(M\setminus D_\delta)}<\epsilon(B),
\qquad
\|\partial_sH_s^e\|_{C^0(M\setminus D_\delta)}<\epsilon(B),
\end{equation}
and
\begin{equation}\label{eqContMonotone}
\partial_sH_s^e\geq0
\qquad\text{on }M\setminus D_{\delta'}.
\end{equation}
\item In the noncompact case the path retains the dissipativity conditions of the object definition, and the two asymptotic Hamiltonians and the continuation family use one common sufficiently small coefficient of the fixed exhaustion $f_0$ at infinity; only the additive constant is allowed to vary there.
\end{enumerate}
An \emph{admissible homotopy} is a piecewise-smooth finite-parameter family of admissible continuation paths for which the above data and estimates may be chosen uniformly, including one common witnessing $\delta'$.

\paragraph{Admissible S-shaped operation data.}
Let $\Sigma$ be a Riemann surface with one output and $n\geq 1$ inputs, equipped with a closed one-form $\alpha$ and with a datum  $(H,J):\Sigma\to\mathcal{J}\times C^{\infty}(M)$ which is $s$-independent near the ends. Writing
\[
c_z:=H_z|_{\rho=\delta},\qquad H_z^e:=H_z-c_z,
\]
call $H_z^e$ the \emph{normalized base exterior Hamiltonian}, and assume
\begin{equation}\label{eqOpMonotone}
d_{\Sigma}H_z^{e}\wedge\alpha\geq 0.
\end{equation}
We refer to this as the \emph{nonnegative exterior curvature condition}. The datum is required to satisfy in addition:
\begin{itemize}
\item there is one common $\delta'\in(\delta/2,\delta)$;
\item the base Hamiltonian has radial slope $\sigma/n$ on the common annulus $A=\rho^{-1}([\delta',\delta])$, and $J$ is of contact type on $A$ and takes values in $B$;
\item the effective Hamiltonian at each input is an object of $\cH(D,\theta,\delta,\sigma/n,\eta_{\mathrm{end}}/n)$, and the effective output Hamiltonian is an object of $\cH(D,\theta,\delta,\sigma,\eta_{\mathrm{end}})$;
\item on the exterior one has \eqref{eqOpMonotone} on $M\setminus D_{\delta'}$, together with
\[
\|X_{H_z}\|_{C^{0}(M\setminus D_{\delta})}<\epsilon(B)/n,
\qquad
\|d_{\Sigma}H_z^{e}\|_{C^{0}(M\setminus D_{\delta})}<\epsilon(B)/n.
\]
\item In the noncompact case, outside one common compact set one has
$J_z=J_0$, and the base Hamiltonian family uses one common sufficiently
small coefficient of $f_0$ at infinity.
\end{itemize}
For a compact parameter family of operation data, admissibility means that the above choices may be made uniformly, including one common $\delta'$ and one common admissible ball $B$. The main example is the $S^1$-family defining the BV operator.
\end{df}

The conditions in Definition \ref{dfSShaped} have different roles, designed to achieve the positivity properties in \ref{lmPositivity} and the propositions following it. Let us briefly comment on the role of the various conditions in Definition \ref{dfSShaped}. The contact-type condition and the  barrier equation \eqref{eqhPrimeDelta} are the local inputs to the integrated maximum principle employed in  Lemma \ref{lmPositivity}. The exterior
smallness condition places the data in the geometric-energy
quantization regime of Proposition \ref{PrpRelHbar}. The
oscillation condition \eqref{eqEndOscillation} controls the boundary terms arising from exterior asymptotic orbits. For continuation and operation data, \eqref{eqContMonotone} and \eqref{eqOpMonotone} ensure that the additional domain-dependent terms in the energy identity are nonnegative.

The barrier condition \eqref{eqhPrimeDelta} is the only non-open condition in this
list. In particular, barrier conditions with different slopes cannot
in general be imposed on the same annulus. This is the reason that the
inner cutting level $\delta'$ is allowed to vary. By taking
$\delta'$ arbitrarily close to $\delta$, the exact barrier condition
may be confined to an arbitrarily thin shell. Finitely many such
conditions can therefore be placed on disjoint shells, while the
remaining smallness and end-oscillation inequalities are satisfied on
the regions between them.

\begin{rem}\label{remMotivationSShaped}
    Strictly constant exterior Hamiltonians would have considerably simplified the energy
    argument. But they produce a highly degenerate family of exterior
    $1$-periodic orbits and are unsuitable for the compactness/Floer setup
    used here. We therefore use dynamically small exterior perturbations
    controlled by Proposition \ref{PrpRelHbar} and Definition
    \ref{dfSShaped}.
    \end{rem}

\subsection{S-shaped acceleration data}\label{subsecDataDep}
We refer to
\[
\mathcal P=(B,\delta,\sigma,\eta_{\mathrm{end}})
\]
as an S-shaped package. In the noncompact case, a package also includes
a sufficiently small asymptotic coefficient
$\epsilon_\infty(\mathcal P)>0$ for the fixed exhaustion $f_0$.
Fix an S-shaped package $\mathcal P$ and a compact set $K\subset D$.
Write
\[
\cH_K
=
\cH_K(D,\theta,\delta,\sigma,\eta_{\mathrm{end}})
\subset
\cH(D,\theta,\delta,\sigma,\eta_{\mathrm{end}})
\]
for the full subcategory of objects satisfying $H|_K<0$.
In the geometrically finite noncompact case we further restrict $\cH_K$ to
objects whose coefficient of $f_0$ at infinity equals
$\epsilon_\infty(\mathcal P)$; this restriction is suppressed from the
notation.

\begin{df}\label{dfSShapedAcceleration}
A \emph{regular S-shaped acceleration datum} for $K$ and $\mathcal P$
is a regular cofinal acceleration datum in the indexing category
$\cH_K$, in the sense of Definition~\ref{dfAccelerationDatum}, with the
following additional uniform requirements:
\begin{enumerate}
\item one common witnessing cutting level $\delta'_A\in(\delta/2,\delta)$
works for the whole acceleration ray;
\item the normalized exterior Hamiltonian
\begin{equation}\label{eqStrongExteriorProfile}
H_\tau^e
:=
H_\tau-H_\tau|_{\rho=\delta}
\end{equation}
is independent of $\tau$ on $M\setminus D_\delta$, say
$H_\tau^e=q_A$.
\end{enumerate}
\end{df}

\begin{df}\label{dfSShapedMorphism}
Let $A^-$ and $A^+$ be regular S-shaped acceleration data for the same
S-shaped package $\mathcal P$ and the same $(D,\theta,K)$.
An \emph{S-shaped morphism}
\[
A^- \longrightarrow A^+
\]
is a morphism of the associated $1$-rays, after a cofinal reindexing
$m:\mathbb N\to\mathbb N$, in the sense of the ray-morphism diagram
\eqref{eqRayMorphism}, whose vertical continuation paths and coherence
homotopies are admissible S-shaped.
\end{df}

\begin{prp}\label{prpSShapedAcceleration}
\label{lmSAccelExistence}
\label{lmStronglyAdmissibleComparison}
\label{lmSShapedPackageBridge}
For every compact $K\subset D$ and every S-shaped package
$\mathcal P$, there exist regular S-shaped acceleration data.

Moreover, the regular S-shaped acceleration data for
$(D,\theta,K,\mathcal P)$ are downward directed with respect to
S-shaped morphisms. Namely, given two such data $A_1$ and $A_2$,
there exists a regular S-shaped acceleration datum $A_0$ and
S-shaped morphisms
\[
A_1 \longleftarrow A_0 \longrightarrow A_2 .
\]
\end{prp}

\begin{proof}
We first construct one regular S-shaped acceleration datum.

Choose a 
\[
\delta'\in(\delta/2,\delta)
\]
and a sufficiently small autonomous normalized exterior profile
\[
q:M\setminus D_\delta\longrightarrow\mathbb R
\]
which vanishes on $\rho=\delta$ and extends to a smooth function on $M\setminus D_{\delta'}$ satisfying all the conditions of Definition~\ref{dfSShaped} for an object when on the complement of $D_{\delta'}$. We require equation \eqref{eqhPrimeDelta} to hold on an open neighborhood of $D_{\delta}\setminus D_{\delta'}$. In the
noncompact case choose $q$ so that, in addition, outside a compact set,
\[
q=\epsilon_\infty(\mathcal P)f_0+C
\]
for some constant $C$.

Keep this normalized exterior profile and the cutting profile fixed
along the acceleration datum. Make the Hamiltonian cofinal by increasing
it on the interior of $D_{\delta'}$, where no restrictions are imposed, and by adding spatially constant shifts
on the exterior. Choose the almost complex structures in the fixed
collar-contact-type class associated to $\mathcal P$. After generic
choices of the almost complex structures, the integer levels and the
integer-to-integer continuation maps are regular. This gives a regular
S-shaped acceleration datum.

We now prove downward directedness. Let $A_1$ and $A_2$ be regular
S-shaped acceleration data for the same package $\mathcal P$, and let
$q_1,q_2$ denote their normalized exterior profiles. Let
\[
q_0\leq \min\{q_1,q_2\}
\]
be an exterior profile with the properties listed in the previous paragraph. It is straightforward to obtain this by smoothing the minimum since $q_1,q_2$ agree on an open neighborhood of the boundary of $M\setminus D_{\delta}$.

Construct a regular S-shaped acceleration datum $A_0$ with normalized
exterior profile $q_0$ as above. Since $A_1$ and $A_2$ are cofinal,
for each $i=1,2$ we may choose a strictly increasing cofinal map
\[
m_i:\mathbb N\longrightarrow\mathbb N
\]
such that
\[
H^0_n\leq H^i_{m_i(n)}
\qquad\text{for all }n.
\]

Fix $i\in\{1,2\}$. For each $n$, choose a monotone interpolation from
$H^0_n$ to $H^i_{m_i(n)}$. On the controlled exterior, after subtracting
the value on $\rho=\delta$, this interpolation joins $q_0$ to $q_i$.
Since $q_0\leq q_i$, it may be chosen so that
\[
\partial_s\bigl(H_s-H_s|_{\rho=\delta}\bigr)\geq0
\]
there. Thus the interpolation is an admissible S-shaped continuation.
Stretching the $s$-parameter gives the required exterior derivative
bound. In the noncompact case the common coefficient
$\epsilon_\infty(\mathcal P)$ keeps the difference of the normalized
exterior Hamiltonians bounded at infinity, so the same argument applies.
Choose the almost complex structures generically inside the prescribed
collar-contact-type class to obtain regular continuation maps.

For the coherence homotopies of Definition~\ref{dfSShapedMorphism}, join
the two continuation paths from $A^0_n$ to $A^i_{m_i(n+1)}$ by a convex
two-parameter interpolation after reparametrizing concatenations.
Convex interpolation preserves ordinary and normalized monotonicity;
common stretching of the continuation parameter gives the exterior
derivative bound; and generic perturbations away from the protected
collars give regularity (in the virtual setting, the corresponding
admissible homotopies are supplied by the fixed virtual Floer package).

We have therefore constructed S-shaped morphisms
\[
A_0\longrightarrow A_i,
\qquad i=1,2,
\]
which proves the claim.
\end{proof}

\subsection{Simultaneous S-shaped data}
The purpose of this subsection is twofold. One is the independence of our constructions on the choice of package. The second and  more important purpose is to allow changing the primitive or the domain. This will be applied when proving Theorem \ref{mainTmB}.

We say that an acceleration datum is \emph{simultaneously S-shaped}
for a finite collection of choices of domain, primitive, and package if
it is a regular S-shaped acceleration datum in the sense of
Definition~\ref{dfSShapedAcceleration} for each of them. A morphism is
\emph{simultaneously S-shaped} if its continuation paths and coherence
homotopies are S-shaped for every choice in the collection.

\begin{lm}\label{lmSShapedFlexibility}
Let $K$ be compact.

\begin{enumerate}

\item
Let
\[
K\subset D^1\subset\cdots\subset D^m\subset M
\]
be a finite chain of Liouville domains, with Liouville forms
$\theta_1,\ldots,\theta_m$. For
$\sigma,\eta_{\mathrm{end}}>0$ sufficiently small, there are
$\delta_1,\ldots,\delta_m>0$ and a regular acceleration datum for $K$
which is simultaneously S-shaped for
\[
(D^i,\theta_i,\delta_i,\sigma,\eta_{\mathrm{end}}),
\qquad i=1,\ldots,m.
\]
The same conclusion holds for finitely many Liouville primitives on one
fixed domain.

Any two such simultaneous acceleration data admit a common regular simultaneous refinement with S-shaped morphisms to both.

\item
Let $\mathcal P_1$ and $\mathcal P_2$ be S-shaped packages for the same
$(D,\theta,K)$. There is a bridge package $\mathcal P_*$ and, for
$i=1,2$, regular acceleration data simultaneously S-shaped for
$\mathcal P_i$ and $\mathcal P_*$. 

\end{enumerate}

In the noncompact case, all packages occurring simultaneously are
required to use one common sufficiently small asymptotic coefficient.
\end{lm}

\begin{proof}
For a finite collection of nested domains, place the finitely many exact
barrier equations on pairwise disjoint arbitrarily thin shells,
proceeding from the innermost domain outward. Make all intervening
Hamiltonian rises sufficiently small; choose one almost-complex-structure
family contact type on all protected shells; put the cofinal growth only
in the unrestricted interior and use spatially constant shifts on the
controlled exterior. Keep these collar and exterior profiles fixed along
the acceleration ray. In the noncompact case use one common asymptotic
coefficient throughout.

On each compact transition region between protected cutting shells, choose a
$C^2$-small autonomous interpolation matching the prescribed data near its
boundary. 
After a $C^2$-small perturbation supported away from the protected shells, it may
be taken Morse, and, by taking the interpolation sufficiently $C^2$-small, its
one-periodic orbits are precisely its constant critical orbits. The intervening
Hamiltonian rises are chosen sufficiently small that all such critical values lie
in the end-value windows required by the relevant S-shaped packages. Thus the
transition regions satisfy the exterior smallness and end-value conditions and
introduce no nonconstant one-periodic orbits.

For finitely many primitives $\theta_j$ on one domain $D$, let $Z_j$ be
the corresponding Liouville vector fields. Since every $Z_j$ is outward
transverse to $\partial D$, there is a common neighborhood of
$\partial D$ on which they are all outward transverse. Choose pairwise
disjoint thin ambient bands in this neighborhood. In the $j$-th band
choose a Liouville collar adapted to the $Z_j$-flow, with coordinate
$\rho_j$, and impose
\[
H=h_j(e^{\rho_j}),
\qquad
e^{\rho_j}h_j'(e^{\rho_j})=\sigma.
\]
The exact collar equations are therefore imposed on disjoint shells.  Since all of
these shells lie in one fixed collar neighborhood of the same boundary
$\partial D$, the regions between them may be taken to be product bands. In this case the regions between protected shells can be taken to contain no generators.

Keeping the finitely many protected shells fixed, the
downward-directedness construction of
Proposition~\ref{prpSShapedAcceleration} applies shellwise: on each
controlled region choose a common lower normalized profile as there, and
choose the continuation interpolations and coherence homotopies
simultaneously on the finitely many shells. Common stretching supplies
the derivative bounds, and regular perturbations remain supported away
from the protected shells. Thus any two regular data simultaneously
S-shaped for the same finite list of packages admit a common lower
regular datum and morphisms to both whose continuation and coherence
data remain S-shaped for every package in the list.

For different packages
$\mathcal P_i=(B_i,\delta_i,\sigma_i,\eta_i)$ having the same
asymptotic coefficient in the noncompact case, choose
\[
\mathcal P_*=(B_*,\delta_*,\sigma_*,\eta_*)
\]
with $\delta_*>\max\{\delta_1,\delta_2\}$, with $B_*$ containing an
appropriate collar-adapted path between reference structures in $B_1$
and $B_2$, and with $\sigma_*+\eta_*$ sufficiently small relative to
$e_0(B_*)$; in the noncompact case set
\[
\epsilon_\infty(\mathcal P_*)
=
\epsilon_\infty(\mathcal P_1)
=
\epsilon_\infty(\mathcal P_2).
\]
Using the same disjoint-thin-shell construction, obtain for each $i$
cofinal data simultaneously S-shaped for $\mathcal P_i$ and
$\mathcal P_*$ (the $\mathcal P_i$-shell in the unrestricted interior of
the bridge package, the $\mathcal P_*$-shell in the controlled exterior
of $\mathcal P_i$). Proposition~\ref{prpSShapedAcceleration}, applied
within the fixed packages $\mathcal P_1$, $\mathcal P_*$, and
$\mathcal P_2$, then gives the finite zigzag between acceleration data
associated to $\mathcal P_1$ and $\mathcal P_2$.
\end{proof}

\subsection{Positivity of relative energy}\label{secPositivity}
Fix an S-shaped package 
\[
\mathcal P=(B,\delta,\sigma,\eta_{\mathrm{end}})
\]

We choose a filtration-gap constant $\hbar$ satisfying
\begin{equation}\label{eqHbarGap}
0<\hbar<e_0(B)-(\eta_{\mathrm{end}}+\sigma).
\end{equation}

Shrinking $\hbar$ further if necessary, we also assume
\[
3\hbar<e_0(B).
\]

We now prove the relative-energy trichotomy.

\begin{lm}\label{lmPositivity}
Let $(H,J)$ be an S-shaped Floer object and let $\hbar$ satisfy
\eqref{eqHbarGap}. Then every solution $u$ of the translation-invariant
Floer equation \eqref{eqTransInvFloer} satisfies $E_{M,\theta}(u)\geq 0$.
Moreover one of the following holds:
\begin{enumerate}
\item[(i)] $E_{M,\theta}(u)>\hbar$;
\item[(ii)] $u$ is contained in $V(D)$;
\item[(iii)] $E_{M,\theta}(u)=0$ and the output is an exterior critical point.
\end{enumerate}
\end{lm}
% Figure 5: Trichotomy of Floer Solutions
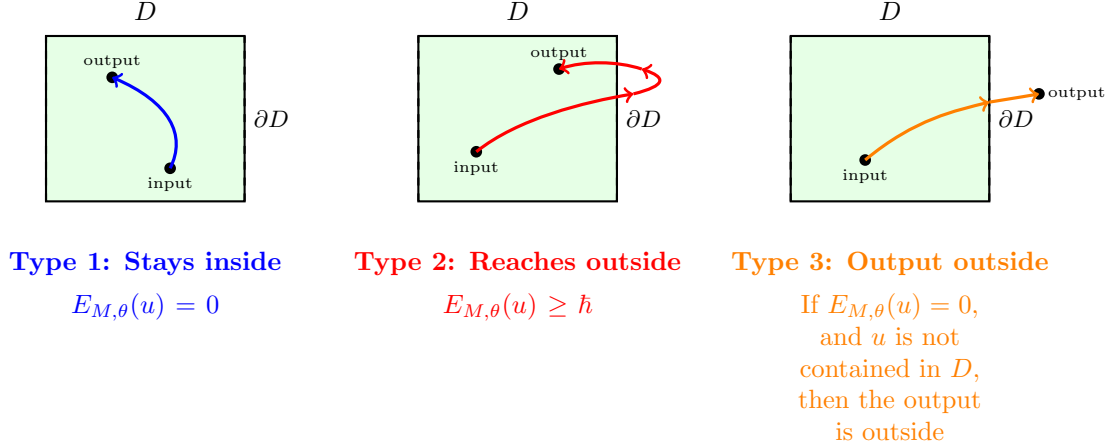
\begin{figure}[h]
\centering
\resizebox{\linewidth}{!}{%
\begin{tikzpicture}[scale=1.1]
    % Domain D (repeated 3 times)
    \foreach \x in {0,4.5,9} {
        \draw[thick, fill=green!10] (\x-1.2,-1) rectangle (\x+1.2,1);
        \draw[thick, dashed] (\x-1.2,-1) -- (\x-1.2,1);
        \draw[thick, dashed] (\x+1.2,-1) -- (\x+1.2,1);
        \node at (\x,1.3) {$D$};
        % Mark boundary
        \node[right] at (\x+1.2,0) {\small $\partial D$};
    }
    
    % Type 1: Stays inside (input and output inside, never leaves)
    \begin{scope}[shift={(0,0)}]
        \fill (0.3,-0.6) circle (2pt) node[below] {\tiny input};
        \fill (-0.4,0.5) circle (2pt) node[above] {\tiny output};
        \draw[->, very thick, blue] (0.3,-0.6) .. controls (0.5,-0.2) and (0.3,0.2) .. (-0.4,0.5);
        \node[blue, below] at (0,-1.5) {\textbf{Type 1: Stays inside}};
        \node[blue, below, text width=2.5cm, align=center] at (0,-2) {$E_{M,\theta}(u) = 0$};
    \end{scope}
    
    % Type 2: Reaches outside (input and output inside, but escapes temporarily)
    \begin{scope}[shift={(4.5,0)}]
        \fill (-0.5,-0.4) circle (2pt) node[below] {\tiny input};
        \fill (0.5,0.6) circle (2pt) node[above] {\tiny output};
        \draw[->, very thick, red] (-0.5,-0.4) .. controls (0,0) and (0.8,0.2) .. (1.4,0.3);
        \draw[->, very thick, red] (1.4,0.3) .. controls (1.8,0.4) and (1.8,0.5) .. (1.5,0.6);
        \draw[->, very thick, red] (1.5,0.6) .. controls (1.2,0.7) and (0.8,0.7) .. (0.5,0.6);
        \node[red, below] at (0,-1.5) {\textbf{Type 2: Reaches outside}};
        \node[red, below, text width=2.5cm, align=center] at (0,-2) {$E_{M,\theta}(u) \geq\hbar $};
    \end{scope}
    
    % Type 3: Output outside (input may be inside or outside)
    \begin{scope}[shift={(9,0)}]
        \fill (-0.3,-0.5) circle (2pt) node[below] {\tiny input};
        \fill (1.8,0.3) circle (2pt) node[right] {\tiny output};
        \draw[->, very thick, orange] (-0.3,-0.5) .. controls (0.3,0) and (0.8,0.1) .. (1.2,0.2);
        \draw[->, very thick, orange] (1.2,0.2) .. controls (1.5,0.25) and (1.7,0.28) .. (1.8,0.3);
        \node[orange, below] at (0,-1.5) {\textbf{Type 3: Output outside}};
        \node[orange, below, text width=2.5cm, align=center] at (0,-2) {If $E_{M,\theta}(u)=0$, and $u$ is not contained in $V(D)$, then the output is outside};
    \end{scope}
    
    % Title
    \node[above, text width=10cm, align=center] at (4.5,2) {
        Trichotomy of Floer trajectories for S-shaped Hamiltonians
    };
    
\end{tikzpicture}%
}
\caption{The trichotomy of Floer solutions for S-shaped Hamiltonians.}
\label{fig:trichotomy}
\end{figure}
\begin{proof}
Choose a witnessing $\delta'$ as in Definition \ref{dfSShaped} and a
regular $\epsilon\in[\delta',\delta]$. Let $S_{\epsilon}$ be the portion
of the domain mapped to the closed exterior of $\rho=\epsilon$, write
$u_{\epsilon}:=u|_{S_{\epsilon}}$, and let $\nu_-,\nu_+\in\{0,1\}$
indicate the presence of an exterior input and output end respectively.
Set
\[
g_{\pm}:=H(p_{\pm})-H|_{\rho=\epsilon}
\]
whenever the corresponding exterior end exists, and $g_{\pm}:=0$
otherwise. We allow $S_{\epsilon}=\varnothing$; in that case
$u_{\epsilon}$ is the empty restriction and all integrals and
geometric-energy terms over $S_{\epsilon}$ are understood to be zero.

The portion of the Floer cylinder removed from $S_{\epsilon}$ lies in
$V(D)$, so it does not change the relative class. Thus
\[
E_{M,\theta}(u)=E_{M,\theta}(u_{\epsilon}).
\]
We prove the estimate
\begin{equation}\label{eqEMEst}
E_{M,\theta}(u)
\geq
E^{\mathrm{geo}}\bigl(u|_{S_{\epsilon}}\bigr)
+(g_-+\sigma)\nu_-
-(g_++\sigma)\nu_+.
\end{equation}
When $S_{\epsilon}=\varnothing$ this is immediate. When $S_{\epsilon}$ is nonempty, let
$\overline{u}_{\epsilon}$ be obtained by one-point compactification at
any exterior input or output end. Write $\alpha$ for the restriction of
$dt$ to $S_{\epsilon}$, and orient $\partial S_{\epsilon}$ by the
boundary orientation.

The geometric-energy identity on the cut domain reads
\begin{equation}\label{eqEgeoueps}
E^{\mathrm{geo}}(u_{\epsilon})=\int u_{\epsilon}^*\omega+\int_{\partial u_{\epsilon}}H\alpha,
\end{equation}
and therefore
\begin{align}\label{eqEgeoPositivity}
E_{M,\theta}(u)&=E_{M,\theta}(u_{\epsilon})
=\int u_{\epsilon}^*\omega-\int_{\partial\overline{u}_{\epsilon}}\theta
=E^{\mathrm{geo}}(u_{\epsilon})-\int_{\partial u_{\epsilon}}H\alpha-\int_{\partial\overline{u}_{\epsilon}}\theta.
\end{align}

On a neighborhood of each collar component of
$\partial\overline{u}_{\epsilon}$ the contact-type equation and the
barrier profile give
\begin{equation}
u^*\theta=-e^{-f}d(\rho\circ u)\circ j-\sigma\alpha.
\end{equation}
Since $\rho\circ u$ achieves a local minimum along the cut and $j$ maps a
positively oriented boundary tangent to an inward-pointing vector,
\begin{equation}\label{eqLNeg}
\int_{\mathrm{cut}}\theta\leq -\sigma\int_{\mathrm{cut}}\alpha.
\end{equation}
The contributions of exterior ends to $-\int H\alpha$ are
computed as follows. An exterior input end ($\nu_-=1$), oriented as an
incoming cylindrical end, contributes $H(p_-)$ to $-\int H\alpha$ after
accounting for $\int\alpha=-1$ on that end relative to the cut-domain
boundary orientation, while the constant-disc compactification
contributes nothing to the relative $\theta$-integral; combined with the
cut integral of $\sigma\alpha$ this yields the term $+(g_-+\sigma)$. An
exterior output end ($\nu_+=1$) likewise yields $-(g_++\sigma)$. (If
both ends are local then $\int_{\partial u_{\epsilon}}\alpha=0$ on the
exterior cut and the Hamiltonian boundary term cancels against the
corresponding part of the barrier contribution.) Altogether one obtains
\eqref{eqEMEst}.

If $E_{M,\theta}(u)\neq0$, additivity over the compactified exterior
components gives one component $v$ with nonzero relative energy.
Proposition \ref{PrpRelHbar} supplies $E^{\mathrm{geo}}(v)>e_0(B)$.
Since $E^{\mathrm{geo}}(u_{\epsilon})\geq E^{\mathrm{geo}}(v)$ and
\eqref{eqEndOscillation} gives $0\leq g_{\pm}\leq\eta_{\mathrm{end}}$,
the worst boundary correction in \eqref{eqEMEst} is
$-(\eta_{\mathrm{end}}+\sigma)$. Hence
\[
E_{M,\theta}(u)
\geq e_0(B)-(\eta_{\mathrm{end}}+\sigma)
>\hbar,
\]
so alternative (i) holds.

Suppose now that $E_{M,\theta}(u)=0$. If $u$ is contained in $V(D)$,
alternative (ii) holds. It remains to treat the case
$E_{M,\theta}(u)=0$ with $u\not\subset V(D)$. We claim $\nu_+=1$.
Indeed, if $\nu_+=0$ then the right-hand side of \eqref{eqEMEst} is
strictly positive: if $\nu_-=1$ then $g_-+\sigma\geq\sigma>0$, while if
$\nu_-=0$ both ends are local and, since $u$ is not contained in $V(D)$,
the exterior portion $u_{\epsilon}$ is nontrivial, so
$E^{\mathrm{geo}}(u_{\epsilon})>0$. Thus $\nu_+=0$ contradicts
$E_{M,\theta}(u)=0$. Hence $\nu_+=1$, the output is an exterior
critical point, and alternative (iii) holds.

In all cases $E_{M,\theta}(u)\geq0$.
\end{proof}

\begin{lm}\label{lmPositivity2}
Let $(H_s,J_s)$ be an admissible S-shaped continuation path and let
$\hbar$ satisfy \eqref{eqHbarGap}. Then every continuation solution
satisfies the trichotomy of Lemma \ref{lmPositivity}.
\end{lm}

\begin{proof}
Choose the common witnessing $\delta'$ and a regular
$\epsilon\in[\delta',\delta]$. With the notation of the proof of Lemma
\ref{lmPositivity}, and writing $H_s^e=H_s-c(s)$, we prove the stronger
estimate
\begin{equation}\label{eqEMineq2}
E_{M,\theta}(u)
\geq
E^{\mathrm{geo}}\bigl(u|_{S_\epsilon}\bigr)
+(g_-+\sigma)\nu_-
-(g_++\sigma)\nu_+,
\end{equation}
where
\[
g_-:=H_-(p_-)-H_-|_{\rho=\epsilon},
\qquad
g_+:=H_+(p_+)-H_+|_{\rho=\epsilon}
\]
when the corresponding exterior end is present, and $g_\pm=0$
otherwise. Subtracting $c(s)$ does not change the Floer equation.
Repeating the cut-domain calculation of Lemma \ref{lmPositivity} and
using
\begin{equation}\label{eqContEnergyIdentity}
E_{\mathrm{top}}(u_\epsilon)
=
E^{\mathrm{geo}}(u_\epsilon)
+
\int_{S_\epsilon}(\partial_sH_s^e\circ u)\,ds\wedge dt,
\end{equation}
with $\partial_sH_s^e\geq0$ by admissibility, one obtains
\eqref{eqEMineq2}. By \eqref{eqEndOscillation} one has
$0\leq g_\pm\leq\eta_{\mathrm{end}}$. Thus \eqref{eqEMineq2} is the same
lower bound as \eqref{eqEMEst}. The remainder of the proof of Lemma
\ref{lmPositivity}, including the nonzero-energy $\hbar$-gap and the
zero-energy exterior-output alternative, applies verbatim.
\end{proof}

\subsection{S-shaped Floer data and positivity for operations}
\label{subsecSShapedOperations}
We use the operation terminology of Definition \ref{dfSShaped}. 

For the BV operator the input and output use the same S-shaped package. For the pair-of-pants product, choose an S-shaped package
\[
\mathcal P=(B,\delta,\sigma,\eta_{\mathrm{end}})
\]
and a normalized exterior profile $q$ sufficiently small that the doubled package
\[
2\mathcal P=(B,\delta,2\sigma,2\eta_{\mathrm{end}})
\]
is admissible. In the noncompact case we likewise require
\[
\epsilon_\infty(2\mathcal P)
=
2\epsilon_\infty(\mathcal P)
\]
to remain sufficiently small. We make the convenient choice that the normalized base exterior Hamiltonian of the pair-of-pants datum is the fixed profile $q$. Since the input weights are $1,1$ and the output weight is $2$, the corresponding effective exterior profiles are
\[
q,\qquad q,\qquad 2q.
\]
Thus the two inputs belong to $\mathcal P$ and the output to $2\mathcal P$. 

For statements involving the product, shrink $\hbar$ if necessary so that
\[
0<\hbar<e_0(B)-2(\eta_{\mathrm{end}}+\sigma),
\qquad
3\hbar<e_0(B).
\]
All preceding estimates remain valid after this decrease. Let $A$ be a regular S-shaped acceleration datum for $\mathcal P$ with normalized exterior profile $q$, and let $2A$ be a regular S-shaped acceleration datum for $2\mathcal P$ with normalized exterior profile $2q$. After passing to cofinal subsequences and relabelling, choose pair-of-pants maps
\[
B_i:
CF^*(A_i)\otimes CF^*(A_i)
\longrightarrow
CF^*((2A)_i)
\]
and one-parameter comparison families filling
\begin{equation}\label{eq1paramProdComp}
\xymatrix{
CF^*(A_i)\otimes CF^*(A_i)
 \ar[r]^{B_i}\ar[d]_{\kappa_i\otimes\kappa_i}
&
CF^*((2A)_i)\ar[d]^{\kappa_i^{2A}}
\\
CF^*(A_{i+1})\otimes CF^*(A_{i+1})
 \ar[r]_{B_{i+1}}
&
CF^*((2A)_{i+1}).
}
\end{equation}
These operation data and comparison families may be chosen admissible in the sense of Definition~\ref{dfSShaped}. On $M\setminus D_{\delta'}$, take the normalized base exterior Hamiltonian to be $q$, independent both of the pair-of-pants variable and of the comparison parameter. Hence
\[
d_\Sigma H_z^e=0
\]
there. In the noncompact case, take the base coefficient of $f_0$ to be $\epsilon_\infty(\mathcal P)$; the effective output coefficient is then $2\epsilon_\infty(\mathcal P)=\epsilon_\infty(2\mathcal P)$.

The remaining domain dependence is confined to the compact interior, where the standard pair-of-pants and comparison data may be chosen to satisfy
\[
d_\Sigma H_z\wedge\alpha\ge0.
\]
Therefore Corollary~\ref{CYRelFilrPres1} applies to the maps $B_i$ and to their comparison families. In the virtual setting we use the same Floer data and interpret the resulting counts using the fixed Floer-local virtual theory.

\begin{lm}\label{lmAlgebraFiltered}
Let $u$ be a Floer solution for an admissible S-shaped operation datum and
let $\hbar$ satisfy \eqref{eqHbarGap}. Then $E_{M,\theta}(u)\geq 0$, and if
$E_{M,\theta}(u)\neq 0$ then $E_{M,\theta}(u)>\hbar$. If
$E_{M,\theta}(u)=0$ and the output is local, then $u$ is entirely local.
Consequently the exterior generators form an ideal for the order-zero
product and BV operator. The same conclusions hold uniformly for compact
parameter families, including the $S^1$-family defining BV.
\end{lm}
\begin{proof}
Choose the common witnessing $\delta'$ and a regular
$\epsilon\in[\delta',\delta]$. For an exterior end $e$, let
$p_e$ denote its asymptotic critical point and set
\[
g_e
:=
H_e^{\mathrm{eff}}(p_e)
-
H_e^{\mathrm{eff}}|_{\rho=\epsilon}.
\]
Thus
\[
0\leq g_i\leq\eta_{\mathrm{end}}/n
\]
at an exterior input and
\[
0\leq g_0\leq\eta_{\mathrm{end}}
\]
at an exterior output.

We prove
\begin{equation}\label{eqEMineq2ops}
E_{M,\theta}(u)
\geq
E^{\mathrm{geo}}\bigl(u|_{S_{\epsilon}}\bigr)
+\sum_{\mathrm{ext.\ inputs}\,i}\bigl(g_i+\sigma/n\bigr)
-(g_0+\sigma)\nu_0,
\end{equation}
where $S_{\epsilon}$ is the exterior cut domain and
$\nu_0\in\{0,1\}$ indicates an exterior output.

The operation energy identity is
\begin{equation}
E_{\mathrm{top}}(u_{\epsilon})
=
E^{\mathrm{geo}}(u_{\epsilon})
+
\int_{S_{\epsilon}}d_{\Sigma}H_z^{e}\wedge\alpha,
\end{equation}
and the last term is nonnegative by \eqref{eqOpMonotone}.

The two kinds of terms in the boundary correction have different
origins. The $g_e$-terms come from
$-\int_{\partial u_\epsilon}H\alpha$ on the punctured cut domain; the
period $w_e=\int_{S^1_e}\alpha$ turns the gap for the base Hamiltonian
into the gap for the effective Hamiltonian
$H_e^{\mathrm{eff}}=w_eH_e$. The barrier terms come from
$-\int_{\partial\overline u_\epsilon}\theta$ on the compactified
relative chain. The constant caps at exterior punctures contribute
nothing, while on the collar cut the slope-$\sigma/n$ equation gives
\[
u^*\theta
\leq
-\frac{\sigma}{n}\alpha.
\]
Hence an exterior input, of weight $1$, contributes
$+(g_i+\sigma/n)$, whereas an exterior output, of weight $n$, contributes
$-(g_0+\sigma)$. This proves \eqref{eqEMineq2ops}. The worst possible
negative boundary contribution is therefore
$-(\eta_{\mathrm{end}}+\sigma)$.

If the output is local and $u$ leaves the local region then
$E_{M,\theta}(u)>0$. If a compactified exterior component carries nonzero
relative energy, Proposition \ref{PrpRelHbar} gives
$E^{\mathrm{geo}}(u_{\epsilon})\geq E^{\mathrm{geo}}(v)>e_0(B)$, whence
$E_{M,\theta}(u)\geq e_0(B)-(\eta_{\mathrm{end}}+\sigma)>\hbar$ by the
package gap \eqref{eqHbarGap}. Zero-relative-energy operations with local
output are entirely local, so an order-zero contribution with local output
cannot involve an exterior input, and the exterior generators form an
ideal for the order-zero product and BV operator.
\end{proof}

\subsection{Two filtrations on the Floer complex}

For the remainder of the discussion we will refer to S-shaped Hamiltonians and to the constant $\hbar$ appearing in Lemmas \ref{lmPositivity} and \ref{lmPositivity2}. This presumes an implicit specification of the parameters $\delta,\sigma,\eta_{\mathrm{end}}$ and of the gap condition \eqref{eqHbarGap}. Moreover, when we talk about continuation maps from $(H_0,J_0)$ to $(H_1,J_1)$, it is assumed these are connected by an admissible path for the same implicit choice of parameters. 

For a regular Floer datum $(H,J)$  the Floer complex $CF^*(H,J)$ is equipped with  a non-Archimedean filtration $\val:CF^*(H,J)\to\bR$ induced by setting $\val(T^{\lambda}a_{\gamma})=\lambda$ for $a_{\gamma}\neq 0\in o_{\gamma}$. The filtration $\val$ defines a  non-Archimedean norm by $|x|=e^{-\val(x)}$. We call this the \emph{norm filtration}.  The inequality 
\begin{equation}\label{eqEtopEgeoNorm}
E_{top}(u)\geq E_{geo}(u)\geq0
\end{equation}
for a monotone Floer datum implies that the Floer differential and the Floer continuation maps are norm non-increasing. 

 Fix a local primitive $\theta$ and consider \emph{$S$-shaped Floer data} $(H,J)$. Using the results of \S\ref{secPositivity} we construct an additional relative valuation on $CF^*(H,J)$. It is defined by 
\begin{equation}\label{eqRelFiltration}
\val_{M,\theta,H}\left(T^{\lambda}a_{\gamma}\right):=\lambda+\cA_{H,\theta}(\gamma).%+\epsilon(\gamma). 
\end{equation}
%where $\epsilon(\gamma)=2\hbar$ if $\gamma$ is an outside critical point and $\epsilon(\gamma)=0$ otherwise. 
We refer to this as the \emph{relative valuation}. We drop $\theta$ or $H$ from the notation when there is no ambiguity. On finite Floer complexes it takes values in $\bR\cup\{+\infty\}$ with $\val(0)=+\infty$; after telescope completion it becomes an \emph{extended} relative valuation $\val_{M,\theta}:SC^*_M(K,A)\to\bR\cup\{\pm\infty\}$, which may take the value $-\infty$. We call the set $\left\{T^{-\cA_H(\gamma)}a_{\gamma}\right\}$ where $\gamma$ runs over the set of $1$-periodic orbits of $H$ the \emph{action normalized basis} of $CF^*(H)$.

Let us now show that $\val_{M,\theta,H}$ indeed defines a filtration. For a Floer solution $u$ with input $\gamma_0$ and output $\gamma_1$ write $$\Delta_{M,\theta}(u):= E_{top}(u)+\val_{M,\theta,H_1}(\gamma_1)-\val_{M,\theta,H_0}(\gamma_0).$$ From now on we omit $H$ from the notation and write $\val_{M,\theta}.$

\begin{lm}\label{lmEMthetaDelta}
For any Floer solution $u$ with input $\gamma_0$ and output $\gamma_1$, we have $E_{M,\theta}(u)=\Delta_{M,\theta}(u)$.
\end{lm}
\begin{proof}
Choose a global smooth one-form $\lambda$ extending $\theta$ on
$V(D)$. For a local orbit $\gamma\subset V(D)$ one has
$\int_\gamma\lambda=\int_\gamma\theta$, since $\lambda=\theta$ there.
For an exterior orbit $p$, the S-shaped hypotheses make $p$ a constant
critical orbit, hence $\int_p\lambda=0$. Therefore in both cases
\[
\cA_{H,\theta}(\gamma)=-\int_\gamma\lambda-\int H\,dt.
\]
Likewise
\[
E_{M,\theta}(u)=\int u^*\omega-\int_{\partial\overline u}\lambda,
\]
since constant-disc compactifications contribute zero. With
$E_{\mathrm{top}}(u)=\int u^*\omega+\int_{\partial u}H\,dt$,
\[
E_{M,\theta}(u)
=E_{\mathrm{top}}(u)+\cA_{H_1,\theta}(\gamma_1)-\cA_{H_0,\theta}(\gamma_0)
=\Delta_{M,\theta}(u).
\]
\end{proof}

\begin{cy}\label{lmTrichotomy}
For $u$ a finite energy solution of a translation-invariant equation or of an admissible S-shaped continuation we have 
$$\Delta_{M,\theta}(u)\geq 0,$$ and if $\Delta_{M,\theta}(u)\neq 0$ then $\Delta_{M,\theta}(u)\geq \hbar$.
Moreover, we have the following trichotomy: either
\begin{itemize}
\item $\Delta_{M,\theta}(u)\geq\hbar$,
\item $u$ is contained in $V(D)$, or,
\item The output is outside.
\end{itemize}
\end{cy}
\begin{proof}
By Lemma \ref{lmEMthetaDelta}, we have $E_{M,\theta}(u)=\Delta_{M,\theta}(u)$. So, the claim follows from Lemmas \ref{lmPositivity} and \ref{lmPositivity2}.
\end{proof}

In the following result we consider a general ground integral domain $R$.
\begin{cy}\label{CYRelFilrPres}

\begin{itemize}
\item For any $S$-shaped Floer datum $(H,J)$, the differential on $CF^*(H,J)$ preserves the filtration $\val_{M,\theta}$. Moreover, we have $d=d_0+\delta$ where $d_0$ is defined over the ground ring $R$ in an action normalized basis, $\val_{M,\theta}(\delta)\ge\hbar$, $|\delta|_\infty<e^{-3\hbar}$, and $d_0$ is norm non-increasing. The full local/outside trichotomy applies to contributions of the differential.
\item Given a pair $(H_1,J_1)\leq (H_2,J_2)$ of S-shaped Floer data and an admissible S-shaped continuation connecting them,  the continuation map $CF^*(H_1,J_1)\to CF^*(H_2,J_2)$ preserves  $\val_{M,\theta}$.  Moreover, we have $\kappa=\kappa_0+\kappa_+$ where $\kappa_0$ is defined over $R$, $\val_{M,\theta}(\kappa_+)\ge\hbar$, $|\kappa_+|_\infty<e^{-3\hbar}$, and $\kappa_0$ is norm non-increasing. The full local/outside trichotomy applies.
\item Given an S-shaped homotopy between a pair of monotone S-shaped Floer as in the previous parts, the corresponding homotopy $h: CF^*(H_1,J_1)\to CF^{*-1}(H_2,J_2)$ preserves $\val_{M,\theta}$ and can be written as $h=h_0+h_+$ where $h_0$ is defined over $R$, $\val_{M,\theta}(h_+)\ge\hbar$, $|h_+|_\infty<e^{-3\hbar}$, and $h_0$ is norm non-increasing.

\end{itemize}
\end{cy}
\begin{proof}
Consider a solution $u$ with  input $\gamma_0$ and  output $\gamma_1$ and denote by $f_u$ the corresponding contribution to an operation on the Floer complex. Then $f_u$ maps $\gamma_0$ to $T^{E_{top}(u)}\gamma_1$. Here we abuse notation and refer by $\gamma_i$ to appropriate generators of the corresponding orientation line. So, we have $\val_{M,\theta}(f_u(\gamma_0))-\val_{M,\theta}(\gamma_0)=\Delta_{M,\theta}(u)\geq 0$. When $\Delta_{M,\theta}(u)\neq 0$, Lemma \ref{lmTrichotomy} tells us $\Delta_{M,\theta}(u)\geq \hbar$. Similarly, by Proposition \ref{PrpRelHbar} we have $E_{top}(u)\geq E^{\mathrm{geo}}(u)>e_0(B)>3\hbar$ whenever an exterior component carries nonzero relative energy, using the quantitative reserve $3\hbar<e_0(B)$ fixed with the package. It follows that $f_u$ increases $\val_{M,\theta}$ by at least $\hbar$ and decreases norm by a factor $<e^{-3\hbar}$. The claim for the operations associated with a $0$ dimensional family of Floer data follows. The case of a homotopy then follows since all solutions contributing to $h$ are solutions to the continuation equation for some intermediate path.
\end{proof}

\subsection{The relative filtration and operations}
More generally, let $\Sigma$ be a punctured Riemann surface with one output labeled by $0$ and $n$ inputs labeled by $1\leq i\leq n$ for $n\leq 2$.  Decorate $\Sigma$ with an admissible S-shaped operation datum connecting asymptotic data $(H_i,J_i)$ for $0\leq i\leq n$ at each end, with end weights $w_i=1$ on the inputs and $w_0=n$ on the output. Let $u:\Sigma\to M$ be a Floer solution with output $\gamma_0$ and inputs $\gamma_i$ for $1\leq i\leq n$. Write $H_i^{\mathrm{eff}}:=w_i H_i$ for the effective asymptotic Hamiltonians and
$$\Delta_{M,\theta}(u):= E_{\mathrm{top}}(u)+\val_{M,\theta,H_0^{\mathrm{eff}}}(\gamma_0)-\sum_{i=1}^n\val_{M,\theta,H_i^{\mathrm{eff}}}(\gamma_i).$$

\begin{lm}\label{lmOpEMthetaDelta}
For any such operation solution $u$ one has
$E_{M,\theta}(u)=\Delta_{M,\theta}(u)$.
\end{lm}
\begin{proof}
Choose the same global extension $\lambda$ of $\theta$ as in Lemma
\ref{lmEMthetaDelta}. By Stokes on the punctured operation surface, with
boundary orientation equal to the output end minus the input ends,
\[
E_{M,\theta}(u)
=
\int u^*\omega-\int_{\partial\overline{u}}\lambda
=
E_{\mathrm{top}}(u)
+\cA_{H_0^{\mathrm{eff}},\theta}(\gamma_0)
-\sum_{i=1}^n\cA_{H_i^{\mathrm{eff}},\theta}(\gamma_i).
\]
Here $H_e^{\mathrm{eff}}=w_eH_e$, with $w_i=1$ on the inputs and $w_0=n$ on
the output. On a local end one has $\lambda=\theta$, while on an exterior
end the orbit is a constant critical orbit so the $\lambda$-integral
vanishes; in both cases the action identity of Lemma~\ref{lmEMthetaDelta}
applies to the effective Hamiltonians. Since
$\val_{M,\theta,H_e^{\mathrm{eff}}}(\gamma_e)=\cA_{H_e^{\mathrm{eff}},\theta}(\gamma_e)$
on generators, the right-hand side equals $\Delta_{M,\theta}(u)$.
\end{proof}

\begin{lm}\label{lmTrichotomy2}
For $u$ a finite energy solution we have 
$$\Delta_{M,\theta}(u)\geq 0,$$ and if $\Delta_{M,\theta}(u)\neq 0$ then $\Delta_{M,\theta}(u)\geq \hbar$.
Moreover, we have the following trichotomy: either
\begin{itemize}
\item $\Delta_{M,\theta}(u)\geq\hbar$,
\item $u$ is contained in $V(D)$, or,
\item The output is outside.
\end{itemize}

\end{lm}

\begin{proof}
By Lemma \ref{lmOpEMthetaDelta} one has
$\Delta_{M,\theta}(u)=E_{M,\theta}(u)$, so the claim follows from Lemma
\ref{lmAlgebraFiltered}.
\end{proof}

This leads to the following whose proof we omit since it is similar to the proof of Corollary \ref{CYRelFilrPres}. 
\begin{cy}\label{CYRelFilrPres1}
    Given an admissible S-shaped datum on the pair of pants, the induced product map
    $B:CF^*(H_1)\otimes CF^*(H_2)\to CF^*(H_3)$ preserves the valuation $\val_{M,\theta}$ and can be written as $B=B_0+B_+$, where $B_0$ is defined over $R$, counts the order-zero contributions, and is norm non-increasing, $\val_{M,\theta}(B_+)\ge\hbar$, and $|B_+|_\infty<e^{-3\hbar}$. A similar claim holds for the BV operator, written $\Delta=\Delta_0+\Delta_+$ with $\val_{M,\theta}(\Delta_+)\ge\hbar$ and $|\Delta_+|_\infty<e^{-3\hbar}$. Compact fixed-arity parameter families needed to compare choices of such data likewise preserve the filtration and admit the same decomposition.
\end{cy}

\subsection{The relative Novikov lattice}
\label{subsecRelativeNovikovLattice}

Fix an S-shaped acceleration datum for some $K\subset D$. As discussed in \ref{SecSHReview} this gives rise via the telescope construction to a chain complex $SC^*_M(K,A)$ over the Novikov ring which is complete with respect to the $T$-adic norm and whose differential preserves the $T$-adic norm.

Each summand $CF^*(H_i)[q]$ carries $\val_{M,\theta,H_i}$ from
\eqref{eqRelFiltration}. On the completed telescope, for
\[
x=\sum_\gamma(a_\gamma+b_\gamma q)\gamma
\]
set
\begin{equation}\label{eqRelFiltrationCompletion}
\val_{M,\theta}(x)
:=
\inf_\gamma
\bigl(
\val_{M,\theta}(\gamma)
+
\min\{\val(a_\gamma),\val(b_\gamma)\}
\bigr)
\in
\bR\cup\{\pm\infty\},
\end{equation}
with the convention $\val_{M,\theta}(0)=+\infty$.
By Corollary~\ref{CYRelFilrPres}, the Floer differentials and all
normalized-monotone continuation maps preserve this filtration. Hence so
do the telescope differential and the induced normalized-monotone
telescope maps.

\begin{df}\label{dfRelativeNovikovLattice}
For a regular S-shaped acceleration datum $A$ the \emph{relative Novikov lattice} is the closed $\Lambda_{\geq0}$-subcomplex
\[
SC^{*,\geq0}_{M,\theta}(K,A)
:=
F^{\val_{M,\theta}\geq0}SC^*_M(K,A)
\]
consisting of elements $x$ satisfying
\[
\val_{M,\theta}(x)\geq0.
\]
In a telescope model it is the closed $\Lambda_{\geq0}$-submodule generated by the action-normalized generators
\[
T^{-\cA_{H_i}(\gamma)}a_{\gamma_i}q^j
\]
(equivalently, the set of norm-convergent $\Lambda_{\geq0}$-linear combinations of these generators).
\end{df}

\begin{lm}\label{lmRelFiltrationClosed}
For every $a\in\bR$, the submodule
\[
F^{\val_{M,\theta}\geq a}SC^*_M(K,A)
\]
is closed in the ambient telescope norm.
\end{lm}
\begin{proof}
Write $x_i=\sum(a^i_\gamma+b^i_\gamma q)\gamma$. Convergence in the
completed telescope means that for each fixed generator $\gamma$ the
coefficient projections $x\mapsto(a_\gamma,b_\gamma)$ are continuous
for the norm topology. If $x_i\to x$ and each $x_i$ has
$\val_{M,\theta}(x_i)\geq a$, then for each fixed $\gamma$,
\[
\val(a^i_\gamma)+\val_{M,\theta}(\gamma)\geq a,
\qquad
\val(b^i_\gamma)+\val_{M,\theta}(\gamma)\geq a,
\]
and these closed inequalities persist under the coefficientwise limit.
Hence every coefficient of $x$ satisfies the same bounds, and taking
the infimum in \eqref{eqRelFiltrationCompletion} gives
$\val_{M,\theta}(x)\geq a$.
\end{proof}

The strict sublevels $F^{\val_{M,\theta}>a}$ need not be closed; throughout, the ambient Floer norm is the completion topology, while the relative valuation is retained as an additional filtration.

\begin{lm}\label{lmRelativeLatticeCompletion}
For every regular S-shaped model,
the natural map
\[
SC^{*,\geq0}_{M,\theta}(K,A)
\otimes_{\Lambda_{\geq0}}\Lambda
\longrightarrow
SC_M^*(K,A)
\]
has norm-dense image. Consequently
\[
SC_M^*(K,A)
\cong
\widehat{
SC^{*,\geq0}_{M,\theta}(K,A)
\otimes_{\Lambda_{\geq0}}\Lambda
}^{\,\|\cdot\|},
\]
where the completion is taken in the ambient Floer norm.
\end{lm}
\begin{proof}
An element of finite telescope support has finite relative valuation. After multiplication by a sufficiently large $T^a$ it therefore lies in the nonnegative lattice. Finite-support telescope elements are norm dense in the completed telescope, so scalar extension of the lattice is norm dense.
\end{proof}

\begin{rem}
At cohomology level scalar extension gives
\[
H^*(SC^{*,\geq0}_{M,\theta}(K,A))\otimes_{\Lambda_{\geq0}}\Lambda,
\]
whereas $SH_M^*(K)$ is obtained by taking cohomology after ambient-norm completion. These need not agree because norm completion need not commute with cohomology.
\end{rem}

\paragraph{\textbf{Pair-of-pants product on completed homology}.}
Choose acceleration data $A,2A$ and pair-of-pants data as in
Subsection~\ref{subsecSShapedOperations}. The levelwise maps $B_i$ and their one-parameter comparison homotopies from \eqref{eq1paramProdComp} define a morphism from the diagonal tensor-product ray
$$
\bigl(CF^*(A_i)\otimes CF^*(A_i)\bigr)_{i\geq 1}
$$
to the acceleration ray $2A$. Write
$$
D_A
:=
\widehat{\operatorname{Tel}}
\bigl(CF^*(A_i)\otimes CF^*(A_i)\bigr)
$$
for its completed telescope, and let
$$
\widetilde\mu_A
\colon
D_A
\longrightarrow
SC_M^*(K,2A)
$$
be the chain map induced by the levelwise maps and comparison homotopies. By the standard diagonal-cofinality comparison of \cite[\S3.3]{tonkonog}, $D_A$ is quasi-isomorphic to the completed tensor product of the two telescopes, compatibly with the pair-of-pants operation; hence the induced homology operation is the usual product after weight doubling.
All filtration and norm estimates
below are carried out directly on $D_A$.

Taking relative-order-zero parts of the level maps and coherence
homotopies yields an order-zero ray morphism; write
$\widetilde\mu_{A,0}$ for the induced chain map
$D_A\to SC_M^*(K,2A)$ with respect to the order-zero differentials, and set
\[
\widetilde\mu_A
=
\widetilde\mu_{A,0}
+
\widetilde\mu_{A,+}.
\]
This is a decomposition of the underlying bounded operator according to
relative order: $\widetilde\mu_A$ is a chain map for the full
differentials, $\widetilde\mu_{A,0}$ is a chain map for the order-zero
differentials, and $\widetilde\mu_{A,+}$ is their positive-order
difference. We do not choose a bilinear chain map
\[
SC_M^*(K,A)\widehat\otimes SC_M^*(K,A)
\longrightarrow
SC_M^*(K,A).
\]

The telescope construction is norm non-increasing and preserves
$\val_{M,\theta}$. By Corollary~\ref{CYRelFilrPres1}, applied to the
levelwise products and to their comparison homotopies, the positive-order
part satisfies
\[
\val_{M,\theta}(\widetilde\mu_{A,+})\geq\hbar,
\qquad
|\widetilde\mu_{A,+}|_\infty<e^{-3\hbar}.
\]
The induced map on homology is the bounded, filtration-preserving
operation
\[
\mu_A:
H^*(SC_M^*(K,A))\otimes H^*(SC_M^*(K,A))
\longrightarrow
H^*(SC_M^*(K,2A)).
\]

Because the levelwise pair-of-pants maps and all coherence homotopies
preserve the relative valuation, the same diagonal construction is defined
over the relative Novikov lattices, may be applied after taking the chain
quotient modulo $T^\hbar$, and commutes with the chain-level reduction
maps. 
\subsection{The general case: virtual chains approach}

Since $g_J$ is uniformly comparable to the geometrically bounded
reference metric $g_0$, the comparison version of
Lemma~\ref{lmMonEst++} gives a constant $e_{\mathrm{sph}}(B)>0$
such that every nonconstant $J$-holomorphic sphere for $J\in B$
has symplectic area greater than $e_{\mathrm{sph}}(B)$. In the
virtual-chain setting we decrease the constant $\hbar$, if necessary,
so that in addition
\begin{equation}\label{eqVirtualSphereReserve}
3\hbar<e_{\mathrm{sph}}(B).
\end{equation}
All preceding S-shaped estimates remain valid after this decrease.

For an S-shaped Floer datum $(H,J)$, let $\mathcal P_H$ denote the
objects of the associated Kuranishi flow category $\mathbb X_H$.
For $p,p'\in\mathcal P_H$ and an energy label $\lambda\geq 0$, write
\[
m_{\lambda,p,p'}:o_p\longrightarrow o_{p'}
\]
for the corresponding virtual contribution to the differential, and
set
\[
\Delta_{M,\theta}(\lambda;p,p')
:=
\lambda+\mathcal A_{H,\theta}(p')
-\mathcal A_{H,\theta}(p).
\]

\begin{lm}
\label{lmVirtualTrichotomy}
Suppose
\[
m_{\lambda,p,p'}\neq 0.
\]
Then
\[
\Delta_{M,\theta}(\lambda;p,p')\geq 0,
\]
and if this number is nonzero then
\[
\Delta_{M,\theta}(\lambda;p,p')\geq\hbar.
\]
Moreover, if
$\Delta_{M,\theta}(\lambda;p,p')=0$ and $p'$ is local, every stable
configuration in
\[
\overline{\mathcal M}_{\lambda}(p,p')
\]
has no nonconstant sphere components and all of its Floer components
are contained in $V(D)$. If the filtration change is zero and the
configuration is not entirely local, its output is an exterior
critical point.

The corresponding statements hold for admissible continuation maps
and continuation homotopies. For pair-of-pants and BV operations,
including the finite cubical chains used to represent their compact
parameter cycles and to compare choices, the same conclusions hold
with $\Delta_{M,\theta}$ defined as in
Lemma~\ref{lmOpEMthetaDelta}. Consequently the filtration and norm
estimates of Corollaries~\ref{CYRelFilrPres} and
\ref{CYRelFilrPres1} remain valid with virtual counts.
\end{lm}

\begin{proof}
By Lemma~\ref{lmVirtualSupport}, a nonzero virtual coefficient
has a nonempty compactified moduli space with the prescribed asymptotics
and energy label. For any stable configuration $u$, with Floer vertices
$u_v$ and nonconstant sphere components $S_j$, additivity of the energy
label and cancellation of the action terms at internal asymptotics give
\begin{equation}\label{eqStableVirtualDelta}
\Delta_{M,\theta}(u)
=
\sum_v E_{M,\theta}(u_v)
+
\sum_j\int_{S_j}\omega,
\end{equation}
with Lemma~\ref{lmOpEMthetaDelta} used at multi-input vertices. By
Corollary~\ref{lmTrichotomy} and Lemma~\ref{lmTrichotomy2}, every
Floer-vertex term is nonnegative and is either zero or at least
$\hbar$; every nonconstant sphere has area
$>e_{\mathrm{sph}}(B)>3\hbar$. Hence $\Delta_{M,\theta}\ge0$, and if
nonzero then $\Delta_{M,\theta}\ge\hbar$.

If $\Delta_{M,\theta}=0$ and the output is local, there are no
nonconstant sphere components and every Floer vertex has zero relative
energy. Applying the zero-energy alternative of
Corollary~\ref{lmTrichotomy} or Lemma~\ref{lmTrichotomy2} at the root
and inducting through the breaking tree shows that every Floer
component is contained in $V(D)$. Thus a zero-energy configuration
which is not entirely local must have exterior output. The same
argument applies to continuation configurations and to the finite
cubical families used for the product and BV operator.

Finally, a positive Floer-vertex contribution has the
$<e^{-3\hbar}$ norm reserve of Corollaries~\ref{CYRelFilrPres} and
\ref{CYRelFilrPres1}, while a sphere contribution has area $>3\hbar$;
additivity of the energy label gives the same norm estimate for the
stable virtual coefficient. The finite cubical-chain statements are
termwise.
\end{proof}

\section{Reduction of SH and the intrinsic SH}\label{SecReduction}

\subsection{Locality and reduction}\label{SubSecRedSH}

\emph{In this section we continue working with a general ground integral domain $R.$}

\begin{df}
The \emph{reduction} of the symplectic cochain complex is defined to be the quotient of the relative Novikov lattice
$$\gr_{\hbar}SC^*_M(K,A)
:=
SC^{*,\geq0}_{M,\theta}(K,A)/
T^\hbar SC^{*,\geq0}_{M,\theta}(K,A).$$
\end{df}

Note that $\gr_{\hbar}SC^*_M(K,A)$ is a normed chain complex over $\Lambda_{[0,\hbar)}$ with the quotient norm induced by the ambient Floer norm.

\begin{lm}
    $\gr_{\hbar}SC^*_M(K,A)$ is functorial with respect to inclusions of compact sets in the following sense.
    For $K_1\subset K_2$ and acceleration data $A_2\leq A_1$ for each, the restriction map in $SC^*_M(K,A)$ descends to $\gr_{\hbar}SC^*_M(K,A)$ and is functorial on homology.
\end{lm}
\begin{proof}
    This is immediate since continuation maps preserve the relative Novikov lattice and the ambient Floer norm.
\end{proof}

\begin{lm}\label{lmReducedSHExactCase}
    Suppose $(M,\omega=d\theta)$ is exact with $M$ geometrically bounded. Then for any $K\subset M$ and any acceleration datum we have an isomorphism 
    \begin{equation}
        \gr_{\hbar}SC^*_M(K,A) =SC^*_{M,\theta}(K,A;R)\otimes \Lambda_{[0,\hbar)}.
    \end{equation}
\end{lm}
\begin{proof}
    This follows from Lemma \ref{lmUnweightedSHcomparison} and the definition of $\val_{M,\theta}$.
\end{proof}

Let $A$ be an S-shaped acceleration datum for the pair $(M,D)$. By definition, the Hamiltonians in $A$ have constant small radial slope on $D_{\delta}\setminus D_{\delta'}$ for a witnessing inner level $\delta'\in(\delta/2,\delta)$. For each level $(H_i,J_i)$ of $A$, define the corresponding local datum $(H_{\widehat{D},i},J_{\widehat{D},i})$ on $\widehat{D}$ so that $H_{\widehat{D},i}=H_i$ and $J_{\widehat{D},i}=J_i$ on the whole local region $D_{\delta}$. Outside $D_{\delta}$, extend the collar Hamiltonian with the same constant radial slope as $A$ on $D_{\delta}\setminus D_{\delta'}$ and extend $J_{\widehat{D},i}$ cylindrically/contact-type in the standard way. We assume this slope is not in the period spectrum of $\partial D$. Make the same levelwise choice for continuation data when needed. This defines an acceleration datum $A_{\widehat{D}}$ on $\widehat{D}$. On $D_{\delta}$ the generators and moduli spaces contributing to $A_{\widehat{D}}$ agree  with those of $A$. 

\begin{lm}\label{lmExteriorOrderZeroRetract}
Let $A$ be a regular S-shaped acceleration datum for $K\subset D$, and write
\[
C=SC^*_M(K,A)
\]
with order-zero differential $d_0$. Let $O\subset(C,d_0)$ be the completed
subcomplex generated by the exterior critical-point telescope generators.
Then $O$ admits a norm-nonincreasing contraction
\[
k=\sum_{m\geq0}Q\kappa^m,
\qquad
d_Ok+kd_O=\id_O,
\qquad
|k|_\infty\leq1,
\]
where $\kappa$ is the order-zero continuation shift on $O$ and $Q$ is signed
multiplication by the telescope variable $q$. Consequently the full
order-zero complex admits a norm-$\leq1$ deformation retract onto the
quotient by the exterior summand.
\end{lm}
\begin{proof}
The order-zero trichotomy implies that $O$ is a subcomplex: an order-zero
trajectory with local output is entirely local. Writing
\[
d_0Q+Qd_0=\id-\kappa
\]
on $O$, the normalized exterior profile is fixed while the exterior
additive constants tend to $+\infty$ along the acceleration ray, so an
action-normalized exterior generator at level $i$ has ambient norm
tending to zero under successive continuation shifts. Thus
$\kappa^mx\to0$ in the ambient telescope norm. Consequently
\[
k:=\sum_{m\geq0}Q\kappa^m
\]
converges and satisfies
\[
d_Ok+kd_O=\id_O,
\qquad
|k|_\infty\leq1.
\]

Put $I=C/O$. Using the splitting by local and exterior generators, write
\[
C=I\oplus O,
\qquad
d_0=
\begin{pmatrix}
d_I&0\\
f&d_O
\end{pmatrix}.
\]
Then
\[
p(x,o)=x,
\qquad
s(x)=(x,-kf(x)),
\qquad
H(x,o)=(0,k(o))
\]
give a deformation retract
\[
(C,d_0)
\ \underset{s}{\overset{p}{\rightleftarrows}}\
(I,d_I),
\]
with
\[
ps=\id_I,
\qquad
\id_C-sp=d_0H+Hd_0.
\]
Since $d_0$ is norm non-increasing, the block $f$ has operator norm at
most one. Hence $p$, $s$, and $H$ all have relative order zero and
operator norm at most one.
\end{proof}
In the following, for an S-shaped acceleration datum $A$ denote by $O_A$ the completed subcomplex generated by the exterior generators. 
\begin{prp}[Fixed-model locality]\label{prpFixedModelLocality}
 There is an
isometric chain isomorphism
\[
\bigl(SC^*_M(K,A),d_0\bigr)/O_A
\;\cong\;
SC^*_{\widehat{D},\theta}(K,A_{\widehat{D}};R)\widehat{\otimes}_R\Lambda.
\]
Moreover, if
\[
\Phi_0:
\bigl(SC^*_M(K_2,A_2),d_0\bigr)
\longrightarrow
\bigl(SC^*_M(K_1,A_1),d_0\bigr)
\]
is the order-zero part of an S-shaped normalized-monotone
continuation/restriction map, then $\Phi_0(O_{A_2})\subset O_{A_1}$ and under the above identifications, the induced map on the quotients 
is the intrinsic restriction map
$$
SC^*_{\widehat{D},\theta}(K_2,A_{\widehat{D}};R)\widehat{\otimes}_R\Lambda\to SC^*_{\widehat{D},\theta}(K_1,A_{\widehat{D}};R)\widehat{\otimes}_R\Lambda.
$$

\end{prp}

\begin{proof}
We first present the proof in the setting of regular Floer data. At the end we will comment on adjustment needed when working with virtual chains.

The order-zero trichotomy of Lemma~\ref{lmTrichotomy} implies that an
order-zero trajectory with local output remains entirely local, so the
exterior critical points form a subcomplex $O_A$. Lemma~\ref{lmExteriorOrderZeroRetract}
supplies the stated relative-order-zero contraction of norm at most $1$.
After quotienting by $O_A$ one is left with precisely the local generators
and the local order-zero trajectories; denote this quotient by $I$.

There is a bijection between the generators of $I$ and those of the
intrinsic S-shaped telescope determined by $A_{\widehat{D}}$. By the
integrated maximum principle \cite{AbouzaidSeidel2010}, all solutions
contributing to the intrinsic telescope are contained in $D_{\delta}$.
The same goes for those contributing to $I$ by Lemma~\ref{lmTrichotomy}.
Therefore the identification of generators is intertwined by all the
differential and continuation maps. Combined with
Lemma~\ref{lmUnweightedSHcomparison}, this yields the displayed isometric
chain isomorphism.

For continuation and restriction maps, apply Lemma~\ref{lmTrichotomy} to
normalized-monotone continuation maps between acceleration data
$A_2\leq A_1$. The maps preserve the exterior subcomplexes, and there is
a bijection between the order-zero continuation trajectories contributing
to the maps of the quotients and those for the corresponding intrinsic
complexes.

Suppose now that the Floer complexes are defined using a Floer-local
theory of virtual counts. Lemma~\ref{lmVirtualTrichotomy} gives the same
order-zero locality statement: a nonzero order-zero coefficient with local
output has no nonconstant sphere components, and all of its Floer components
lie in $V(D)$. Thus the exterior generators form the same subcomplex $O_A$.

On the local quotient, the ambient and intrinsic Floer problems 
have the same germ in the sense of Definition~\ref{dfLocalVirtualCounts} along the relevant compactified moduli spaces. Proposition~\ref{prpVirtualLocality}
therefore identifies the corresponding virtual coefficients. The same argument
applies to the normalized-monotone continuation and restriction maps.

\end{proof}

\begin{prp}\label{prpGrLocality}
There is a surjection of normed chain complexes over $\Lambda_{[0,\hbar)}$
\[
\gr_{\hbar}SC^*_M(K,A)
\longrightarrow
SC^*_{\widehat{D},\theta}(K,A_{\widehat{D}};R)\otimes_R\Lambda_{[0,\hbar)}
\]
whose kernel admits a norm-nonincreasing contraction. Moreover, for
$K_1\subset K_2\subset D$ with acceleration data $A_2\leq A_1$ for each,
and $A_{i,\widehat{D}}$ the corresponding local acceleration data for
$K_i\subset\widehat{D}$, the following diagram commutes:
\begin{equation}\label{eqProp44Commutative}
\xymatrix{
\gr_{\hbar}SC^*_M(K_2,A_2) \ar[d] \ar[r] & SC^*_{\widehat{D},\theta}(K_2,A_{2,\widehat{D}};R)\otimes_R\Lambda_{[0,\hbar)} \ar[d] \\
\gr_{\hbar}SC^*_M(K_1,A_1) \ar[r] & SC^*_{\widehat{D},\theta}(K_1,A_{1,\widehat{D}};R)\otimes_R\Lambda_{[0,\hbar)}.
}
\end{equation}
\end{prp}

\begin{proof}
The identification of the intrinsic reduction with
$SC^*_{\widehat{D},\theta}(K,A_{\widehat{D}};R)\otimes_R\Lambda_{[0,\hbar)}$
is Lemma~\ref{lmReducedSHExactCase}.

By the $\hbar$-gap of Lemma~\ref{lmTrichotomy}, every contribution surviving
modulo $T^\hbar$ has relative order zero. The norm-$\leq1$, order-zero
contraction of Lemma~\ref{lmExteriorOrderZeroRetract} therefore descends
to the exterior summand of $\gr_{\hbar}SC^*_M(K,A)$, so the kernel of the
surjection admits a norm-nonincreasing contraction. Proposition~\ref{prpFixedModelLocality}
identifies the resulting local quotient with the intrinsic telescope and
identifies the corresponding normalized-monotone restriction maps, whence
the commutative diagram. The surjection is therefore a quasi-isomorphism.

In the virtual setting the same conclusions follow from the virtual clause of
Proposition~\ref{prpFixedModelLocality}.
\end{proof}

\begin{rem}
    At this point let us preempt a potential concern. The usual definition of the intrinsic Floer complex involves a sequence of Hamiltonians whose slope goes to infinity. In our construction we are forced to use Hamiltonians whose slope becomes small outside a small neighborhood of $D$. This introduces upper periodic orbits which appear to cancel the lower orbits, and give the wrong answer. As explained in \S\ref{subsec:viterbo-relative}, the canonical Viterbo-to-relative map
    \[
    \iota_D:SH^*_{\mathrm{Viterbo}}(D;R)
    \longrightarrow
    SH^*_{\widehat{D},\theta}(D;R)
    \]
    is an isomorphism by Lemma~\ref{lmRelativeViterboComparison}. Hence the S-shaped relative presentation used here computes the usual intrinsic symplectic cohomology.
\end{rem}

\subsection{Independence of the S-shaped model}\label{SubSecModelIndep}

\begin{lm}\label{lmFilteredMonotoneComparison}
Fix an S-shaped package and a compact set $K\subset D$.
Let $A$ and $A'$ be regular S-shaped acceleration data for $K$, and let
\[
\Phi:SC^*_M(K,A)\longrightarrow SC^*_M(K,A')
\]
be the telescope map induced by an S-shaped morphism
$A\to A'$.
Then the restriction

\[
SC^{*,\geq0}_{M,\theta}(K,A)
\longrightarrow
SC^{*,\geq0}_{M,\theta}(K,A')
\]
is a chain-homotopy equivalence.
\end{lm}

\begin{proof}
Write
\[
C=SC^*_M(K,A),
\qquad
C'=SC^*_M(K,A').
\]
In the action-normalized bases, Corollary~\ref{CYRelFilrPres} gives
\[
d=d_0+\delta,
\qquad
d'=d'_0+\delta',
\qquad
\Phi=\Phi_0+\Phi_+,
\]
where $d_0,d'_0,\Phi_0$ have relative order zero, whereas
\[
\val_{M,\theta}(\delta),
\val_{M,\theta}(\delta'),
\val_{M,\theta}(\Phi_+)\geq\hbar
\]
and
\[
|\delta|_\infty,\ |\delta'|_\infty,\ |\Phi_+|_\infty
<e^{-3\hbar}.
\]
We first show that
\[
\Phi_0:(C,d_0)\longrightarrow(C',d'_0)
\]
is a norm-nonincreasing, order-zero chain-homotopy equivalence.

By Lemma~\ref{lmExteriorOrderZeroRetract}, the exterior summands of
$(C,d_0)$ and $(C',d'_0)$ are contractible by norm-nonincreasing
homotopies, and the full order-zero complexes admit norm-$\leq1$
deformation retracts onto the corresponding quotients $I$ and $I'$.
Since $\Phi_0$ preserves the exterior subcomplex, it induces
\[
\overline{\Phi}_0:I\longrightarrow I'.
\]
By Proposition~\ref{prpFixedModelLocality}, the order-zero quotients are
identified isometrically with the corresponding intrinsic telescope
models. Cofinal interleaving of
the two intrinsic rays yields a norm-nonincreasing reverse comparison
whose composites are continuation-homotopic to tail shifts and hence to
the identity. Thus $\overline{\Phi}_0$ is a chain-homotopy equivalence
whose inverse and homotopies have relative order zero and norm at most
one.

The two bounded order-zero deformation retracts
\[
(C,d_0)\simeq(I,d_I),
\qquad
(C',d'_0)\simeq(I',d_{I'})
\]
together with the bounded order-zero chain-homotopy equivalence
$\overline{\Phi}_0:I\simeq I'$ show that $\Phi_0$ is a bounded
order-zero chain-homotopy equivalence. With the maximum norm on the
cone, the standard cone contraction built from these maps and
homotopies has relative order zero and norm at most one; write $h_0$
for this contraction of $\operatorname{Cone}(\Phi_0)$, so that
$|h_0|_\infty\leq1$.

Consider now $\operatorname{Cone}(\Phi)$ with the maximum norm and the induced
relative valuation.  Its differential has the form
\[
D=D_0+\Delta,
\qquad
D_0=d_{\operatorname{Cone}(\Phi_0)},
\]
where
\[
\val_{M,\theta}(\Delta)\geq\hbar,
\qquad
|\Delta|_\infty<e^{-3\hbar}.
\]
In particular $|\Delta h_0|_\infty<1$, so
\[
h
=
h_0(\id+\Delta h_0)^{-1}
\]
is a contracting homotopy for $(\operatorname{Cone}(\Phi),D)$.
Each finite partial sum of the Neumann series for
$(\id+\Delta h_0)^{-1}$ preserves $F^{\val_{M,\theta}\geq a}$, since
$h_0$ has relative order zero and $\Delta$ has positive relative order.
By Lemma~\ref{lmRelFiltrationClosed} the filtration submodule is
ambient-norm closed, so $h$ preserves it as well. Hence the cone of
the restriction of $\Phi$ to $F^{\val_{M,\theta}\geq a}$ is
contractible, and this restriction is a chain-homotopy equivalence for
every $a\in\bR$. Taking $a=0$ gives the assertion for the relative
Novikov lattices.
\end{proof}

\begin{cy}\label{lmSameProfileFiltind}
If $A$ and $A'$ are two regular S-shaped acceleration data for the same
$(D,\theta,K)$, possibly associated to different S-shaped packages
(and, in the noncompact case, with the same asymptotic coefficient),
then the corresponding telescopes are connected by a finite zigzag of
filtration-preserving chain-homotopy equivalences. In particular, their
relative Novikov lattices are chain-homotopy equivalent over
$\Lambda_{\geq0}$.

For an inclusion $K_1\subset K_2$, normalized-monotone restriction maps
\[
SC^*_M(K_2,A_2)\longrightarrow SC^*_M(K_1,A_1)
\]
preserve $\val_{M,\theta}$ and hence preserve the relative Novikov lattices.
\end{cy}

\begin{proof}
For a same-package comparison, Proposition~\ref{prpSShapedAcceleration}
and Lemma~\ref{lmFilteredMonotoneComparison} give a zigzag of
filtration-preserving chain-homotopy equivalences.

For different packages, Lemma~\ref{lmSShapedFlexibility} supplies a
bridge; the same-package case then applies along the resulting zigzag.

Restriction maps preserve the relative Novikov lattices by
Corollary~\ref{CYRelFilrPres}.
\end{proof}

By Corollary~\ref{lmSameProfileFiltind}, the quasi-isomorphism type of
the relative Novikov lattice over $\Lambda_{\geq0}$ is independent of
the regular S-shaped acceleration datum. We therefore suppress $A$ from
the notation and write
\[
SC^{*,\geq0}_{M,\theta}(K).
\]
In the geometrically finite noncompact case this notation is understood
relative to the fixed asymptotic coefficient at infinity.
(The lattice is closed in the ambient norm by
Lemma~\ref{lmRelFiltrationClosed} with $a=0$.)

\subsection{Weight doubling for the product}\label{SubSecWeightDoubling}
Let $\mathcal P=(B,\delta,\sigma,\eta_{\mathrm{end}})$ be an S-shaped
package in the sense of \S\ref{subsecDataDep}, and assume the doubled
package
\[
2\mathcal P=(B,\delta,2\sigma,2\eta_{\mathrm{end}})
\]
of the pair-of-pants setup still satisfies the parameter constraints of
Definition~\ref{dfSShaped}. Choose regular S-shaped acceleration data
$A,2A$ and pair-of-pants data as in
Subsection~\ref{subsecSShapedOperations}.  After passing, if
necessary, to compatible cofinal tails and relabelling, use the
slice-wise pair-of-pants maps
\[
B_i:
CF^*(A_i)\otimes CF^*(A_i)
\longrightarrow
CF^*((2A)_i).
\]
After passing to a cofinal tail of $A$, choose a sufficiently small
S-shaped Morse Hamiltonian in the package $\mathcal P$ and continue its
standard unit cycle to the first level of the tail.  Denote the
resulting cycle by $e_1$ and set
\[
e_{i+1}=\kappa_i(e_i).
\]
Thus every $e_i$ represents the unit and the sequence is strictly
compatible with continuation.  Define
\[
U_i(x):=B_i(x,e_i).
\]
The one-parameter comparison families for the $B_i$, together with
$e_{i+1}=\kappa_i(e_i)$, make the $U_i$ a morphism of $1$-rays.
The telescope construction therefore gives a chain map
\[
U:
SC_M^*(K,A)\longrightarrow SC_M^*(K,2A).
\]
\begin{lm}\label{lmWeightDoubling}
The map $U$ is a $\val_{M,\theta}$-filtered chain-homotopy equivalence.
In action-normalized bases one may write
\[
U=U_0+U_+,
\qquad
V=V_0+V_+,
\]
where $V$ is a filtered chain-homotopy inverse of $U$,
\[
|U_0|_\infty,\ |V_0|_\infty\leq1,
\]
and
\[
\val_{M,\theta}(U_+),\,
\val_{M,\theta}(V_+)\geq\hbar,
\qquad
|U_+|_\infty,\ |V_+|_\infty<e^{-3\hbar}.
\]
The inverse $V_0$ may be chosen to preserve the outside summand.
On homology, $U_*$ is multiplication by the unit in the
pair-of-pants product with output in the doubled acceleration ray, and
gives the standard weight-doubling identification.
\end{lm}

\begin{proof}
The continuation maps defining the $e_i$, the pair-of-pants maps
$B_i$, and their one-parameter comparison families preserve the
relative filtration.  Corollaries~\ref{CYRelFilrPres} and
\ref{CYRelFilrPres1}, together with the ultrametric inequality, give
\[
U=U_0+U_+,
\qquad
|U_0|_\infty\leq1,
\qquad
\val_{M,\theta}(U_+)\geq\hbar,
\qquad
|U_+|_\infty<e^{-3\hbar}.
\]

At order zero, the exterior subcomplexes of source and target are
contractible. After quotienting them out, fixed-model locality identifies
the two quotients with the corresponding intrinsic telescope models.
Under these identifications, $U_0$ is multiplication by the intrinsic unit
followed by the canonical comparison between the two intrinsic
acceleration rays. By \cite[Lemma~5.18 and Corollary~5.19]{tonkonog}, the
induced map on cohomology is the continuation comparison between the
doubled and original acceleration systems associated with multiplication
by the unit. The chain-level unit--continuation homotopy, together with
cofinal interleaving of these rays, shows that the induced map on the
intrinsic quotients is a chain-homotopy equivalence. The corresponding
contraction of $\operatorname{Cone}(U_0)$ is obtained from the
norm-nonincreasing reverse cofinal comparison maps, continuation maps,
and continuation homotopies already constructed in the proof of
Lemma~\ref{lmFilteredMonotoneComparison}, together with the order-zero
unit--continuation homotopy, all of norm at most one. Thus the cone
contraction has norm at most one, and consequently
\[
|V_0|_\infty\leq1.
\]
Lifting across the contractions of the exterior subcomplexes therefore
gives an order-zero chain-homotopy inverse $V_0$, which may be chosen to
preserve the exterior subcomplex.

The cone-perturbation argument of
Lemma~\ref{lmFilteredMonotoneComparison} now gives a filtered
chain-homotopy inverse
\[
V=V_0+V_+
\]
with
\[
|V_0|_\infty\leq1,
\qquad
\val_{M,\theta}(V_+)\geq\hbar,
\qquad
|V_+|_\infty<e^{-3\hbar}.
\]
In the virtual setting the filtration estimates follow from
Lemma~\ref{lmVirtualTrichotomy}; the remaining argument is algebraic.
\end{proof}

\subsection{Formal locality: Theorem A}
\label{SubSecProofFiltration}
\label{ssFormalLocality}

We now collect the conclusions of the preceding subsections.

\begin{thmA}[Formal locality]
\label{tmFormalLocality}

Let $D\subset M$ be a Liouville domain with local primitive $\theta$.

\begin{enumerate}

\item
\label{tmFormalLocalityReduction}
For every compact subset $K\subset D$ there is a closed relative Novikov
lattice
\[
SC^{*,\geq 0}_{M,\theta}(K)\subset SC^*_M(K).
\]
There is a constant
\[
\hbar=\hbar(M,D,\theta)>0,
\]
independent of $K$, and a quasi-isomorphism
\begin{equation}\label{eqLocalityDef}
SC^{*,\geq 0}_{M,\theta}(K)/
T^\hbar SC^{*,\geq 0}_{M,\theta}(K)
\longrightarrow
SC^*_{\widehat D,\theta}(K;R)
\otimes_R\Lambda_{[0,\hbar)}.
\end{equation}
Moreover, the natural map
\[
SC^{*,\geq 0}_{M,\theta}(K)
\otimes_{\Lambda_{\geq 0}}\Lambda
\longrightarrow
SC^*_M(K)
\]
has norm-dense image, and its completion in the ambient Floer norm
recovers $SC^*_M(K)$.

\item
\label{tmFormalLocalitySupports}
If $K_1\subset K_2\subset D$, the ambient restriction map preserves the
relative Novikov lattices,
\[
SC^{*,\geq 0}_{M,\theta}(K_2)
\longrightarrow
SC^{*,\geq 0}_{M,\theta}(K_1).
\]
Under the comparisons in \textup{(1)}, its reduction modulo $T^\hbar$
agrees, up to chain homotopy, with the intrinsic restriction map
\[
SC^*_{\widehat D,\theta}(K_2;R)
\longrightarrow
SC^*_{\widehat D,\theta}(K_1;R).
\]

\item
\label{tmFormalLocalityViterbo}
Let
\[
D_1\subset D_2\subset M
\]
be an exact inclusion of Liouville domains with respect to the same
local primitive $\theta$, and put
\[
\hbar
=
\min\{
\hbar(M,D_1,\theta),
\hbar(M,D_2,\theta)
\}.
\]
The ambient restriction map preserves the corresponding relative
Novikov lattices. Under the comparisons in \textup{(1)}, its reduction
modulo $T^\hbar$ fits into a homotopy-commutative square with the
intrinsic restriction map, and the latter induces the standard Viterbo
restriction
\[
SH^*(D_2;R)\longrightarrow SH^*(D_1;R)
\]
on homology.

\item
\label{tmFormalLocalityOperations}
The ambient pair-of-pants product and BV operator admit representatives
compatible with the relative Novikov filtration. Under the comparison in
\textup{(1)}, their reductions modulo $T^\hbar$ are the intrinsic
pair-of-pants product and BV operator.

\end{enumerate}

\end{thmA}

\begin{proof}
For part~\textup{(1)}, choose a regular S-shaped acceleration datum $A$.
Proposition~\ref{prpGrLocality} supplies the fixed-model reduction comparison.
Corollary~\ref{lmSameProfileFiltind} removes dependence on the chosen
acceleration model, and Lemma~\ref{lmRelativeLatticeCompletion} gives the
completed-localization statement.

For part~\textup{(2)}, compatibility of the reduction with normalized-monotone
restriction for chosen ordered S-shaped acceleration data is the
commutative diagram of Proposition~\ref{prpGrLocality}. Corollary~\ref{lmSameProfileFiltind}
removes dependence on the chosen acceleration models. Therefore the
model-independent reduction comparison is natural under restriction, up to
chain homotopy.

For part~\textup{(3)}, using Lemma~\ref{lmSShapedFlexibility}, choose a
regular acceleration datum $\mathcal A_1$ simultaneously S-shaped relative to
$D_1$ and $D_2$. The Hamiltonians
have constant small radial slope on each corresponding
cutting annulus, and hence determine compatible local acceleration data
on $\widehat D_1$ and $\widehat D_2$. Choose also an S-shaped acceleration
datum $\mathcal A_2$ for $(M,D_2)$, with corresponding local datum on
$\widehat D_2$, so that globally
\[
\mathcal A_2\leq \mathcal A_1.
\]
Use the same $\hbar$ throughout the resulting comparison.

Apply Proposition~\ref{prpGrLocality} to the pairs $(M,D_2)$, $(\widehat D_2,D_1)$,
and $(M,D_1)$. Together with the comparison of acceleration data, this
gives the required homotopy-commutative square between the reduced
ambient restriction and the corresponding intrinsic restriction in
$\widehat D_2$.

It remains to identify the latter with Viterbo restriction. Let
\[
\iota_{D_2}\colon
SH^*_{\mathrm{Viterbo}}(D_2;R)
\longrightarrow
SH^*_{\widehat D_2,\theta}(D_2;R)
\]
and
\[
\iota_{D_1}\colon
SH^*_{\mathrm{Viterbo}}(D_2;R)
\longrightarrow
SH^*_{\widehat D_2,\theta}(D_1;R)
\]
be the canonical maps of \S\ref{subsec:viterbo-relative}. By \cite[Lemma~4.22]{GromanVarolgunes2022},
\[
\operatorname{res}\circ\iota_{D_2}=\iota_{D_1},
\]
while Lemma~\ref{lmRelativeViterboComparison} identifies $\iota_{D_2}$ with an isomorphism.
Composing the resulting restriction map in $\widehat D_2$ with the
local-to-intrinsic comparison for $D_1$ gives the standard Viterbo
restriction
\[
SH^*(D_2;R)\longrightarrow SH^*(D_1;R).
\]
The simultaneous S-shaped data give the required chain
representatives and homotopy-commutative square, and
Propositions~\ref{lmWeightedWellDefined} and~\ref{lmUnweightedWellDefined} give independence of the chosen
representatives.

Finally, for part~\textup{(4)}, it suffices to exhibit filtered chain-level
representatives of the standard ambient operations. Fix a regular
S-shaped acceleration datum $\mathcal A$ for $K\subset D$, and write
\[
L_{\mathcal A}
:=
SC^{*,\geq0}_{M,\theta}(K,\mathcal A),
\qquad
\overline L_{\mathcal A}
:=
L_{\mathcal A}/T^\hbar L_{\mathcal A}.
\]
Use the diagonal S-shaped pair-of-pants construction of
Subsection~\ref{subsecRelativeNovikovLattice},
\[
\widetilde\mu_{\mathcal A}
\colon
D_{\mathcal A}
\longrightarrow
SC_M^*(K,2\mathcal A),
\]
with induced homology operation $\mu_{\mathcal A}$, together with the
filtered comparison
\[
V:SC_M^*(K,2\mathcal A)\longrightarrow SC_M^*(K,\mathcal A)
\]
of Lemma~\ref{lmWeightDoubling}.

The pair-of-pants construction with two inputs from the acceleration ray
$\mathcal A$ has output in the doubled ray $2\mathcal A$. Composing that
output with the comparison $V$ of Lemma~\ref{lmWeightDoubling} returns to
the $\mathcal A$-model. The tensor-telescope formalism of
\cite[\S3.3]{tonkonog} relates the diagonal pair-of-pants construction to
the product operation, and the product construction of that paper
identifies the resulting homology operation with the standard
pair-of-pants product on symplectic cohomology; compare
\cite[\S5]{tonkonog}. Thus $V_*\mu_{\mathcal A}$ represents the usual
ambient product on $SH_M^*(K)$.

The levelwise pair-of-pants maps, the coherence homotopies entering the
telescope construction, and the comparison $V$ all preserve
$\val_{M,\theta}$. Therefore this S-shaped representative of the
standard ambient product is compatible with the relative Novikov
filtration, and in particular is defined over $L_{\mathcal A}$.

Pass to the chain quotient $\overline L_{\mathcal A}$. Homology is taken
after this chain-level reduction; no identification
$H(L_{\mathcal A})/T^\hbar\simeq H(\overline L_{\mathcal A})$ is used.
Because the chosen representative is filtered, it descends to the
quotient.

It remains to identify the reduced order-zero product with the intrinsic
product. By Lemma~\ref{lmAlgebraFiltered}, the exterior generators form an
ideal for the order-zero pair-of-pants product.
Lemma~\ref{lmExteriorOrderZeroRetract} contracts the exterior order-zero
subcomplex by a norm-nonincreasing contraction, and
Proposition~\ref{prpFixedModelLocality} identifies the resulting local
quotient isometrically with the intrinsic S-shaped complex. Moreover
$V_0$ preserves the exterior subcomplex, while
Proposition~\ref{prpGrLocality} identifies the actual chain quotient
modulo $T^\hbar$ with the intrinsic fixed model after removing the
contractible exterior sector. Consequently, under the comparison of
\textup{(1)}, reduction of this filtered ambient representative is the
intrinsic pair-of-pants product.

The same filtration and chain-reduction argument applies to the BV
operator; by Lemma~\ref{lmAlgebraFiltered} an order-zero BV configuration
with local output is entirely local, so the reduced BV operation is the
intrinsic BV operator.

The chosen S-shaped representative has already been identified on
homology with the standard ambient product, and its reduction with the
standard intrinsic product; the resulting statement is therefore
independent of the auxiliary S-shaped model. In the virtual setting the
same argument uses Lemma~\ref{lmVirtualTrichotomy} and
Proposition~\ref{prpVirtualLocality}.
\end{proof}

\subsection{The locality spectral sequence}\label{SubSecLocalSS}

\begin{cy}\label{prp1stPageiso}
There is a natural isomorphism over $\Lambda_{[0,\hbar)}$
\begin{equation}\label{eqReducedIso}
H^*(\gr_{\hbar}SC^*_M(K))=SH^*_{\widehat{D},\theta}(K;R)\otimes \Lambda_{[0,\hbar)}.
\end{equation} 
Here naturality means the isomorphism commutes with restriction maps. 
\end{cy}
\begin{proof}
This is the homology-level form of Theorem~\ref{tmFormalLocality}\textup{(1)}--\textup{(2)}.
\end{proof}

Choose a regular S-shaped acceleration datum $A$ for $K$ and write
\[
L_A(K):=SC^{*,\geq0}_{M,\theta}(K,A)
\]
for the corresponding relative Novikov lattice. By
Corollary~\ref{lmSameProfileFiltind}, different choices of $A$ give
chain-homotopy equivalent lattices. We retain the
symbol $A$ below only to specify a chain-level presentation.

Choose a discretization quantum
$0<\hbar_{\mathrm{disc}}\le\hbar$, where $\hbar$ is the geometric reduction threshold from Theorem~\ref{tmFormalLocality}, and define the decreasing filtration of the lattice by
\begin{equation}\label{eqMFiltration}
F_A^p L_A(K)
=
\begin{cases}
L_A(K),&p\le0,\\
T^{p\hbar_{\mathrm{disc}}}L_A(K),&p\ge0.
\end{cases}
\end{equation}

The $T^{\hbar_{\mathrm{disc}}}$-adic filtration is exhaustive and separated on the lattice because $F_A^0=L_A(K)$ and
\[
\bigcap_{p\ge0}T^{p\hbar_{\mathrm{disc}}}L_A(K)=0.
\]

Let $E_r^{p,q}(A)$ be the resulting spectral sequence. We call it the
\emph{locality spectral sequence}. Its chain-level presentation and its
$E_0$ complex depend on $A$, while its spectral sequence from the $E_1$
page onward is independent of this choice.

\begin{prp}\label{prpE1page}
The page $E_1$ satisfies
\[
E_1^{p,q}
\cong
SH_{\widehat D,\theta}^{p+q}(K)
\otimes
\Lambda_{[p\hbar_{\mathrm{disc}},(p+1)\hbar_{\mathrm{disc}})},
\qquad p\ge0,
\]
and
\[
E_1^{p,q}=0,\qquad p<0.
\]
\end{prp}
\begin{proof}
This follows from Corollary~\ref{prp1stPageiso} and the definition of
the $T^{\hbar_{\mathrm{disc}}}$-adic filtration of the relative Novikov lattice.
\end{proof}

Let us recall some basic terminology around convergence of spectral
sequences. We use the terminology of
\cite[Convergence~5.2.11]{weibel}.
We say a spectral sequence $E_r$ over a ring $R$ \emph{converges weakly} to a filtered graded $R$-module $M^*$ if there is a filtration $F^pM^*$ on $M^*$ such that $E_{\infty}^{p,q}\cong F^pM^{p+q}/F^{p+1}M^{p+q}$ in the associated graded. We say that a spectral sequence \emph{abuts} to a filtered graded $R$-module $M^*$ if it converges weakly to $M^*$ and the induced filtration on homology is exhaustive and Hausdorff.

\begin{prp}\label{prpLocalSSModelIndep}
The locality spectral sequence is independent of the regular S-shaped
acceleration datum from the $E_1$ page onward. Under the identification
of Proposition~\ref{prpE1page}, the comparison isomorphisms preserve the
residual real valuation on the graded Novikov factor.

In the geometrically finite noncompact case this statement is understood
relative to the fixed asymptotic behavior at infinity.
\end{prp}

\begin{proof}
By Corollary~\ref{lmSameProfileFiltind}, any two regular S-shaped acceleration
models for the same $K$ are connected by a finite zigzag of
comparisons which restrict to the relative Novikov lattices and induce
filtered chain-homotopy equivalences for the $T^{\hbar_{\mathrm{disc}}}$-adic
filtrations. Hence they induce isomorphisms on the associated-graded complexes
and on the $E_1$ page and every subsequent page.

Under the identification
\[
E_1^{p,q}
\cong
SH^*_{\widehat D,\theta}(K)
\otimes
\Lambda_{[p\hbar_{\mathrm{disc}},(p+1)\hbar_{\mathrm{disc}})}
\qquad(p\ge0),
\]
the comparison map is the intrinsic acceleration-data comparison on the
first factor and the identity on the graded Novikov factor. It therefore
preserves the residual real valuation.
\end{proof}

\begin{cy}[Collapse under dense restriction]
    \label{cyCollapseDenseRestriction}
    Fix a discretization quantum
    $0<\hbar_{\mathrm{disc}}\leq\hbar$.
    Let
    \[
    K_1\subset K_2\subset D.
    \]
    Suppose that the locality spectral sequence for $K_2$ collapses at the
    $E_1$ page and that the intrinsic restriction map
    \[
    SH_{\widehat D,\theta}^*(K_2)
    \longrightarrow
    SH_{\widehat D,\theta}^*(K_1)
    \]
    has dense image. Then the locality spectral sequence for $K_1$ also
    collapses at the $E_1$ page.
    \end{cy}
    
    \begin{proof}
    By Corollary~\ref{prp1stPageiso} and
    Proposition~\ref{prpE1page}, restriction induces a natural morphism of
    the two locality spectral sequences whose $E_1$ map has dense image.
    
    Suppose that $d_r^{K_1}$ is the first nonzero differential for $K_1$.
    Since all earlier differentials vanish for both $K_1$ and $K_2$, the
    page-$r$ restriction still has dense image. Naturality gives
    \[
    d_r^{K_1}\circ\operatorname{res}
    =
    \operatorname{res}\circ d_r^{K_2}
    =
    0.
    \]
    The page differentials are continuous for the quotient
    non-Archimedean seminorms, since they are induced by bounded filtered
    chain data. Hence density implies $d_r^{K_1}=0$, a contradiction.
    \end{proof}
The locality spectral sequence need not by itself determine the completed ambient invariant:

\begin{ex}\label{exCP1}
Let $M=CP^1$ and let $D\subset CP^1$ be a disc containing more than half the area. Then the complement of $D$ is displaceable, so $SH_{M}(D)=QH(M)\neq 0$, while $SH_{\widehat{D},\theta}(D)=0$. In particular, the locality spectral sequence vanishes identically from the first page onward.
\end{ex}

It is instructive to analyze how the pathology in this example arises. For this we construct a model for the underlying chain complex.  Normalize $\int_{CP^1}\omega=1$. Let
$0<r<1$ be the area of the disc. Let $x$ be of even degree and $\partial_x$ be the even and odd degree Morse Bott generators associated with the simple orbit associated with $\partial D$. Let $C^0$
and $C^1$ be the Banach completions, with respect to the Gauss norm, of
\[
\bigoplus_{i\geq0}\Lambda x^i
\qquad\text{and}\qquad
\bigoplus_{i\geq0}\Lambda x^i\partial_x,
\]
respectively, normalized so that
$|x^i|=|x^i\partial_x|=1$. Define
\[
d(x^i)=0,
\qquad
d(x^i\partial_x)=T^r x^i+T^{1-r}x^{i+2}.
\]
Write $SC^*=C^0\oplus C^1$. This can be shown to be a model for $SC^*_{CP^1}(D)$. 

Let $I(x^i\partial_x)=x^i$ and
$S(x^i\partial_x)=x^{i+2}$. Then
\[
d=T^r\bigl(I+T^{1-2r}S\bigr).
\]
If $r<1/2$, then
\[
\bigl\|T^{1-2r}S\bigr\|=e^{2r-1}<1,
\]
so the Neumann series gives a bounded inverse to $d$. Equivalently,
if $d=d_0+d_1$ denotes the two displayed summands, then
\[
\|d_0\|=e^{-r},
\qquad
\|d_1\|=e^{r-1},
\qquad
\|d_0^{-1}d_1\|=e^{2r-1}<1.
\]
Thus $H^*(SC)=0$ for $r<1/2$.

Suppose now that $r\geq1/2$ and put
\[
\eta=T^{2r-1},
\qquad |\eta|\leq1.
\]
For $\varepsilon\in\{0,1\}$ define
\[
\ell_\varepsilon\left(\sum_{i\geq0}a_i x^i\right)
=
\sum_{k\geq0}(-\eta)^k a_{2k+\varepsilon}.
\]
The series converges because $a_i\to0$ in the Gauss completion. These
functionals vanish on $\im d$. A direct recurrence shows that
\[
\ker(\ell_0,\ell_1)=\im d,
\]
while $d:C^1\to C^0$ is injective. Consequently,
\[
H^1(SC)=0,
\qquad
H^0(SC)
\cong
C^0/\im d
\cong
\Lambda\cdot[1]\oplus\Lambda\cdot[x],
\qquad r\geq1/2.
\]

The relative Novikov lattice is the closed
$\Lambda_{\geq0}$-span of
\[
y^i=T^{-ir}x^i,
\qquad
z^i=T^{-(i+1)r}x^i\partial_x,
\]
and its differential is
\begin{equation}\label{eqDiffCP1Rel}
dz^i=y^i+Ty^{i+2}.
\end{equation}
For $r\geq1/2$, the nonzero classes therefore satisfy
\[
[y^i]=-T[y^{i+2}]
     =T^2[y^{i+4}]
     =\cdots .
\]
Every such class lies in every positive
$T^{\hbar_{\mathrm{disc}}}$-adic filtration level. Thus its induced
homology-filtration valuation is $+\infty$, and the induced filtration
on homology is not Hausdorff.

\begin{comment}
\begin{rem}\label{remCollapseDenseRestriction}
    Suppose the locality spectral sequence collapses on the first page for $D$, and
    let $K\subset D$ be such that the intrinsic restriction
    \[
    SH_\theta^*(D)\longrightarrow SH_\theta^*(K)
    \]
    has dense image.  Then the locality spectral sequence also collapses on the first
    page for $K$.  Indeed, if $d_r^K$ were the first nonzero differential, then all
    earlier differentials vanish, so the page-$r$ restriction still has dense image.
    Naturality gives
    \[
    d_r^K\circ \mathrm{res}=\mathrm{res}\circ d_r^D=0.
    \]
    Page differentials carry the quotient non-Archimedean seminorm and are continuous because they are induced by bounded filtered chain data. Since $d_r^K$ is continuous, density implies $d_r^K=0$, a contradiction.
\end{rem}

\end{comment}

\section{Boundary depth}\label{SecBoundaryDepth}

\subsection{Boundary depth and filtered homology}

Let $(C^*,d,|\cdot|)$ be a complete non-Archimedean normed chain complex satisfying
\[
|dx|\leq |x|
\]
for all $x\in C^*$. (In particular $C^*_{>0}$ below is a subcomplex.) The \emph{boundary depth} $\beta(C^*,d,|\cdot|)$ is defined by
\[
\beta=\begin{cases}\sup_{x\neq 0\in \im d}\inf\{\val(x)-\val(y):d y=x\}&d\neq 0\\0&d=0.
\end{cases}
\]
Here $\val(x)=-\ln |x|$.
An equivalent characterization of boundary depth comes from considering the exact sequence
\[
0 \to Z^* \to C^* \to Q^* \to 0
\]
where $Z^*=\ker\partial$ and $Q^*=C^*/Z^*$. The complex $Q^*$ carries an induced norm and an induced map $\overline{d}:Q^*\to B^*=\im d$ which is a bijection. Then
\[
e^\beta=\left|\overline{d}^{-1}\right|_{\infty}:=\sup_{x\in B^*}\frac{|\overline{d}^{-1}x|}{|x|},
\]
with the convention that for the unique operator on the zero space the latter is $1$.

\subsubsection{Homological characterization}
We now show that boundary depth can be characterized homologically. In
particular, it is invariant under  filtrationwise quasi-isomorphisms as
defined below.

Given a normed chain complex $(C^*,d,|\cdot|)$ over a ground field $\bK$, we can form the completed tensor product $C^*\hat{\otimes}\Lambda$ with the Novikov field $\Lambda$. This gives a chain complex over $\Lambda$ with a valuation induced by the norm.

Let now $(C^*,d,|\cdot|)$ be a complete valued chain complex over the
Novikov field with $|dx|\leq |x|$. For $a\in\bR$, write
\[
C^*_{>a}:=\{x\in C^*:\val(x)>a\}.
\]
Each $C^*_{>a}$ is a chain complex over the nonnegative valuation ring
$\Lambda_{\geq0}$.

\begin{df}\label{dfFiltrationwiseQI}
A filtration-preserving $\Lambda$-linear chain map
$f:C^*\to D^*$ is a \emph{filtrationwise quasi-isomorphism} if
\[
f:C^*_{>a}\longrightarrow D^*_{>a}
\]
is a quasi-isomorphism over $\Lambda_{\geq0}$ for every $a\in\bR$.
Since multiplication by $T^a$ identifies $C^*_{>0}$ with $C^*_{>a}$,
and $f$ is $\Lambda$-linear, it is enough to check this condition at
$a=0$. Throughout this paper, the phrase \emph{filtered
quasi-isomorphism} means a filtrationwise quasi-isomorphism in this
sense.
\end{df}

Denote by $H(C^*_{>0})^\tau$ the torsion part of $H(C^*_{>0})$ as a
$\Lambda_{\geq0}$-module. For each $x\neq 0\in H(C^*_{>0})^\tau$ let
\[
\tau(x):=\inf\{ r\geq 0:T^rx=0\}=\sup\{r\geq 0:T^rx\neq 0\}.
\]

The following characterization also makes the behavior under completed
scalar extension transparent.

\begin{lm}\label{lmBoundaryDepthTorsion}
We have 
\[
\beta(C^*)=\sup_{x\in H(C^*_{>0})^\tau}\tau(x),
\]
with the convention that the supremum is $0$ if the torsion submodule vanishes.
\end{lm}

\begin{proof}
For a nonzero boundary $b$ put
\[
s(b):=\sup\{\val(y):dy=b\},
\]
so that $\beta(b):=\val(b)-s(b)$. If $b\in C_{>0}$, its class in $H(C_{>0})$ is torsion precisely because $b$ becomes a boundary over the Novikov field, and
\[
\tau([b])=\max\{0,-s(b)\}=\max\{0,\beta(b)-\val(b)\}.
\]
Conversely every torsion class of $H(C_{>0})$ is represented by such a $b$: if $T^r[b]=0$ for some $r\geq 0$ then $b$ is already a boundary over the Novikov field. Multiplying a boundary by an appropriate power of $T$ does not change $\beta(b)$ and makes its valuation positive and arbitrarily close to $0$. Taking suprema therefore yields
\[
\sup_b\beta(b)=\sup_{x\in H(C_{>0})^\tau}\tau(x).
\]
\end{proof}

We will also use the following immediate consequence of the proof. Let
$(C_0,d)$ be a complete normed complex over the ground field carrying the
countable orthogonal basis/Gauss-completion structure occurring in the
intrinsic Floer telescope, so that every element of
\[
C:=C_0\widehat\otimes \Lambda
\]
admits a convergent coefficientwise Novikov expansion. Equip $C$ with the
induced Gauss norm. Then
\[
\beta(C)=\beta(C_0).
\]
Indeed, write a boundary in $C$ coefficientwise as
$b=\sum_\lambda T^\lambda b_\lambda$. Since the differential is defined
over the ground field, each $b_\lambda$ is a boundary in $C_0$. Given
$\epsilon>0$, choose primitives $y_\lambda$ with
\[
dy_\lambda=b_\lambda,\qquad
\val(y_\lambda)\geq
\val(b_\lambda)-\beta(C_0)-\epsilon .
\]
Then $\sum_\lambda T^\lambda y_\lambda$ converges and is a primitive of
$b$ with the same bound, so
$\beta(C)\leq\beta(C_0)$.

Conversely, if $b\in C_0$ is a boundary and
$y=\sum_\lambda T^\lambda y_\lambda\in C$ satisfies $dy=b$, then the
$T^0$-coefficient satisfies $dy_0=b$ and
$\val(y_0)\geq\val(y)$. Thus extension of scalars cannot improve the
optimal primitive of a ground-field boundary. Taking the supremum over
such boundaries gives the reverse inequality.

\begin{cy}\label{cyBoundaryDepthInv}
Boundary depth is preserved under filtrationwise quasi-isomorphism.
\end{cy}

\begin{proof}
Let $f:C^*\to D^*$ be a filtrationwise quasi-isomorphism. Then $f$ induces an
isomorphism of torsion $\Lambda_{\geq0}$-modules
$H(C^*_{>0})^\tau\to H(D^*_{>0})^\tau$ which moreover identifies torsion
exponents. Therefore the suprema in Lemma~\ref{lmBoundaryDepthTorsion}
agree, and $\beta(C^*)=\beta(D^*)$.
\end{proof}

\begin{lm}\label{lmBoundaryDepthClosed}
If $(C^*,d)$ has finite boundary depth then $\im d$ is closed. In particular:
\begin{enumerate}
\item $H(C^*,d)$ is a Banach space;
\item a filtrationwise quasi-isomorphism induces an isometric isomorphism on homology over the Novikov field.
\end{enumerate}
\end{lm}

\begin{proof}
Since $d$ is continuous, $Z=\ker d$ is closed, so $Q=C/Z$ is Banach. Finite boundary depth means $\overline{d}^{-1}:\im d\to Q$ is bounded. If $b_i\in\im d$ converges in $C$, then $z_i=\overline{d}^{-1}(b_i)$ is Cauchy in $Q$, hence converges to some $z\in Q$. Continuity of $\overline{d}$ gives $b_i=\overline{d}(z_i)\to\overline{d}(z)$, so the limit lies in $\im d$. Thus $\im d$ is closed and $H(C)=\ker d/\im d$ is Banach.

For the isometry claim (which does not itself require finite boundary depth), note that the natural map $H(C_{>0})\to H(C)$ has torsion kernel, and after quotienting by torsion its image consists exactly of the homology classes of positive spectral valuation: a field-valued class has positive valuation precisely when it admits a cycle representative in $C_{>0}$. Thus $H(C_{>0})/H(C_{>0})^\tau$ identifies with the open unit lattice in $H(C)$. A filtrationwise quasi-isomorphism identifies these lattices. Since the valuation of an arbitrary field-valued homology class is determined by which powers $T^r$ place it in the positive lattice, the resulting field-valued homology isomorphism preserves the spectral valuation, hence the norm. Finite boundary depth ensures moreover that the quotient seminorm on $H(C)$ is a genuine Banach norm.
\end{proof}

\subsection{Boundary depth of Liouville domains}\label{subsecBoundaryDepthLiouville}
For a Liouville domain $D\subset M$ we have defined the boundary depth $\beta(D)$ as 
\begin{equation}
    \beta(D):=\beta\left(SC^*_{\widehat{D},\theta}(D;\bQ)\right).
\end{equation}
where $SC^*_{\widehat{D},\theta}(D;\bQ)$ is the completed telescope model associated with some acceleration datum $A$. We list the following properties of $\beta(D)$:
\begin{lm}\label{lmBoundaryDepthLiouville}

    \begin{enumerate}
     \item $\beta(D)$ is independent of the choice of acceleration datum $A$.
     \item $\beta$ is independent of the choice of primitive.
     \item $\beta(t\cdot D)=t\beta(D)$.
     \item If
     \[
     H(C_{>0})
     \]
     is torsion-free as a $\Lambda_{\geq0}$-module, where
     \[
     C=SC^*_{\widehat{D},\theta}(D;\bQ)\widehat\otimes\Lambda,
     \]
     then $\beta(D)=0$.
    \end{enumerate}
\end{lm}
\begin{proof}
By the scalar-extension observation after
Lemma~\ref{lmBoundaryDepthTorsion}, $\beta(D)$ may be computed after
completed extension of scalars to $\Lambda$. We prove each property in
turn.

    \begin{enumerate}
        \item Independence of acceleration datum: let $A$ and $A'$ be two
        acceleration data. Then the models are
        filtrationwise quasi-isomorphic. Corollary~\ref{cyBoundaryDepthInv}
        thus gives equality of their boundary depths.

        \item Independence of primitive: by Lemma~\ref{lmUnweightedSHcomparison}
        and the action-norm convention used there, the comparison
        \[
        SC^*_{\widehat{D},\theta}(D;\bQ)\hat{\otimes}\Lambda
        \xrightarrow{\simeq}
        SC^*_{\widehat{D}}(D;\Lambda)
        \]
        is an isometry, while the weighted complex is defined without reference to the primitive. Hence $\beta$ is primitive-independent.

        \item Scaling property: under the canonical Liouville rescaling identification the action valuation is multiplied by $t$ (equivalently, Novikov exponents are rescaled by $\lambda\mapsto t\lambda$), so boundary-depth differences scale by $t$: $\beta(t\cdot D)=t\,\beta(D)$.
        
        \item Torsion-free case: if $H(C_{>0})$ is torsion-free as a
        $\Lambda_{\geq0}$-module, then by Lemma~\ref{lmBoundaryDepthTorsion}
        we have $\beta(C)=0$. Combined with the scalar-extension identity
        $\beta(C)=\beta(D)$, this yields $\beta(D)=0$.
    \end{enumerate}
\end{proof}

\begin{lm}\label{lmSkeletonBoundaryDepth}
Let $K=\operatorname{Skel}(D)\subset\widehat D$. Then
\[
\beta(K):=
\beta\left(SC^*_{\widehat D,\theta}(K;\bQ)\right)=0.
\]
\end{lm}

\begin{proof}
Write $r\in(0,\infty)$ for the exponentiated Liouville coordinate on
$\widehat D\setminus K$, normalized by $r=1$ on $\partial D$. Fix
$\epsilon>0$. Choose slopes $\lambda_i\to\infty$, outside the Reeb
period spectrum, and numbers $\rho_i\to0$ such that
$\lambda_i\rho_i<\epsilon$.

Using the standard flexibility in the choice of relative acceleration
data, choose a monotone cofinal sequence $H_i$ for $K$ which is
negative and $C^2$-small near $K$, is radial on
$\widehat D\setminus K$, and agrees with
\[
h_i(r)=\lambda_i r-c_i
\]
for $r\geq\rho_i$, where $|c_i|<2\epsilon$. The constant-to-linear
transition can be chosen so that
\[
|rh_i'(r)-h_i(r)|<3\epsilon.
\]
After an arbitrarily small regular perturbation, every orbit generator
therefore has action in $(-4\epsilon,4\epsilon)$, since the action of a
radial orbit is
\[
\cA_{H_i}(\gamma)=rh_i'(r)-h_i(r).
\]

The ground field $\bQ$ has the trivial valuation. Since the norms of
all telescope generators are consequently bounded above and bounded
away from zero, the action-norm completion introduces no infinite
support in this model, and every nonzero chain $x$ satisfies
\[
e^{-4\epsilon}\leq |x|\leq e^{4\epsilon}.
\]
Hence, for every nonzero boundary $b$ and every $y$ with $dy=b$,
\[
\log\frac{|y|}{|b|}\leq8\epsilon.
\]
Thus the boundary depth of this acceleration model is at most
$8\epsilon$. It is independent of the acceleration datum by
Lemma~\ref{lmBoundaryDepthLiouville}. Since $\epsilon>0$ was arbitrary,
$\beta(K)=0$.
\end{proof}

\begin{cy}\label{cyConvexTorusBetaZero}
Let
\[
K=T^n\times P\subset T^*T^n,
\]
with $P\subset\bR^n$ compact and convex, be a standard convex cotangent-torus
domain. Then
\[
\beta(K)=0.
\]
\end{cy}

\begin{proof}
In the standard Viterbo/Morse--Bott computation the complex is governed by the
topology of the corresponding loop-space components, and the relevant
positive-lattice complex has no torsion (equivalently, in the standard model no
differential producing finite boundary depth appears); cf.~\cite[\S1.1.1]{GromanVarolgunes2021}.
Lemma~\ref{lmBoundaryDepthLiouville}(4) therefore gives $\beta(K)=0$.
\end{proof}

\section{Homological perturbation}\label{subsecFloerType}
\subsection{Floer type complexes}
The results of this section require working with a field $\bK$. Let $V^*$ be a graded non-Archimedean Banach space over the Novikov field $\Lambda=\Lambda_{\bK}$. We assume $V$ admits a countable orthogonal basis $E$ over the Novikov field. That is, $E$ consists of a sequence of elements $e_1, e_2, \dots,$ such that any element $x\in V$ has an expansion as a convergent sum $x=\sum c_ie_i$ with $c_i\in\Lambda$ and such that 
\begin{equation}
|x|=\sup_i|c_ie_i|.
\end{equation}

Crucially, \emph{we do not assume $|e_i|=1$}. In particular, while $V$ is a Banach space with respect to the norm on $V$, it is \emph{not assumed to be complete with respect to the $T$-adic filtration $\val_E$ induced by the basis $\{e_i\}$}.

Let $G^*$ be the degreewise complete graded $\bK$ module generated over $\bK$ by the basis elements of $E$. In each degree $G^*$ is a Banach space over $\bK$ with the trivial valuation. Note we have an isomorphism 
$$
V^*=G^*\hat{\otimes}\Lambda
$$
which is natural with respect to filtration preserving morphisms.

Let $d:V^*\to V^{*+1}$ be a differential which preserves the valuation $\val_E$. We assume that:
\begin{enumerate}
\item $|dx|\leq |x|$ for all $x\in V$, and,
\item there is a real number $\hbar>0$ such that $d=d_0\otimes \id+ T^{\hbar}d_1$ where $d_0:G^*\to G^{*+1}$ is a norm decreasing differential, and $d_1$ decreases norms and preserves the valuation $\val_E$. 
\end{enumerate}
 
We refer to $(V,d)$ as above as a \emph{Floer type complex}. We refer to $(G^*,d_0)$ as the \emph{reduced complex}. 

\begin{lm}\label{lmTelFloerType}
Let $D\subset M$ be a symplectically embedded Liouville domain with a primitive $\theta$. Let $K\subset D $ be a compact set and fix an $S$-shaped acceleration datum $A=(H_{\tau},J_{\tau},f)$ for $K\subset D\subset M$. Let $SC^*_M(K;A)$ be the associated telescope chain complex constructed from the sequence of Floer complexes $CF^*(H_i,J_i)$ with continuation maps. 

Then $SC^*_M(K;A)$, together with the lattice generated by basis elements 
$$T^{-\val_{M,\theta}(\gamma)}\gamma, \quad qT^{-\val_{M,\theta}(\gamma)}\gamma$$
for the periodic orbits $\gamma$ of the S-shaped Hamiltonians $H_i$, conforms to the definition of a Floer type complex provided we pick a primitive $\theta$ of Liouville type. Here $\val_{M,\theta}$ is the valuation defined in \eqref{eqRelFiltration}.
\end{lm}

\begin{proof}
Floer and continuation contributions are norm non-increasing by the
topological/geometric energy identity, while the algebraic identity terms
in the telescope preserve norm. Hence the telescope differential is norm
non-increasing. The decomposition by relative valuation follows from
Corollary~\ref{CYRelFilrPres} and the telescope construction.
\end{proof}
\subsection{homological perturbation}\label{subsecHomologicalPerturbation}
Our aim is to relate the chain complex $(V^*,d)$ and the chain complex $(G^*,d_0)$. We recall the notion of a special deformation retract. 

\begin{df}
    A map $p:L^*\to M^*$ of chain complexes is called a \emph{special deformation retraction} if there exists a chain map $i:M^*\to L^*$ and a homotopy $h$ satisfying
    \[
    \id_L-ip=d_Lh+hd_L,
    \]
    and the following hold:
    \begin{enumerate}
        \item $p\circ i=\id$,
        \item $h^2=0$, $h\circ i=0$, and $p\circ h=0.$
    \end{enumerate}
    The data $i,p,h$ is called \emph{special D.R. data}.  
\end{df}

We first formulate a general claim concerning the countably generated chain complex $(G^*,d_0)$.
\begin{lm}\label{lmSpecialDR}
     Suppose $(G^*,d_0)$ has finite boundary depth $\beta$. There exists a special deformation retraction $p:(G^*,d_0)\to (H^*(G^*,d_0),0)$ which induces the identity map on homology. Moreover for any $\epsilon>0$ we can choose the data $i,p,h$ so that
     \begin{enumerate}
        \item $i,p$ increase norms by at most $e^{\epsilon}$, and
        \item $h$ increases norms by at most $e^{\beta+\epsilon}$.
     \end{enumerate}  
\end{lm}

Before proving this lemma, we need a few notions and lemmas from non-Archimedean analysis. The degreewise countably generated Banach spaces occurring here are of countable type in the terminology of \cite{PerezGarciaSchikhof}.

\begin{df}
If $V=V_1\oplus\cdots\oplus V_m$ is a finite topological direct sum, call the decomposition \emph{$r$-orthogonal}, for $r\in(0,1)$, if
\[
|v|\geq r\max_j|\pi_j(v)|
\]
for every $v\in V$. For $m=2$, say that $V_2$ is an \emph{$r$-orthogonal complement} of $V_1$. In particular $|\pi_j|\leq r^{-1}$.
\end{df}

\begin{lm}\label{lmDecomposition}
Let $(G^*,d)$ be a Banach chain complex of countable type and finite boundary depth. For every $r\in(0,1)$ there is a graded $r$-orthogonal decomposition
\begin{equation}
G=D\oplus B\oplus C
\end{equation}
(chosen degreewise with the same $r$ in every degree) with
\[
B\oplus C=\ker d,\qquad B=\im d=d(D),
\]
and
\[
\bigl|(d|_D)^{-1}:B\to D\bigr|\leq r^{-1}e^\beta.
\]
\end{lm}

\begin{proof}
Set $Z=\ker d$ and $B=\im d$. By finite boundary depth and Section~\ref{SecBoundaryDepth}, $B$ is closed; also $Z$ is closed. Hence $Z$ is itself a Banach space of countable type \cite[Corollary~2.3.14]{PerezGarciaSchikhof}.

Fix $s=\sqrt{r}$. By \cite[Theorem~2.3.13]{PerezGarciaSchikhof}, choose a complement $D$ of $Z$ in $G$ such that $G=D\oplus Z$ is $s$-orthogonal. Similarly choose a complement $C$ of $B$ in $Z$ such that $Z=B\oplus C$ is $s$-orthogonal. Then for $x=x_D+x_B+x_C$ the two inequalities imply
\[
|x|\geq r\max\{|x_D|,|x_B|,|x_C|\},
\]
after decreasing the chosen constants arbitrarily slightly if necessary. Thus $G=D\oplus B\oplus C$ is $r$-orthogonal.

Since $D\cap\ker d=0$, the restriction $d|_D:D\to B$ is injective; it is surjective because every $dy\in B$ equals $d(p_Dy)$. For the inverse estimate, if $x\in B$, boundary depth gives, for every $\eta>0$, a $y\in G$ satisfying $dy=x$ and $|y|<e^{\beta+\eta}|x|$. Then $(d|_D)^{-1}x=p_Dy$, and $r$-orthogonality gives $|p_Dy|\leq r^{-1}|y|$. Letting $\eta\to0$ yields $|(d|_D)^{-1}|\leq r^{-1}e^\beta$.
\end{proof}

\begin{rem}
    It is helpful to consider this Lemma in the case of the completed telescope model $tel\{H_i\}$  for a sequence of convex Hamiltonians $H_i$ under the assumption that the differential on $CF^*(H_i)$ vanishes for each $i$. That is, the differential in the telescope encodes only the continuation maps. In this case, a good choice for $D$ is the subspace $qtel\{H_i\}$.  $B$ is of course the image $d(D)$. $C$ can be obtained by picking for each Reeb orbit $\gamma$ its incarnation in one of the $H_i$. For $r$ to be $\epsilon$ close to $1$, we need to pick the incarnation of each orbit so that the action is $\log (1+\epsilon)$ close to the period of the corresponding Reeb orbit.\end{rem}

\begin{proof}[Proof of Lemma \ref{lmSpecialDR}]
    Fix any $\epsilon>0$ and let $r>e^{-\epsilon/2}$. We consider an $r$-orthogonal decomposition $G^*=D\oplus B\oplus C$ as in Lemma \ref{lmDecomposition}. We have an isomorphism $j: C^*\to H^*(G^*,d_0)$ taking a cycle to its cohomology class. We then define $p=j\circ p_C$ and $i=j^{-1}$. By definition $p$ induces the identity map on homology. To define $h$, let $(d_0|_D)^{-1}:B\to D$ be the inverse of $d_0|_D$. We define $h=(d_0|_D)^{-1}\circ p_B$. 
    Then 
    \begin{itemize}
    \item $i\circ p =p_C$, so $\id-i\circ p=p_D+p_B = h\circ d_0+d_0\circ h$;
    \item $h\circ i = (d_0|_D)^{-1}\circ p_B\circ j^{-1} = 0$ since $p_B\circ j^{-1} = 0$ (as $j^{-1}$ maps to $C$ and $p_B$ projects to $B$); 
    \item $p\circ h = j\circ p_C\circ (d_0|_D)^{-1}\circ p_B = 0$ since $p_C\circ (d_0|_D)^{-1} = 0$ (as $(d_0|_D)^{-1}$ maps to $D$ and $p_C$ projects to $C$); and,
    \item  $h^2 = (d_0|_D)^{-1}\circ p_B\circ (d_0|_D)^{-1}\circ p_B = 0$ since $p_B\circ (d_0|_D)^{-1} = 0$ (as $(d_0|_D)^{-1}$ maps to $D$ and $p_B$ projects to $B$).
    \end{itemize}
    For the norm estimates, $r$-orthogonality gives $|p_D|,|p_B|,|p_C|\leq r^{-1}$. The quotient map $j:C\to H(G,d_0)$ is norm non-increasing, while $r$-orthogonality gives $|j(c)|=\inf_{b\in B}|c+b|\geq r|c|$, so $|j^{-1}|\leq r^{-1}$. Therefore
    \[
    |p|\leq r^{-1}<e^{\epsilon},\qquad |i|\leq r^{-1}<e^{\epsilon}.
    \]
    Lemma~\ref{lmDecomposition} gives $|(d_0|_D)^{-1}|\leq r^{-1}e^\beta$, and hence
    \[
    |h|\leq \bigl|(d_0|_D)^{-1}\bigr|\,|p_B|\leq r^{-2}e^\beta<e^{\beta+\epsilon}.
    \]
\end{proof}

We now study the perturbed complex
\[
(V^*,d)=(G^*\hat{\otimes}\Lambda,d_0+\delta),
\qquad
\delta=T^\hbar d_1,
\]
and put
\[
W:=H^*(G,d_0)\hat{\otimes}\Lambda,
\qquad
\beta:=\beta(G,d_0).
\]
Suppose
\begin{equation}\label{eqPerturbationBoundLemma}
    |\delta|_\infty<e^{-\beta}.
\end{equation}
Choose $\epsilon>0$ so that
\[
e^{6\epsilon+\beta}<|\delta|_\infty^{-1},
\]
and fix special D.R. data $i,p,h$ as in Lemma \ref{lmSpecialDR}. Thus
\[
\id-ip=d_0h+hd_0,
\qquad
pi=\id,
\qquad
h^2=0,\quad hi=0,\quad ph=0,
\]
and
\[
|p|_\infty,\ |i|_\infty<e^\epsilon,
\qquad
|h|_\infty<e^{\beta+\epsilon}.
\]
In particular,
\[
|\delta h|_\infty
<
|\delta|_\infty e^{\beta+\epsilon}
<
e^{-5\epsilon}
<
1,
\]
so $\id+\delta h$ is invertible. Define
\[
A:=(\id+\delta h)^{-1}\delta
\]
and set
\begin{equation}
\label{eqPerturbedDRData}
\widetilde{\delta}:=pAi,
\qquad
i_1:=i-hAi,
\qquad
p_1:=p-pAh,
\qquad
h_1:=h-hAh.
\end{equation}

\begin{lm}\label{lmCompPert}
With the notation above, $\widetilde{\delta}$ is a differential on $W$, and
\[
p_1:(V^*,d)\longrightarrow (W,\widetilde{\delta}),
\]
together with $i_1$ and $h_1$, is a special deformation retraction. In
particular,
\[
p_1i_1=\id,
\qquad
\id-i_1p_1=dh_1+h_1d,
\qquad
h_1^2=0,\quad h_1i_1=0,\quad p_1h_1=0.
\]
Moreover, with respect to $\val_E$ on $V$ and the $T$-adic filtration
on $W$,
\[
\val_E(\widetilde{\delta})\geq\hbar,
\qquad
\val_E(i_1-i)\geq \hbar,
\qquad
\val_E(p_1-p)\geq \hbar,
\qquad
\val_E(h_1-h)\geq \hbar.
\]
In particular, $p_1$ is filtered with respect to $\val_E$ on the source and
the $T$-adic filtration on the target.
\end{lm}

\begin{proof}
The Homological Perturbation Lemma \cite{PertubationLemma}, applied with
homotopy $-h$, gives
\[
\widetilde{\delta}^2=0,
\qquad
di_1=i_1\widetilde{\delta},
\qquad
p_1d=\widetilde{\delta}p_1,
\]
and the special deformation-retract identities stated above. Indeed, the
usual convention for the perturbation lemma is obtained from ours by replacing
$h$ with $-h$, since
\[
ip-\id=d_0(-h)+(-h)d_0,
\]
and then
\[
(\id-\delta(-h))^{-1}\delta
=
(\id+\delta h)^{-1}\delta=A.
\]
Thus the standard formulas become exactly those in
\eqref{eqPerturbedDRData}.

The non-Archimedean triangle inequality gives
\[
|A|_\infty\leq |\delta|_\infty.
\]
Consequently,
\begin{equation}
\label{eqHPLCorrectionEstimate}
|p_1-p|_\infty,\ |i_1-i|_\infty
<
e^{\beta+2\epsilon}|\delta|_\infty
<
e^{-4\epsilon}
<
1.
\end{equation}

We next check the filtration conclusions. The maps $i,p,h$ are defined over
the ground field and extended $\Lambda$-linearly, hence have order zero with
respect to $\val_E$ on $V$ and the $T$-adic filtration on $W$. Write
\[
A
=
(\id+\delta h)^{-1}\delta
=
\delta-\delta h\delta+\delta h\delta h\delta-\cdots.
\]
Every term of this Neumann series contains at least one occurrence of
$\delta$. Thus if $\delta$ has $\val_E$-order at least $\hbar>0$, then $A$ has
$\val_E$-order at least $\hbar$. From
\[
\widetilde{\delta}=pAi,
\qquad
i_1-i=-hAi,
\qquad
p_1-p=-pAh,
\qquad
h_1-h=-hAh,
\]
the four filtration inequalities of the statement follow at once.
\end{proof}

\begin{lm}[Isometry of the perturbed retraction]
\label{lmCompPertIsometry}
Under the hypotheses and notation of Lemma~\ref{lmCompPert}, the mutually
inverse maps on homology induced by
\[
p_1:(V^*,d)\longrightarrow (W,\widetilde{\delta})
\]
and by $i_1$ are isometries for the Banach norm.
\end{lm}

\begin{proof}
Let $c\in V$ be a $d$-cycle. Since
\[
d_0c=-\delta c,
\]
the identity $\id-ip=d_0h+hd_0$ gives
\[
c-ip(c)=d_0h(c)-h\delta(c),
\]
and hence
\begin{equation}
\label{eqIsometryFirstRepresentative}
ip(c)+d_0h(c)=c+h\delta(c).
\end{equation}
The left-hand side is a $d_0$-cycle representing $p(c)\in W$. Since the norm
on $H(G,d_0)$ is the quotient norm,
\[
|p(c)|
\leq
|c+h\delta(c)|.
\]
Moreover,
\[
|h\delta|_\infty
<
e^{\beta+\epsilon}|\delta|_\infty
<
1.
\]
Thus, for $c\neq 0$,
\[
|h\delta(c)|<|c|,
\]
and the strong non-Archimedean triangle inequality applied to
\eqref{eqIsometryFirstRepresentative} gives
\[
|p(c)|\leq |c|.
\]
Together with \eqref{eqHPLCorrectionEstimate},
\[
|p_1(c)|
\leq
\max\{|p(c)|,|(p_1-p)c|\}
\leq
|c|.
\]
Taking the infimum over all $d$-cycle representatives of a homology class gives
\begin{equation}
\label{eqP1NonIncreasingHomology}
|p_1(\alpha)|\leq |\alpha|
\qquad
\text{for every }\alpha\in H(V,d).
\end{equation}

For the reverse inequality, let $z\in W$ be a
$\widetilde{\delta}$-cycle. For every $\eta>0$, choose a $d_0$-cycle
$\widehat z\in V$ representing $z$ such that
\begin{equation}
\label{eqNearlyMinimalLift}
|\widehat z|<e^\eta |z|.
\end{equation}
Since $|z|$ is the quotient norm, we also have
\begin{equation}
\label{eqQuotientLowerBound}
|z|\leq |\widehat z|.
\end{equation}
Choose $0<\eta<\epsilon$.

Both $i(z)$ and $\widehat z$ are $d_0$-cycles representing the same
$d_0$-homology class. Put
\[
b:=i(z)-\widehat z.
\]
Then $d_0b=0$ and $p(b)=0$. Applying
$\id-ip=d_0h+hd_0$ to $b$ gives
\[
b=d_0h(b).
\]
Set
\[
y:=h(b).
\]
Then
\begin{equation}
\label{eqPrimitiveDifference}
d_0y=i(z)-\widehat z.
\end{equation}
Since
\[
|i(z)|<e^\epsilon |z|,
\qquad
|\widehat z|<e^\epsilon |z|,
\]
we have
\[
|b|<e^\epsilon |z|,
\]
and therefore
\begin{equation}
\label{eqYEstimate}
|y|
\leq
|h|_\infty |b|
<
e^{\beta+2\epsilon}|z|.
\end{equation}

Because $i_1$ is a chain map and $z$ is a
$\widetilde{\delta}$-cycle, $i_1(z)$ is a $d$-cycle. Subtracting the
$d$-boundary $dy$ and using \eqref{eqPrimitiveDifference},
\begin{align}
i_1(z)-dy
&=
i(z)+(i_1-i)z-d_0y-\delta y \notag\\
&=
\widehat z+(i_1-i)z-\delta y.
\label{eqCorrectedRepresentative}
\end{align}
The first error term satisfies
\[
|(i_1-i)z|
<
e^{-4\epsilon}|z|
<
|z|,
\]
while by \eqref{eqYEstimate},
\[
|\delta y|
<
|\delta|_\infty e^{\beta+2\epsilon}|z|
<
e^{-4\epsilon}|z|
<
|z|.
\]
By \eqref{eqQuotientLowerBound}, both error terms in
\eqref{eqCorrectedRepresentative} have norm strictly smaller than
$|\widehat z|$. Hence the strong non-Archimedean triangle inequality gives
\[
|i_1(z)-dy|=|\widehat z|.
\]
Since $i_1(z)-dy$ represents the same $d$-homology class as $i_1(z)$,
\[
|i_1([z])|
\leq
|\widehat z|
<
e^\eta |z|.
\]
Letting $\eta\to 0$ and then taking the infimum over
$\widetilde{\delta}$-cycle representatives gives
\begin{equation}
\label{eqI1NonIncreasingHomology}
|i_1(\xi)|\leq |\xi|
\qquad
\text{for every }\xi\in H(W,\widetilde{\delta}).
\end{equation}

Finally, by Lemma~\ref{lmCompPert}, $p_1$ and $i_1$ induce mutually
inverse maps on homology. Hence, for
every $\alpha\in H(V,d)$,
\[
|\alpha|
=
|i_1p_1(\alpha)|
\leq
|p_1(\alpha)|
\leq
|\alpha|,
\]
where the first inequality is
\eqref{eqI1NonIncreasingHomology} and the second is
\eqref{eqP1NonIncreasingHomology}. Thus
\[
|p_1(\alpha)|=|\alpha|,
\]
as claimed.
\end{proof}

\subsection{Quantitative deformation: Theorem B}
\label{SubSecProofThmA}

\begin{thmB}[Quantitative deformation]
\label{tmQuantitativeDeformation}

Set
\[
\hbar=\hbar(M,D,\theta).
\]

\begin{enumerate}

\item
\label{tmQuantitativeDeformationModel}
For every compact subset $K\subset D$ satisfying
\[
\beta(K)<\hbar,
\]
there are a differential
\[
d_{\mathrm{def}}\colon
SH^*_{\widehat D,\theta}(K;\bQ)\hat{\otimes}\Lambda
\longrightarrow
SH^{*+1}_{\widehat D,\theta}(K;\bQ)\hat{\otimes}\Lambda
\]
and a special deformation retraction
\[
\rho\colon
SC^*_M(K)
\longrightarrow
\bigl(
SH^*_{\widehat D,\theta}(K;\bQ)\hat{\otimes}\Lambda,
d_{\mathrm{def}}
\bigr)
\]
such that
\[
\val_{M,\theta}(d_{\mathrm{def}})\geq\hbar,
\qquad
|d_{\mathrm{def}}|_\infty<e^{-\hbar},
\]
and
\[
\val_{M,\theta}(\rho)\geq 0,
\qquad
\val_{M,\theta}(\rho-\rho_0)\geq\hbar,
\]
where $\rho_0$ is the order-zero reduction map. The induced map on
homology is an isometry, and its map on associated gradeds agrees with
the homology-level comparison supplied by
Theorem~\ref{tmFormalLocality}\textup{(1)}.
Taking $K=D$ gives the corresponding statement for $SH^*_\theta(D)$.

\item
\label{tmQuantitativeDeformationNaturality}
Let $K_1\subset K_2\subset D$ and suppose
$\beta(K_i)<\hbar$ for $i=1,2$. If
\[
r\colon
SH^*_{\widehat D,\theta}(K_2;\bQ)
\longrightarrow
SH^*_{\widehat D,\theta}(K_1;\bQ)
\]
is the intrinsic restriction map, then there is a chain map
\[
r_{\mathrm{def}}\colon
\bigl(SH^*_{\widehat D,\theta}(K_2;\bQ)
      \hat{\otimes}\Lambda,d_{\mathrm{def}}\bigr)
\longrightarrow
\bigl(SH^*_{\widehat D,\theta}(K_1;\bQ)
      \hat{\otimes}\Lambda,d_{\mathrm{def}}\bigr)
\]
satisfying
\[
\bigl|r\otimes 1-r_{\mathrm{def}}\bigr|_\infty<e^{-\hbar},
\qquad
\val_\theta(r\otimes 1-r_{\mathrm{def}})\geq\hbar.
\]
Under the deformation retractions of
Theorem~\ref{tmQuantitativeDeformation}\textup{(1)}, the square formed by
$r_{\mathrm{def}}$ and the ambient restriction map
\[
\xymatrix{
SC^*_{M}(K_2;\Lambda)\ar[d]\ar[r]&
SH^*_{\widehat D,\theta}(K_2;\bQ)\hat{\otimes}\Lambda\ar[d]_{r_{\mathrm{def}}}\\
SC^*_{M}(K_1;\Lambda)\ar[r]&
SH^*_{\widehat D,\theta}(K_1;\bQ)\hat{\otimes}\Lambda
}
\]
commutes up to chain homotopy.

\item
\label{tmQuantitativeDeformationOperations}
For a compact $K\subset D$ with $\beta(K)<\hbar$, set
\[
\mathcal A^*(K)
:=
\bigl(
SH^*_{\widehat D,\theta}(K;\bQ)\hat{\otimes}\Lambda,
d_{\mathrm{def}}
\bigr).
\]
\begin{enumerate}
\item
\label{tmQuantitativeDeformationBV}
\emph{(BV operator).}
There is a chain map
\[
\Delta_{\mathrm{def}}\colon
\mathcal A^*(K)\longrightarrow\mathcal A^{*-1}(K)
\]
whose action on homology corresponds, under
Theorem~\ref{tmQuantitativeDeformation}\textup{(1)}, to the ambient BV
operator. If $\Delta$ is the intrinsic BV operator, then
\[
|\Delta-\Delta_{\mathrm{def}}|_\infty<e^{-\hbar},
\qquad
\val_\theta(\Delta-\Delta_{\mathrm{def}})\geq\hbar.
\]

\item
\label{tmQuantitativeDeformationProduct}
\emph{(Pair-of-pants product).}
The ambient pair-of-pants product transports to a bounded product on
$H^*(\mathcal A^*(K))$. If, in addition, $d_{\mathrm{def}}=0$, the
transported ambient product defines a product $\ast_{\mathrm{def}}$ on
\[
SH^*_{\widehat D,\theta}(K;\bQ)\hat{\otimes}\Lambda
\]
and, writing $\ast$ for the intrinsic product,
\[
|\ast-\ast_{\mathrm{def}}|_\infty<e^{-\hbar},
\qquad
\val_\theta(\ast-\ast_{\mathrm{def}})\geq\hbar.
\]
\end{enumerate}

\end{enumerate}

\end{thmB}
\begin{proof}[Proof of \textup{(1)}]
    Fix an $S$-shaped acceleration datum $\mathcal A$ for $K\subset D\subset M$.
    By Lemma~\ref{lmTelFloerType} and Corollary~\ref{CYRelFilrPres},
    \[
    SC^*_M(K,\mathcal A)
    =
    \bigl(G^*\hat\otimes\Lambda,d_0+\delta\bigr),
    \]
    where $G^*$ is defined over the ground field and
    \begin{equation}
    \label{eqQuantDefPerturbationBound}
    \val_{M,\theta}(\delta)\geq\hbar,
    \qquad
    |\delta|_\infty<e^{-3\hbar}.
    \end{equation}
    
    At order zero, Lemma~\ref{lmExteriorOrderZeroRetract} contracts the exterior
    subcomplex and Proposition~\ref{prpFixedModelLocality} identifies the quotient
    isometrically with the intrinsic complex. Hence, by
    Corollary~\ref{cyBoundaryDepthInv},
    \[
    \beta(G,d_0)=\beta(K).
    \]
    Choose $\epsilon>0$ so that
    \begin{equation}
    \label{eqQuantDefEpsilonChoice}
    \beta(K)+4\epsilon<\hbar,
    \end{equation}
    and choose special deformation-retract data $i,p,h$ as in
    Lemma~\ref{lmSpecialDR}. Since $\beta(K)\geq0$, \eqref{eqQuantDefEpsilonChoice} gives
    $\epsilon<\hbar/4$, and therefore
    \[
    \beta(G,d_0)+6\epsilon
    =
    \beta(K)+6\epsilon
    <
    3\hbar.
    \]
    Together with \eqref{eqQuantDefPerturbationBound}, this gives
    \[
    |\delta|_\infty e^{\beta(G,d_0)+6\epsilon}<1,
    \]
    so the perturbation construction preceding Lemma~\ref{lmCompPert} applies.
    
    Let $A,\widetilde\delta,i_1,p_1,h_1$ be the resulting perturbation data.
    Lemma~\ref{lmCompPert} gives a special deformation retraction onto
    \[
    \bigl(H^*(G,d_0)\hat\otimes\Lambda,\widetilde\delta\bigr).
    \]
    By \eqref{eqHPLCorrectionEstimate} and
    \eqref{eqQuantDefPerturbationBound},
    \[
    |p_1-p|_\infty,\ |i_1-i|_\infty
    <
    e^{-3\hbar+\beta(K)+2\epsilon}
    <
    e^{-\hbar}.
    \]
    Since $\widetilde\delta=pAi$ and $|A|_\infty\leq|\delta|_\infty$, the same
    bounds together with \eqref{eqQuantDefEpsilonChoice} give
    \[
    |\widetilde\delta|_\infty
    <
    e^{2\epsilon-3\hbar}
    <
    e^{-\hbar}.
    \]
    Moreover, for the Floer type complex of Lemma~\ref{lmTelFloerType} the
    filtration $\val_E$ coincides with $\val_{M,\theta}$. Since
    $\val_{M,\theta}(\delta)\geq\hbar$, Lemma~\ref{lmCompPert} yields
    \[
    \val_{M,\theta}(\widetilde\delta)\geq\hbar,
    \qquad
    \val_{M,\theta}(p_1-p)\geq\hbar,
    \qquad
    \val_{M,\theta}(i_1-i)\geq\hbar.
    \]
    
    Let
    \[
    j:H^*(G,d_0)\hat\otimes\Lambda
    \longrightarrow
    SH^*_{\widehat D,\theta}(K;\bQ)\hat\otimes\Lambda
    \]
    be the order-zero isometric identification induced by
    Lemma~\ref{lmExteriorOrderZeroRetract} and
    Proposition~\ref{prpFixedModelLocality}. Transport the preceding special
    deformation retraction along $j$, and set
    \[
    \rho=j\circ p_1,
    \qquad
    \rho_0=j\circ p,
    \qquad
    d_{\mathrm{def}}=j\circ\widetilde\delta\circ j^{-1}.
    \]
    Thus the corresponding inclusion is $i_1\circ j^{-1}$ and the homotopy is
    $h_1$. Since $j$ is an isometry,
    \[
    |d_{\mathrm{def}}|_\infty<e^{-\hbar}.
    \]
    The order-zero map $j$ has relative order zero, so
    \[
    \val_{M,\theta}(d_{\mathrm{def}})\geq\hbar,
    \qquad
    \val_{M,\theta}(\rho-\rho_0)\geq\hbar.
    \]
    The order-zero map $\rho_0$ has relative order zero, and therefore
    \[
    \val_{M,\theta}(\rho)\geq0.
    \]
    The order-zero map induced by $\rho_0$ is precisely the homology-level
    comparison of Theorem~\ref{tmFormalLocality}\textup{(1)}. Finally,
    Lemma~\ref{lmCompPertIsometry} and the isometry of $j$ imply that
    the map induced by $\rho$ on homology is an isometry.
    \end{proof}

\begin{proof}[Proof of \textup{(2)}]
Fix S-shaped acceleration data $\mathcal A_2\leq \mathcal A_1$ for $K_1,K_2$ respectively, with $K_1\subset K_2$. Let $G_j$ be the reduced complex of the Floer type complex $SC^*_M(K_j,\mathcal A_j)$.  Then, denoting by $d_0$ the reduced differential, $(G_j,d_0)$ is a model for $SC^*_{\widehat{D},\theta}(K_j;\bK)$, for $j=1,2$. Write $R$ for the ambient Floer restriction map
\[
R:SC^*_M(K_2,\mathcal A_2)\longrightarrow SC^*_M(K_1,\mathcal A_1)
\]
and $R_0$ for its order-zero part on the reduced complexes $G_2\to G_1$.

Choose $\epsilon>0$ so small that
\[
\max\{\beta(K_1),\beta(K_2)\}+4\epsilon<\hbar.
\]
Fix special D.R.\ data $h^j,p^j,i^j$ for $G_j$ and the corresponding perturbed data $h^j_1,p^j_1,i^j_1$ from Lemma \ref{lmCompPert}, for $j=1,2$, with this common $\epsilon$. Then we have homotopy commutative diagrams
$$
\xymatrix{
G_2^* \ar[d]_{R_0} \ar[r]^{p^2} & SH^*_{\widehat{D},\theta}(K_2;\bK) \ar[d]_{r} \\
G_1^* \ar[r]^{p^1} & SH^*_{\widehat{D},\theta}(K_1;\bK)
}
$$
for $r=p^1\circ R_0\circ i^2$ the intrinsic homology-level restriction map, and
$$
\xymatrix{
SC^*_{M}(K_2;\Lambda) \ar[d]_{R} \ar[r]^{p^2_{1}} & SH^*_{\widehat{D},\theta}(K_2;\bK)\hat{\otimes}\Lambda \ar[d]_{r_{\mathrm{def}}} \\
SC^*_{M}(K_1;\Lambda) \ar[r]^{p^1_{1}} & SH^*_{\widehat{D},\theta}(K_1;\bK)\hat{\otimes}\Lambda
}
$$
where
\[
r_{\mathrm{def}}=p^1_{1}\circ R\circ i^2_{1}.
\]
Using $\id-i^2_1p^2_1=dh^2_1+h^2_1d$, the difference $p^1_1R-r_{\mathrm{def}}p^2_1=p^1_1R(\id-i^2_1p^2_1)$ is null-homotopic via $p^1_1R h^2_1$ (with the standard cochain sign convention).

Let $\widetilde r=p^1\circ R\circ i^2$. First,
\[
r_{\mathrm{def}}-\widetilde r
=
(p^1_1-p^1)R i^2_1
+
p^1R(i^2_1-i^2).
\]
Using norm non-increase of the Floer restriction map, the bounds $|p^j|,|i^j|,|p^j_1|,|i^j_1|<e^\epsilon$, \eqref{eqHPLCorrectionEstimate}, and $|\delta|_\infty<e^{-3\hbar}$, we obtain $|r_{\mathrm{def}}-\widetilde r|_\infty<e^{-\hbar}$. Second,
\[
\widetilde r-r\otimes 1
=
p^1(R-R_0)i^2.
\]
By Corollary \ref{CYRelFilrPres}, $|R-R_0|_\infty<e^{-3\hbar}$, and therefore $|\widetilde r-r\otimes 1|_\infty<e^{2\epsilon-3\hbar}<e^{-\hbar}$. By the ultrametric inequality, $|r\otimes 1-r_{\mathrm{def}}|_\infty<e^{-\hbar}$.

For the relative filtration, Lemma~\ref{lmCompPert} gives
$\val_{M,\theta}(p^j_1-p^j),\ \val_{M,\theta}(i^j_1-i^j)\geq\hbar$, while $R-R_0$ has relative order at least $\hbar$. Therefore $\val_{\theta}(r\otimes 1-r_{\mathrm{def}})\ge\hbar$.
\end{proof}

\begin{proof}[Proof of \textup{(3a)}]
Fix an S-shaped acceleration datum $\mathcal A$ and let $G^*$, the special
D.R.\ data $p,i,h$, and the perturbed maps $p_1,i_1$ be as in the
proof of part~(1), with $\epsilon$ satisfying \eqref{eqQuantDefEpsilonChoice}.
As in the proof of part~(1), we suppress the fixed
isometric identification
\[
j:H^*(G,d_0)\widehat\otimes\Lambda
  \longrightarrow
  SH^*_{\widehat D,\theta}(K)\widehat\otimes\Lambda .
\]

Write
\[
\Delta_M=\Delta_{M,0}+\Delta_{M,+},
\qquad
\Delta_{\mathrm{def}}=p_1\Delta_M i_1.
\]
By Corollary~\ref{CYRelFilrPres1},
\[
|\Delta_{M,0}|_\infty\leq1,\qquad
\val_{M,\theta}(\Delta_{M,+})\geq\hbar,\qquad
|\Delta_{M,+}|_\infty<e^{-3\hbar}.
\]
Proposition~\ref{prpFixedModelLocality} identifies the intrinsic BV
operator with
\[
\Delta=p\Delta_{M,0}i.
\]
Expanding directly,
\[
\Delta_{\mathrm{def}}-\Delta
=
(p_1-p)\Delta_M i_1
+
p\Delta_{M,+}i_1
+
p\Delta_{M,0}(i_1-i).
\]
This is the same unary transfer estimate used in the proof of
Theorem~\ref{tmQuantitativeDeformation}\textup{(2)}: every error term
contains either an HPL correction or the positive-order part
$\Delta_{M,+}$. Using \eqref{eqHPLCorrectionEstimate} with
$|\delta|_\infty<e^{-3\hbar}$, the bounds
$|p|_\infty,|i|_\infty,|p_1|_\infty,|i_1|_\infty\leq e^\epsilon$,
$|\Delta_M|_\infty\leq1$, Corollary~\ref{CYRelFilrPres1}, and the filtration
conclusions of Lemma~\ref{lmCompPert}, one obtains
\[
|\Delta_{\mathrm{def}}-\Delta|_\infty<e^{-\hbar},
\qquad
\val_\theta(\Delta_{\mathrm{def}}-\Delta)\geq\hbar.
\]
Since $p_1$, $i_1$, and $\Delta_M$ are chain maps, $\Delta_{\mathrm{def}}$
is a chain map and its induced homology operation is ambient BV.
\end{proof}

\begin{proof}[Proof of \textup{(3b)}]
The first assertion follows by transporting the bounded ambient product
through the isometric homology identification of part~(1). Fix an
S-shaped acceleration datum $\mathcal A$ and retain the notation of the
proof of part~(1), with $\epsilon$ satisfying
\eqref{eqQuantDefEpsilonChoice}. As in the proof of~(3a), suppress the
fixed order-zero isometric identification and write
\[
H:=SH^*_{\widehat D,\theta}(K;\bQ)\widehat\otimes\Lambda.
\]

Assume $d_{\mathrm{def}}=0$. Then the perturbed maps
\[
i_1:(H,0)\longrightarrow SC_M^*(K,\mathcal A),
\qquad
p_1:SC_M^*(K,\mathcal A)\longrightarrow(H,0)
\]
are chain maps inducing mutually inverse isometries on homology. If
\[
\mu_{\mathcal A}:
H^*(SC_M^*(K,\mathcal A))^{\otimes2}
\longrightarrow
H^*(SC_M^*(K,2\mathcal A))
\]
is the S-shaped pair-of-pants operation of
Subsection~\ref{subsecRelativeNovikovLattice} and $V$ is the filtered
weight-doubling inverse of Lemma~\ref{lmWeightDoubling}, then
\[
x*_{\mathrm{def}}y
=
(p_1)_*V_*\mu_{\mathcal A}\bigl((i_1)_*x,(i_1)_*y\bigr).
\]

We estimate this operation directly from the configurations defining the
S-shaped product. Decorate the two incoming ends by $i_1$, and the
outgoing end by $V$ followed by $p_1$. The completed-telescope
construction is linear in the levelwise pair-of-pants maps and in their
comparison homotopies, so we may expand these decorated configurations
according to relative order. Here the expansion is made in the ray-level
configuration sums defining the homology operation of
Subsection~\ref{subsecRelativeNovikovLattice}; no bilinear chain map on
the completed telescope is being chosen. The individual summands need not
themselves be chain maps. Write
\[
p_1=p+(p_1-p),
\qquad
i_1=i+(i_1-i),
\qquad
V=V_0+(V-V_0),
\]
and decompose every pair-of-pants map and comparison homotopy entering
the product construction into its order-zero and positive-order parts.
The all-order-zero part is the product induced on $H(G,d_0)$ by the
order-zero S-shaped pair-of-pants construction. By
Lemma~\ref{lmAlgebraFiltered} the exterior subcomplex is an ideal for
this product, while Lemma~\ref{lmExteriorOrderZeroRetract} contracts it
by a norm-nonincreasing homotopy and
Proposition~\ref{prpFixedModelLocality} identifies the quotient
isometrically with the intrinsic complex. Since $V_0$ preserves the
exterior subcomplex, Theorem~\ref{tmFormalLocality}\textup{(4)}
identifies the resulting order-zero product, under the fixed isometric
identification $j$, with the intrinsic product $*$.

Every remaining decorated configuration contains at least one of the
factors
\[
p_1-p,\qquad i_1-i,\qquad V-V_0,
\]
or a positive-order part of a pair-of-pants map or comparison homotopy.
By Lemma~\ref{lmCompPert}, Lemma~\ref{lmWeightDoubling}, and
Corollary~\ref{CYRelFilrPres1}, each distinguished factor has relative
order at least $\hbar$, while every other factor preserves the relative
filtration. Hence every non-order-zero configuration has relative order
at least $\hbar$, and therefore
\[
\val_{\theta}(*_{\mathrm{def}}-*)\ge\hbar.
\]

For the norm estimate, Lemma~\ref{lmCompPert} and the estimate used in
part~(1) give
\[
|p_1-p|_\infty,\ |i_1-i|_\infty
<
e^{-3\hbar+\beta(K)+2\epsilon},
\]
whereas Lemma~\ref{lmWeightDoubling} and Corollary~\ref{CYRelFilrPres1}
give
\[
|V-V_0|_\infty<e^{-3\hbar}
\]
and the same bound for every positive-order pair-of-pants or comparison
contribution. Moreover,
\[
|p|_\infty,\ |i|_\infty,\ |p_1|_\infty,\ |i_1|_\infty<e^{\epsilon},
\qquad
|V|_\infty,\ |V_0|_\infty\le1,
\]
and the full and order-zero pair-of-pants configuration maps are norm
non-increasing.

Use a telescoping expansion, changing successively $p_1$ to $p$, $V$ to
$V_0$, the product configuration to its order-zero part, and the two
input decorations $i_1$ to $i$. A term whose distinguished small factor
is $p_1-p$ or $i_1-i$ has norm less than
\[
e^{-3\hbar+\beta(K)+4\epsilon},
\]
while a term whose distinguished small factor is $V-V_0$ or a
positive-order pair-of-pants contribution has norm less than
\[
e^{-3\hbar+3\epsilon}.
\]
Both quantities are strictly smaller than $e^{-\hbar}$ by
\eqref{eqQuantDefEpsilonChoice}. The completed sums converge by the same
bounds used in Subsection~\ref{subsecRelativeNovikovLattice} and
Lemma~\ref{lmCompPert}. The ultrametric inequality therefore gives
\[
|*_{\mathrm{def}}-*|_\infty<e^{-\hbar}.
\]

In the virtual setting, Lemma~\ref{lmVirtualTrichotomy} supplies the
same operation estimates, and the remainder of the argument is unchanged.
\end{proof}

We next record how the deformation model interacts with the PSS map.

Assume throughout this discussion that $M$ is closed, that
$D\subset M$ is a Liouville domain carrying the local primitive
$\theta$, and that $K\subset D$ is compact.  Write
\[
j:D\hookrightarrow M
\]
for the inclusion.  Let
\[
\mathfrak p_K:
H^*(M;\Lambda)\longrightarrow SH_M^*(K)
\]
be PSS followed by restriction, and let
\[
\mathfrak p_K^{\mathrm{loc}}:
H^*(D;\bQ)\longrightarrow
SH^*_{\widehat D,\theta}(K;\bQ)
\]
be the intrinsic fixed-point map.

The relative filtration on the S-shaped model of
$\mathfrak p_K$ has an order-zero local quotient.  Using
Proposition~\ref{prpFixedModelLocality} to identify this quotient with the
intrinsic theory, denote the induced map by
\[
\mathfrak p_K^0:
H^*(M;\bQ)\longrightarrow
SH^*_{\widehat D,\theta}(K;\bQ).
\]

\begin{lm}\label{lmPSSOrderZero}
One has
\[
\mathfrak p_K^0
=
\mathfrak p_K^{\mathrm{loc}}\circ j^*.
\]
These maps are natural under restriction in $K$.
\end{lm}

\begin{proof}
Choose a sufficiently small S-shaped Morse/Floer model for PSS and an
admissible continuation into an S-shaped acceleration datum for $K$.
Choose the Morse/Floer model so that its local generators and
trajectories compute the Morse complex of $D$, and so that passage to
the local quotient induces the ordinary restriction map
\[
j^*:H^*(M;\bQ)\longrightarrow H^*(D;\bQ).
\]
By Corollary~\ref{CYRelFilrPres}, the continuation is filtration
preserving.  Its order-zero contributions with local output remain in
the local region.  Proposition~\ref{prpFixedModelLocality} identifies the
resulting target quotient with
$SC^*_{\widehat D,\theta}(K;\bQ)$, and under this identification the
induced continuation is the intrinsic fixed-point map.  This proves
the formula.  Naturality follows from the corresponding continuation
and restriction maps.
\end{proof}

Put
\[
\mathcal L_M
:=
H^*(M;\bQ)\otimes_{\bQ}\Lambda_{\geq0}
\subset H^*(M;\Lambda),
\]
and let $|\cdot|_M$ be the non-Archimedean lattice norm determined by
$\mathcal L_M$.  In particular,
\[
|a|_M=1
\qquad
\text{for every }0\neq a\in H^*(M;\bQ).
\]

Suppose now that $\beta(K)<\hbar$ and that the transferred differential
$d_{\mathrm{def}}$ of Theorem~\ref{tmQuantitativeDeformation}\textup{(1)} vanishes.  Let
\[
\mathcal J_K:
SH_M^*(K)\xrightarrow{\;\sim\;}
SH^*_{\widehat D,\theta}(K;\bQ)\widehat\otimes\Lambda
\]
be the isometric identification supplied by that theorem, and define
\[
E_K
:=
\mathcal J_K\circ\mathfrak p_K
-
\bigl(\mathfrak p_K^{\mathrm{loc}}\circ j^*\bigr)\otimes1.
\]
Here $\val_\theta(E_K)$ means the filtration order of this operator
with respect to the source lattice $\mathcal L_M$ and the intrinsic
relative valuation on the target. The operator norm of $E_K$ is taken
with respect to $|\cdot|_M$ on $H^*(M;\Lambda)$ and the intrinsic
Banach norm on the target.

\begin{cy}[PSS--restriction interface]
\label{cyPSSRestrictionInterface}
One has
\[
\val_{\theta}(E_K)\geq\hbar,
\qquad
|E_K|_\infty<1.
\]
Moreover,
\[
\Delta_M\,\mathfrak p_K(a)=0
\]
for every $a\in H^*(M;\Lambda)$.
\end{cy}

\begin{proof}
Use the small S-shaped PSS model from the proof of
Lemma~\ref{lmPSSOrderZero}, normalized so that its ground-field lattice
realizes $\mathcal L_M$ as the unit ball.  Write the continuation map as
\[
\kappa=\kappa_0+\kappa_+.
\]
By Corollary~\ref{CYRelFilrPres},
\[
\val_{M,\theta}(\kappa_+)\geq\hbar,
\qquad
|\kappa_+|_\infty<e^{-3\hbar}.
\]

Let $p_1=p-pAh$ be the Homological Perturbation projection from the
proof of Theorem~\ref{tmQuantitativeDeformation}\textup{(1)}.  Suppress the fixed isometric
identification
\[
H^*(G,d_0)\widehat\otimes\Lambda
\longrightarrow
SH^*_{\widehat D,\theta}(K;\bQ)\widehat\otimes\Lambda.
\]
By \eqref{eqHPLCorrectionEstimate} with $|\delta|_\infty<e^{-3\hbar}$, and by
Lemma~\ref{lmCompPert},
\[
|p_1-p|_\infty<e^{-\hbar},
\qquad
\val_{M,\theta}(p_1-p)\geq\hbar.
\]
Hence
\[
\mathcal J_K\circ\mathfrak p_K
=
p\kappa_0
+(p_1-p)\kappa
+p\kappa_+.
\]
Lemma~\ref{lmPSSOrderZero} identifies the first summand with
\[
\bigl(\mathfrak p_K^{\mathrm{loc}}\circ j^*\bigr)\otimes1.
\]
The other two summands have relative order at least $\hbar$.  Since
$|\kappa|_\infty\leq1$,
\[
|(p_1-p)\kappa|_\infty<e^{-\hbar},
\qquad
|p\kappa_+|_\infty<e^{\epsilon-3\hbar}<1.
\]
The ultrametric inequality gives
\[
|E_K|_\infty<1.
\]

Finally, the ambient BV operator vanishes on quantum cohomology, PSS is
BV-compatible \cite[\S2a]{SeidelLefschetzIII}, and restriction
commutes with BV.  Thus
\[
\Delta_M\,\mathfrak p_K(a)=0.
\]
\end{proof}

For $a\in H^*(M;\bQ)$ set
\[
x_0(a):=\mathfrak p_K^{\mathrm{loc}}(j^*a),
\qquad
x_+(a):=E_K(a).
\]
Then
\[
\mathcal J_K\bigl(\mathfrak p_K(a)\bigr)
=
x_0(a)+x_+(a),
\qquad
\val_\theta(x_+(a))\geq\hbar,
\qquad
|x_+(a)|<1.
\]
In particular, if $|x_0(a)|=1$, then
\[
\bigl|\mathcal J_K(\mathfrak p_K(a))\bigr|=1.
\]

\subsection{Dependence of $\hbar$ on the geometry}
\begin{lm}\label{lmOutsideContributions}
Fix $\delta>0$. Suppose that $D\subset M$ is $\delta'$-extendable for some
$\delta'>\delta$. There is a positive number
\[
\underline{\hbar}_{\mathrm{loc}}
=
\underline{\hbar}_{\mathrm{loc}}(D,J,\delta)>0
\]
depending only on the fixed local geometry near $D$, such that the constant
$\hbar(M,D)$ appearing in Theorem~\ref{tmQuantitativeDeformation} may be chosen with the lower
estimate
\[
\hbar(M,D)\ge
\underline{\hbar}_{\mathrm{loc}}.
\]
In particular this
choice is independent of the geometry of $M\setminus D_\delta$. Under
Liouville rescaling of the \emph{local model}, the intrinsic lower
estimate $\underline{\hbar}_{\mathrm{loc}}$ scales linearly.
\end{lm}

We need the following Lemma.

\begin{lm}\label{lmEpsilonExtendable}
    Fix some $\delta>0$. Let $D$ be a Liouville domain which is embedded in $M$ and is $\delta'$-extendable for some $\delta'>\delta$. Let $J$ be an almost complex structure on $M$ which is of contact type near the boundary of $D$. Then for curves which have at least one end inside, $\hbar$ can be made to depend only on $D$ and on $J$ restricted to a $\delta$-neighborhood of $D$. Denote this constant by $\hbar_{in}(D,J,\delta)$. For any $r>0$ we have that $\hbar_{in}(r\cdot D,\psi_r^*J,\delta)=r\hbar_{in}(D,J,\delta)$. Here $\psi_r$ is the dilation map $x\mapsto r\cdot x$ where $r$ acts by Liouville scaling.
    \end{lm}
    \begin{proof}
    This follows by examining the proof of Proposition \ref{PrpRelHbar}. With the reference-metric formulation of monotonicity, it is enough that we find a path in the domain which gets mapped into a region where the Hamiltonian is arbitrarily small and where the length of the image of that path is bounded from below by some uniform constant. A solution whose relative energy is non-zero and has one end inside would have to go across between the upper and lower boundary of an annulus of size $\delta$ around $D$, so Lemma~\ref{lmSmallSlicesRelative} produces such a path in a fixed $\delta$-neighborhood of $D$. The last clause follows by observing that the scaling map $\psi_r$ maps Floer solutions to appropriate Floer solutions while scaling the symplectic form and with it the geometric energy by $r$.
    \end{proof}

    \begin{proof}[Proof of Lemma~\ref{lmOutsideContributions}]
    Fix a compact $K\subset D$ and S-shaped acceleration data as in the proof of
    Theorem~\ref{tmQuantitativeDeformation}. Let $(G,d_0)$ be the corresponding
    order-zero reduced telescope, let $O\subset G$ be the exterior subcomplex, and put
    \[
    I:=G/O.
    \]
    By Lemma~\ref{lmExteriorOrderZeroRetract}, there are deformation-retract data
    \[
    (G,d_0)\underset{s}{\overset{\pi}{\rightleftarrows}}(I,d_I),
    \qquad
    |\pi|_\infty,\ |s|_\infty,\ |H|_\infty\leq1.
    \]
    Thus the exterior summand can be eliminated at no loss in norm.
    
    Write the full differential as $d=d_0+\delta$. The preliminary S-shaped
    energy gap of Section~3 gives $|\delta|_\infty<1$, and hence
    $|\delta H|_\infty<1$. Therefore
    \[
    A_O:=(\id+\delta H)^{-1}\delta
    \]
    is defined, and homological perturbation transfers the ambient complex to
    \[
    (I,d_I+\delta_I),
    \qquad
    \delta_I=\pi A_Os,
    \]
    with perturbed projection and inclusion
    \[
    p'=\pi-\pi A_OH,
    \qquad
    i'=s-HA_Os.
    \]
    All algebraic maps occurring here have norm at most one.
    
    Choose a common constant $E_{\mathrm{Floer}}>0$ for the inside-output energy
    bounds for the differential and for the finitely many continuation, product, and
    BV domains used in Theorem~\ref{tmQuantitativeDeformation}. By
    Lemma~\ref{lmEpsilonExtendable} and the analogous fixed-domain estimates for
    these operations, $E_{\mathrm{Floer}}$ depends only on the Floer geometry in a
    fixed $\delta$-neighborhood of $D$.
    
    In the virtual setting choose, inside the prescribed extendability region, an exact
    Liouville neighborhood $D^+$ containing the local region. Monotonicity gives a
    constant $E_{\mathrm{bub}}>0$, depending only on the geometry on $D^+$, such
    that every nonconstant holomorphic sphere meeting the local region has area at
    least $E_{\mathrm{bub}}$: such a sphere cannot be contained in $D^+$ and hence
    must cross a collar of positive width. Set
    \[
    E_{\mathrm{loc}}
    :=
    \begin{cases}
    E_{\mathrm{Floer}},&\text{in the classical setting},\\[2mm]
    \min\{E_{\mathrm{Floer}},E_{\mathrm{bub}}\},
    &\text{in the virtual setting}.
    \end{cases}
    \]
    
    We claim that every transferred contribution with output in $I$ which is not the
    corresponding intrinsic contribution contains either
    \begin{enumerate}
    \item a nonlocal Floer component with inside output, or
    \item in the virtual setting, a nonconstant sphere component meeting the local
    region.
    \end{enumerate}
    Indeed, expanding
    \[
    \delta_I
    =
    \pi\delta s-\pi\delta H\delta s
    +\pi\delta H\delta H\delta s-\cdots,
    \]
    if a contributing configuration contains a nonlocal Floer component, follow the
    corresponding term toward the final inside output; the first transition back to the
    local part is a nonlocal Floer component with inside output. The same argument
    applies to the expansions of $p'$, $i'$, and to transferred continuation, product,
    and BV operations. If all Floer components are local and there are no nonconstant
    sphere components, the contribution is the intrinsic one; in the virtual setting
    this is Proposition~\ref{prpVirtualLocality}.
    
    A witness of the first type contributes norm at most
    $e^{-E_{\mathrm{Floer}}}$, while a witness of the second type contributes norm at
    most $e^{-E_{\mathrm{bub}}}$. Since the remaining algebraic factors have norm at
    most one, it follows that
    \[
    |\delta_I|_\infty,\qquad
    |p'-\pi|_\infty,\qquad
    |B'_I-B_I|_\infty
    \leq e^{-E_{\mathrm{loc}}},
    \]
    where
    \[
    B'_I=p'_0B(i'_1,\ldots,i'_n)
    \]
    denotes any transferred restriction, product, or BV operation and $B_I$ the
    corresponding intrinsic operation. The same witnesses give relative-filtration
    order at least $E_{\mathrm{loc}}$. In the virtual setting we use
    Lemma~\ref{lmVirtualTrichotomy} and
    Proposition~\ref{prpVirtualLocality}.
    
    Choose
    \[
    0<3\hbar<E_{\mathrm{loc}}.
    \]
    This choice depends only on the fixed local geometry near $D$, and in particular
    is independent of $K$ and of the geometry of $M\setminus D_\delta$. Whenever
    $\beta(K)<\hbar$, Lemma~\ref{lmCompPert} applies to
    $(I,d_I+\delta_I)$, and the quantitative estimates in the proof of
    Theorem~\ref{tmQuantitativeDeformation} apply verbatim with this value of
    $\hbar$. Hence
    \[
    \hbar(M,D)\geq \underline{\hbar}_{\mathrm{loc}}
    \]
    for some positive $\underline{\hbar}_{\mathrm{loc}}$ depending only on the local
    geometry specified in the statement.
    
    Finally, the Floer inside-output bound of
    Lemma~\ref{lmEpsilonExtendable} and, in the virtual case, the monotonicity bound
    for sphere bubbles both scale linearly under Liouville rescaling. Thus
    $\underline{\hbar}_{\mathrm{loc}}$ scales linearly as well.
    \end{proof}
\section{Embedding torsion}\label{SecTau}

\subsection{The embedding torsion invariant}

Let $D$ be a Liouville domain with primitive $\theta$, and let
$K\subset D$ be compact. Fix an S-shaped package $\mathcal P$ for
$(D,\theta)$ and a corresponding reduction threshold $\hbar>0$ as in
Theorem~A. We call $\hbar_{\mathrm{disc}}$ \emph{admissible} if
\[
0<\hbar_{\mathrm{disc}}\leq \hbar.
\]
When several packages are compared, ``admissible'' means admissible for
all of them.

For an admissible $\hbar_{\mathrm{disc}}$, the locality spectral
sequence is the spectral sequence of the one-sided filtration
\[
F^pL(K)=T^{p\hbar_{\mathrm{disc}}}L(K),\qquad p\geq0,
\]
as in Section~4.5. Thus
\[
E_1^{p,q}
=
SH_{\widehat D,\theta}^{p+q}(K;\mathbb Q)
\otimes
\Lambda_{[p\hbar_{\mathrm{disc}},(p+1)\hbar_{\mathrm{disc}})},
\qquad p\geq0,
\]
and $E_1^{p,q}=0$ for $p<0$. The interval Novikov factor carries the
residual real valuation, with $\operatorname{val}(0)=+\infty$.

If $d_r$ is the first nonzero higher differential, then
\[
r\hbar_{\mathrm{disc}}
\leq
\operatorname{val}(d_r)
<
(r+1)\hbar_{\mathrm{disc}},
\]
and we set
\[
\tau_{\hbar_{\mathrm{disc}}}(K,D,M,\theta)
:=
\operatorname{val}(d_r)
:=
\inf_{0\neq x\in E_r}
\bigl(
\operatorname{val}(d_rx)-\operatorname{val}(x)
\bigr).
\]
If every higher differential vanishes, set
\[
\tau_{\hbar_{\mathrm{disc}}}(K,D,M,\theta)=+\infty.
\]

\begin{prp}[Comparison properties of
$\tau_{\hbar_{\mathrm{disc}}}$]
\label{prpTauComparison}
The following hold.

\begin{enumerate}

\item[(a)]
For every common admissible $\hbar_{\mathrm{disc}}$, the quantity
\[
\tau_{\hbar_{\mathrm{disc}}}(K,D,M,\theta)
\]
is independent of the regular S-shaped model and of the numerical
parameters $\delta,\sigma,\eta_{\mathrm{end}}$. In the geometrically
finite noncompact case the asymptotic coefficient at infinity is kept
fixed.

\item[(b)]
If
\[
\beta(K)<\hbar,
\]
then $\tau_{\hbar_{\mathrm{disc}}}(K,D,M,\theta)$ is independent of the
admissible discretization width and, for the special deformation model
of Theorem~B(1),
\[
\tau_{\hbar_{\mathrm{disc}}}(K,D,M,\theta)
=
\operatorname{val}_{M,\theta}(d_{\mathrm{def}}).
\]

\item[(c)]
If
\[
K\subset D_1\subset D_2
\]
and $D_1,D_2$ are Liouville domains for the same primitive $\theta$,
then for every common admissible discretization,
\[
\tau_{\hbar_{\mathrm{disc}}}(K,D_1,M,\theta)
=
\tau_{\hbar_{\mathrm{disc}}}(K,D_2,M,\theta).
\]

\item[(d)]
If $\theta_0,\theta_1$ are Liouville primitives near $D$ and
\[
[\omega,\theta_0]=[\omega,\theta_1]
\in H^2(M,D;\mathbb R),
\]
then for every common admissible discretization,
\[
\tau_{\hbar_{\mathrm{disc}}}(K,D,M,\theta_0)
=
\tau_{\hbar_{\mathrm{disc}}}(K,D,M,\theta_1).
\]

\item[(e)]
If $K_1\subset K_2\subset D$ and the intrinsic restriction map
\[
SH^*_{\widehat D}(K_2;\mathbb Q)
\longrightarrow
SH^*_{\widehat D}(K_1;\mathbb Q)
\]
is injective, then for every admissible discretization,
\[
\tau_{\hbar_{\mathrm{disc}}}(K_1,D,M,\theta)
\leq
\tau_{\hbar_{\mathrm{disc}}}(K_2,D,M,\theta).
\]

\item[(f)]
Suppose $K\subset D$ are Liouville domains for $\theta$ and
\[
\partial K\cap\operatorname{Skel}(D)=\varnothing.
\]
Then for every admissible discretization,
\[
\tau_{\hbar_{\mathrm{disc}}}(K,D,M,\theta)
=
\tau_{\hbar_{\mathrm{disc}}}(D,D,M,\theta).
\]
If $\beta(D)<\infty$, both sides are independent of every sufficiently
fine admissible discretization, with no assumption on $\beta(K)$.

\end{enumerate}
\end{prp}

\begin{proof}
For (a), fix
\[
0<\hbar_{\mathrm{disc}}\leq\hbar.
\]
Proposition~4.14 gives independence of the regular S-shaped model and
of the numerical parameters of the package. In the geometrically finite
noncompact case the asymptotic coefficient is fixed throughout.

For (b), assume $\beta(K)<\hbar$. The perturbed deformation-retract
data of Theorem~B(1) are
\[
i_1=i-hAi,\qquad
p_1=p-pAh,\qquad
h_1=h-hAh.
\]
The maps $i,p,h$ have relative filtration order zero and $A$ has
positive relative filtration order. Hence $i_1,p_1,h_1$ have
nonnegative relative filtration order. They restrict to the
nonnegative lattices and identify, from the $E_1$-page onward, the
locality spectral sequence of the ambient complex with that of
\[
\bigl(
SH^*_{\widehat D,\theta}(K;\mathbb Q)
\widehat\otimes\Lambda,
d_{\mathrm{def}}
\bigr),
\]
while preserving the residual real valuation. Therefore
\[
\tau_{\hbar_{\mathrm{disc}}}(K,D,M,\theta)
=
\operatorname{val}_{M,\theta}(d_{\mathrm{def}}),
\]
and the right-hand side is independent of
$\hbar_{\mathrm{disc}}$.

For (c), by Lemma~3.10 we may choose one regular acceleration datum for
$K$ which is simultaneously S-shaped for $D_1$ and $D_2$. The two
locality spectral sequences are then constructed from the same
filtered complex, hence from the same nonnegative lattice. Their
fixed-discretization torsions therefore agree.

For (d), by Lemma~3.10, after shrinking the auxiliary parameters if
necessary, choose one regular acceleration datum which is
simultaneously S-shaped for $\theta_0$ and $\theta_1$. Equality of the
relative classes gives
\[
\theta_1-\theta_0=\eta|_D+df
\]
for a closed one-form $\eta$ on $M$. On the summand indexed by
$h\in H_1(M;\mathbb Z)$, the two relative valuations differ by
\[
-\langle\eta,h\rangle.
\]
Multiplication on that summand by the corresponding power of $T$ gives
a filtered conjugacy of the two lattices. Since Floer differentials and
continuation maps preserve $h$, the two fixed-discretization spectral
sequences have the same higher-differential orders and residual
valuations.

For (e), if the $K_2$ spectral sequence has no nonzero higher
differential, the claim is immediate. Otherwise let $r$ be its first
nonzero page. If the $K_1$ spectral sequence has a nonzero differential
on a page $s<r$, then
\[
\tau_{\hbar_{\mathrm{disc}}}(K_1)
<
\tau_{\hbar_{\mathrm{disc}}}(K_2).
\]
Otherwise $E_r(K_j)=E_1(K_j)$ for $j=1,2$, and the restriction map on
$E_r$ is the intrinsic restriction map tensored with the interval
Novikov factor. Since the intrinsic restriction is injective and
defined over the trivially valued ground field, it preserves the
leading Novikov valuation of nonzero elements. Naturality gives
\[
f_r d_r^{K_2}=d_r^{K_1}f_r.
\]
Thus for every $x$ with $d_r^{K_2}x\neq0$,
\[
\operatorname{val}(d_r^{K_1}f_rx)-\operatorname{val}(f_rx)
=
\operatorname{val}(d_r^{K_2}x)-\operatorname{val}(x).
\]
Taking an approximating sequence if the infimum is not attained gives
the claim.

For (f), flow along the Liouville vector field identifies
$D\setminus\operatorname{int}K$ with a product Liouville cobordism,
hence
\[
\widehat K\simeq\widehat D
\]
by an exact symplectomorphism. Forgetting the norm, the canonical
intrinsic restriction map
\[
SH^*_{\theta,\widehat D}(D;\mathbb Q)
\longrightarrow
SH^*_{\theta,\widehat D}(K;\mathbb Q)
\]
is therefore an isomorphism.\footnote{The inverse need not be bounded.
Over the trivially valued ground field this does not affect the
algebraic isomorphism used here.}
By naturality of Proposition~4.12, the restriction map of relative
Novikov lattices induces an isomorphism of locality spectral sequences
from the $E_1$-page onward, preserving the residual real valuation.
Hence
\[
\tau_{\hbar_{\mathrm{disc}}}(K,D,M,\theta)
=
\tau_{\hbar_{\mathrm{disc}}}(D,D,M,\theta).
\]

Now assume $\beta(D)<\infty$. Choose $t>0$ sufficiently small that
\[
\beta(tD)=t\beta(D)<\hbar.
\]
Applying the equality just proved to $tD\subset D$ gives
\[
\tau_{\hbar_{\mathrm{disc}}}(D,D,M,\theta)
=
\tau_{\hbar_{\mathrm{disc}}}(tD,D,M,\theta).
\]
By (b), the right-hand side is independent of every sufficiently fine
admissible $\hbar_{\mathrm{disc}}$. The same is therefore true of the
left-hand side, and then of the corresponding quantity for every $K$
as in (f).
\end{proof}

Whenever $\beta(K)<\hbar$, write
\[
\tau(K,D,M,\theta)
:=
\tau_{\hbar_{\mathrm{disc}}}(K,D,M,\theta)
\]
for the common value in Proposition~\ref{prpTauComparison}(b).

If $\beta(D)<\infty$, Proposition~\ref{prpTauComparison}(f) shows that
\[
\tau(D,\theta,M)
:=
\tau_{\hbar_{\mathrm{disc}}}(D,D,M,\theta)
\]
is well defined for every sufficiently fine admissible discretization.
Moreover, if $K\subset D$ is a Liouville subdomain with
\[
\partial K\cap\operatorname{Skel}(D)=\varnothing,
\]
then
\[
\tau(D,\theta,M)
=
\tau_{\hbar_{\mathrm{disc}}}(K,D,M,\theta)
\]
for every sufficiently fine admissible discretization, without any
assumption on $\beta(K)$.

\begin{cy}
\label{cyTauSkeleton}
\label{cySkelDep}
The invariant $\tau(D,\theta,M)$ depends only on the skeleton of $D$
with respect to $\theta$.
\end{cy}

\begin{proof}
Let $D'$ be another Liouville domain for $\theta$ with the same
skeleton. For $t>0$ sufficiently small, $tD\subset D'$ and contains the
common skeleton. Proposition~\ref{prpTauComparison}(f), then (c), then
(f) again give
\[
\tau(D',\theta,M)
=
\tau_{\hbar_{\mathrm{disc}}}(tD,D',M,\theta)
=
\tau_{\hbar_{\mathrm{disc}}}(tD,D,M,\theta)
=
\tau(D,\theta,M).
\]
\end{proof}

\subsection{Families of filtrations and concavity}\label{SubSecTauConcavity}
From now on we consider only the case $K=D$. We ask how $\tau$ changes when $\theta$ is varied by a \emph{non-exact} 1-form. Fix a Liouville form $\theta$. Let $\cL$ be the set of all Liouville forms for $D$ and consider the map $f:\cL\to H^1(D;\bR)$ taking the Liouville form $\theta'$ to the cohomology class $[\theta'-\theta]$. Choose a convex open neighborhood $U\subset H^1(D;\bR)$ of the origin contained in the image of $f$.

\begin{rem}
The open set $U$ is chosen convex, and the fibers of $f$ over $U$ are convex.
\end{rem}

We obtain a function $\tau_{\theta}:U\to\bR\cup\{\infty\}$ by taking $\tau_{\theta}(v):=\tau(D,M,\theta+\alpha)$ for any $\alpha$ representing the cohomology class $v\in U\subset H^1(D;\bR)$ and so that $\theta+\alpha$ is a Liouville form for $D$. Note that $\tau_{\theta}$ is well defined by Proposition~\ref{prpTauComparison}\textup{(d)} together with the canonical definition of $\tau$.

Fix a smooth local section
\[
\sigma:U\to\cL,
\qquad
\sigma(0)=\theta.
\]
For $t>0$ sufficiently small put $K=tD$. After shrinking $U$, $K$ is a Liouville subdomain for every $\sigma(v)$, and $\operatorname{Skel}(D,\sigma(v))\subset\operatorname{int}K$. Consequently,
\[
\tau(D,D,M,\sigma(v))
=
\tau(K,D,M,\sigma(v)),
\]
and
\[
\beta(K,\sigma(v))
=
\beta(K,\theta)
=
t\beta(D).
\]
The first two assertions are $C^1$-open on the compact cobordism $D\setminus\operatorname{int}K$; the equality of $\tau$ follows from Proposition~\ref{prpTauComparison}\textup{(f)}, and the boundary-depth identity follows from primitive independence and Liouville scaling.

\begin{prp}[Concavity under variation of the primitive]
\label{prpTauConcave}
\label{tmConcave}
Suppose $\beta(D)<\infty$.  Then
\[
\tau_\theta:U\to\bR\cup\{\infty\}
\]
is either identically infinite or finite everywhere and concave.
\end{prp}

For the proof of the proposition, fix full-dimensional compact convex
polytopes
\[
0\in\operatorname{int}P,
\qquad
P\Subset\operatorname{int}\widetilde P\Subset U,
\]
and let $w_1,\ldots,w_N$ be the vertices of $\widetilde P$. For each
$i$, choose an S-shaped package $\mathcal P_i$ for the primitive
$\sigma(w_i)$, and for $\sigma(0)$ with trichotomy constant $\hbar_i>0$. 
Lemma~\ref{lmSShapedFlexibility} supplies regular acceleration data
simultaneously S-shaped for this finite collection of packages. In the
geometrically finite noncompact case, choose the same asymptotic
coefficient at infinity for all the packages. Any two such choices of
acceleration data admit a common lower regular datum and comparison
morphisms which remain simultaneously S-shaped for the same collection
of packages. Since
\[
P\Subset\operatorname{int}\widetilde P,
\]
there is $c>0$ such that every $v\in P$ admits a convex expression
\[
v=\sum_i t_iw_i,
\qquad
t_i\geq c,
\qquad
\sum_i t_i=1.
\]
Set
\[
\hbar_P:=c\min_i\hbar_i.
\]

\begin{lm}[Uniform family gap]
\label{lmFamilyGap}
For any Floer differential, continuation, or coherence-homotopy datum
which is simultaneously S-shaped for all the packages $\mathcal P_i$,
every solution $u$ occurring in that datum satisfies
\[
E_{\sigma(v)}(u)\geq0
\qquad
\text{for all }v\in\widetilde P.
\]
Moreover, either
\[
E_{\sigma(v)}(u)=0
\qquad
\text{for all }v\in\widetilde P,
\]
or
\[
E_{\sigma(v)}(u)>\hbar_P
\qquad
\text{for all }v\in P.
\]
In the latter case,
\[
E_{\mathrm{top}}(u)>3\hbar_P.
\]
\end{lm}

\begin{proof}
At every vertex $w_i$, simultaneous S-shapedness and the trichotomy give
\[
E_{\sigma(w_i)}(u)\geq0.
\]
For
\[
v=\sum_i t_iw_i
\]
the forms $\sigma(v)$ and $\sum_i t_i\sigma(w_i)$ differ by an exact
form, so
\[
E_{\sigma(v)}(u)
=
\sum_i t_iE_{\sigma(w_i)}(u).
\]
This proves nonnegativity on $\widetilde P$.

If this affine function is not identically zero, then
$E_{\sigma(w_i)}(u)>0$ for some $i$. The trichotomy for
$\mathcal P_i$ gives
\[
E_{\sigma(w_i)}(u)>\hbar_i.
\]
Hence, for every $v\in P$,
\[
E_{\sigma(v)}(u)
>
c\hbar_i
\geq
\hbar_P.
\]
The corresponding topological-energy estimate at $w_i$ gives
\[
E_{\mathrm{top}}(u)>3\hbar_i\geq3\hbar_P.
\]
\end{proof}

We now package the filtered Floer complexes for the primitives $\sigma(v)$, $v\in P$, into a single family complex from which each will be obtained by specialization. We fix an acceleration datum as in Lemma \ref{lmFamilyGap}. We first construct the coefficient ring for the family. For $v\in P$ define
\[
\ell_v:H_2(M,D;\bZ)\longrightarrow\bR,
\qquad
\ell_v(\alpha)=\langle[\omega,\sigma(v)],\alpha\rangle,
\]
and let
\[
C(P):=
\left\{
\alpha\in H_2(M,D;\bZ):
\ell_v(\alpha)\geq0
\text{ for every }v\in P
\right\}.
\]
Set
\[
R_{\mathrm{alg}}(P):=\Lambda[C(P)].
\]
For every $v\in P$ there is a specialization
\[
\operatorname{sp}_v:R_{\mathrm{alg}}(P)\longrightarrow\Lambda,
\qquad
\operatorname{sp}_v(e^\alpha)=T^{\ell_v(\alpha)}.
\]
Define
\[
|a|_P:=\sup_{v\in P}|\operatorname{sp}_v(a)|,
\qquad
N_P:=\{a\in R_{\mathrm{alg}}(P):|a|_P=0\},
\]
and let
\[
R(P):=\widehat{R_{\mathrm{alg}}(P)/N_P}
\]
with respect to the induced non-Archimedean norm. Each
$\operatorname{sp}_v$ extends to a contracting homomorphism
\[
\operatorname{sp}_v:R(P)\longrightarrow\Lambda.
\]
For a monomial,
\[
-\log|e^\alpha|_P=\min_{v\in P}\ell_v(\alpha),
\]
and if $\ell_v(\alpha)=0$ for every $v\in P$, then $e^\alpha=1$ in
$R(P)$.

For fixed asymptotics, the difference
$\ell_v(\alpha)-\ell_0(\alpha)$ is uniformly bounded for $v\in P$.
Hence the usual Novikov convergence at $v=0$ implies convergence in
$R(P)$, both for Floer differentials and for continuation maps.

Let $K=tD$ be a common Liouville core as above. For each Floer datum
$(H_i,J_i)$ in the acceleration datum $\cA$, let
\[
CF^*(H_i,J_i;P)
\]
be the finite free $R(P)$-module on the normalized generators
$S^{[\gamma]}\gamma$. Floer and continuation coefficients are weighted
by the corresponding monomials $e^{[u]}$. Specialization is given by
\[
\operatorname{sp}_v
\left(
S^{[\gamma]}\gamma\,e^{[u]}
\right)
=
T^{-\cA_{\sigma(v)}(\gamma)}
T^{\ell_v([u])}\gamma,
\]
and therefore recovers the Floer data for the primitive $\sigma(v)$.

Form the completed telescope exactly as in \S\ref{subsecTelescope}, using the
seminorm
\[
|x|_P:=\sup_{v\in P}|\operatorname{sp}_v(x)|
\]
and quotienting by the null submodule before completion. Denote the
result by
\[
SC_M^*(K,P).
\]
Every specialization
\[
\operatorname{sp}_v:
SC_M^*(K,P)\longrightarrow SC_M^*(K;\sigma(v))
\]
is contracting.

\begin{lm}[The order-zero family model]
\label{lmFamilyOrderZero}
Let $G$ denote the underlying $R(P)$-module of $SC_M^*(K,P)$. Its
differential has a decomposition
\[
d=d_0+\delta_P,
\qquad
|\delta_P|_P<e^{-3\hbar_P},
\]
where
\[
(G,d_0)=(G_0,d_0)\hat\otimes R(P)
\]
for a ground-field complex $(G_0,d_0)$ satisfying
\[
H^*(G_0,d_0)
\cong
H_K:=SH^*_{\widehat D,\theta}(K;\bQ),
\qquad
\beta(G_0,d_0)=\beta(K).
\]

Moreover, $d_0$ preserves the grading by
\[
\Gamma:=H_1(K;\bZ)/\operatorname{Tor},
\]
and for every $\epsilon>0$ there are $\Gamma$-graded special
deformation-retract data $(i_0,p_0,h_0)$ for $(G_0,d_0)$ such that,
after extending $R(P)$-linearly,
\[
|h_0|_P<e^{\beta(K)+\epsilon},
\qquad
|i_0|_P,\ |p_0|_P\leq e^\epsilon.
\]
\end{lm}

\begin{proof}
By Lemma~\ref{lmFamilyGap}, every contribution has either family order
zero or family order greater than $\hbar_P$, and in the latter case its
topological energy is greater than $3\hbar_P$. Corollary~\ref{CYRelFilrPres} therefore
gives
\[
d=d_0+\delta_P,
\qquad
|\delta_P|_P<e^{-3\hbar_P}.
\]

If a contribution $u$ has family order zero, then
\[
E_{\sigma(v)}(u)=0
\]
on the full-dimensional polytope $\widetilde P$. Hence
$\ell_v([u])=0$ for every $v\in P$, so $e^{[u]}=1$ in $R(P)$.
Thus the entire order-zero telescope, including the exterior
critical-point generators, is the completed scalar extension of a
ground-field complex:
\[
(G,d_0)=(G_0,d_0)\hat\otimes R(P).
\]
Conversely, because $0\in P$, Lemma~\ref{lmFamilyGap} implies that if
\[
E_{\sigma(0)}(u)=0,
\]
then the positive-family-order alternative is impossible, hence
\[
E_{\sigma(v)}(u)=0
\qquad
\text{for all }v\in\widetilde P.
\]
Thus a contribution has family order zero if and only if it has
$\sigma(0)$-order zero. Hence $G_0$ is exactly the reduced order-zero
complex associated to the base primitive $\sigma(0)=\theta$. Identifying
this complex with the reduced order-zero complex of
Lemma~\ref{lmTelFloerType} as in the proof of
Theorem~\ref{tmQuantitativeDeformation}\textup{(1)},
Proposition~\ref{prpGrLocality}, Corollary~\ref{cyBoundaryDepthInv}, and
Lemma~\ref{lmBoundaryDepthClosed} give
\[
H^*(G_0,d_0)\cong H_K,
\qquad
\beta(G_0,d_0)=\beta(K).
\]

For an order-zero contribution the affine energy vanishes on the
full-dimensional polytope $\widetilde P$, hence
\[
\langle v,\partial[u]\rangle=0
\]
for every $v$ and $d_0$ preserves the
$\Gamma$-grading. This grading is orthogonal, so every graded summand
has boundary depth at most $\beta(K)$. Applying Lemma~\ref{lmSpecialDR} grade by
grade with a common $\epsilon>0$ gives $\Gamma$-graded special
deformation-retract data $(i_0,p_0,h_0)$. Changing the primitive from
$\sigma(0)$ to $\sigma(v)$ rescales an entire $\Gamma$-summand by one
scalar, so the same norm estimates hold at every specialization.
After extending $R(P)$-linearly,
\[
|h_0|_P<e^{\beta(K)+\epsilon},
\qquad
|i_0|_P,\ |p_0|_P\leq e^\epsilon.
\]
\end{proof}

\begin{lm}\label{lmFamilySDR}
    Assume $\beta(K)<\hbar_P$. Then family homological perturbation
    produces a transferred differential
    \[
    d^P_{\mathrm{def}}
    \]
    on $H_K\widehat\otimes R(P)$ such that, for every $v\in P$ and every
    sufficiently fine admissible discretization,
    \[
    \tau_{\hbar_{\mathrm{disc}}}(K,D,M,\sigma(v))
    =
    \operatorname{val}_{\sigma(v)}
    \bigl(\operatorname{sp}_v(d^P_{\mathrm{def}})\bigr).
    \]
\end{lm}

\begin{proof}
Choose $\epsilon>0$ small enough that
$e^{-3\hbar_P}e^{\beta(K)+\epsilon}<1$.  Lemma~\ref{lmFamilyOrderZero}
then gives $|\delta_Ph_0|_P<1$, so the Neumann series for
$(\operatorname{id}+\delta_Ph_0)^{-1}$ converges and the Homological
Perturbation Lemma yields
\[
A_P
:=
(\operatorname{id}+\delta_Ph_0)^{-1}\delta_P,
\qquad
d_{\mathrm{def}}^P
:=
p_0A_Pi_0.
\]
Specialization is contracting, hence continuous, and therefore commutes with
the convergent HPL formulas.

The unperturbed maps $i_0,p_0,h_0$ are ground-field and $\Gamma$-graded,
hence preserve the relative filtration after every specialization; every HPL
correction contains at least one positive-family-order factor.  Thus the
specialized HPL deformation retraction is $\val_{\sigma(v)}$-filtered and
preserves the residual real valuation.

Fix $v\in P$. By Lemma~\ref{lmSShapedFlexibility} choose a regular acceleration datum
$\mathcal A_v$ which is simultaneously S-shaped for the fixed vertex
packages and for a package associated to $\sigma(v)$. Comparing
$\mathcal A$ and $\mathcal A_v$ through a common simultaneous refinement,
Lemma~\ref{lmFamilyGap} shows that all comparison maps and homotopies preserve
$\operatorname{val}_{\sigma(v)}$. Hence the specialized family telescope
is filtered-chain-homotopy equivalent to the genuine
$\sigma(v)$-S-shaped telescope associated to $\mathcal A_v$, preserving
the residual real valuation.

Finally, since $A_v$ is a genuine $\sigma(v)$-S-shaped acceleration
datum, the filtration-preserving comparison above identifies the
locality spectral sequence of the specialized family telescope with that
of $A_v$, preserving residual real valuation. Proposition~\ref{prpTauComparison}\textup{(a)}
gives model independence of $\tau_{\hbar_{\mathrm{disc}}}$ at fixed
discretization, so the spectral sequence of $A_v$ computes
$\tau_{\hbar_{\mathrm{disc}}}(K,D,M,\sigma(v))$. Combining this with the
specialized HPL deformation retraction yields
\[
\tau_{\hbar_{\mathrm{disc}}}(K,D,M,\sigma(v))
=
\val_{\sigma(v)}
\bigl(
\operatorname{sp}_v(d_{\mathrm{def}}^P)
\bigr).
\]
\end{proof}

\begin{proof}[Proof of Proposition~\ref{prpTauConcave}]
Fix $v_0\in U$, set $\theta_0:=\sigma(v_0)$, and use coordinates
$w=v-v_0$ with local section
\[
\widetilde\sigma(w):=\sigma(v_0+w),
\qquad
\widetilde\sigma(0)=\theta_0.
\]
Relabel $w$ as $v$ and $\widetilde\sigma$ as $\sigma$.
Choose full-dimensional compact convex polytopes
\[
0\in\operatorname{int}P
\Subset\operatorname{int}\widetilde P\Subset U.
\]
Choose vertex packages for $\widetilde P$ and the constant $\hbar_P$ as in
the preamble to Lemma~\ref{lmFamilyGap}. Choose $t>0$ small enough that
\[
\beta(tD)=t\beta(D)<\hbar_P,
\]
and set $K=tD$. By Proposition~\ref{prpTauComparison}\textup{(f)}, for every
sufficiently fine common discretization,
\[
\tau(D,\sigma(v),M)
=
\tau_{\hbar_{\mathrm{disc}}}(K,D,M,\sigma(v)).
\]
Lemma~\ref{lmFamilySDR} then yields
\[
\tau(D,\sigma(v),M)
=
\val_{\sigma(v)}
\bigl(
\operatorname{sp}_v(d_{\mathrm{def}}^P)
\bigr).
\]

The isometric identification $H^*(G_0,d_0)\cong H_K$ equips $H_K$ with the
intrinsic norm, for which the homology-class decomposition
\[
H_K^*
=
\widehat{\bigoplus}_{h\in H_1(K;\bZ)}
(H_K)_h^*
\]
is orthogonal.  The transferred family differential decomposes into blocks
\[
d_{\mathrm{def}}^P
=
\sum_{h,h'}
d_{h',h}^P,
\qquad
d_{h',h}^P:
(H_K)_h^*
\longrightarrow
(H_K)_{h'}^{*+1}.
\]
Every relative class $\alpha$ in the block $h\to h'$ has $\partial\alpha=h'-h$,
so
\[
\operatorname{sp}_v(d_{h',h}^P)
=
T^{\langle v,h-h'\rangle}\,
\operatorname{sp}_0(d_{h',h}^P).
\]
A block therefore vanishes at one specialization if and only if it vanishes at
every specialization.  Writing
\[
c_{h',h}
:=
\val\bigl(\operatorname{sp}_0(d_{h',h}^P)\bigr)
\in\bR\cup\{+\infty\},
\]
orthogonality of the source and target homology-class decompositions yields
\[
\tau(v)
=
\inf_{h,h'}
\bigl(
c_{h',h}+\langle v,h-h'\rangle
\bigr).
\]
Local concavity and the local finite/$+\infty$ dichotomy follow at once: if
every block vanishes then $\tau\equiv+\infty$; otherwise each nonzero block has
finite affine valuation for every $v$.

Every point of $U$ has such a neighborhood, so the finite and infinite loci
are both open.  Connectedness of $U$ forces one of them empty.  In the finite
case the restriction to every line segment is locally concave, hence concave
on the segment, and $\tau$ is globally concave.
\end{proof}

\subsection{Theorem C}\label{SubSecTheoremC}

\begin{mainthm}
\label{mainTmB}
Fix a symplectic embedding $i\colon D\hookrightarrow M$ and suppose
$\beta(D)<\infty$. Then the invariant
$\tau(D,\theta,i,M)$ has the following properties.
\begin{enumerate}
\item
It is independent of the regular S-shaped acceleration model and of
the sufficiently fine discretization of the relative filtration. In
the geometrically finite noncompact case, the asymptotic coefficient
at infinity is kept fixed.

\item
For fixed $D$, $i$, and $M$, it depends on the symplectic data only
through the relative cohomology class
\[
[\omega,\theta]\in H^2(M,D;\bR).
\]

\item
Fix a Liouville primitive $\theta$, and let $U\subset H^1(D;\bR)$ be a
connected convex open set consisting of classes $[\eta]$ of closed
one-forms for which $\theta+\eta$ is again a Liouville primitive.
Writing $\tau_\theta$ for the corresponding family of invariants, the
map
\[
\tau_\theta\colon U\longrightarrow\bR\cup\{+\infty\}
\]
is either identically $+\infty$, or finite everywhere and concave.

\item
Let $D_1\subset D_2$ be an exact inclusion of Liouville domains with
compatible primitives and
\[
\beta(D_1),\beta(D_2)<\infty.
\]
If the intrinsic Viterbo restriction map
\[
SH^*(D_2;\bQ)\longrightarrow SH^*(D_1;\bQ)
\]
is injective, then
\[
\tau(D_1)\leq\tau(D_2).
\]

\item
The invariant is unchanged by Liouville rescaling:
\[
\tau(t\cdot D)=\tau(D)
\]
whenever the rescaled domain is defined.
\end{enumerate}
\end{mainthm}

\begin{proof}
    \begin{enumerate}
        \item \emph{(invariance)}
Proposition~\ref{prpTauComparison}\textup{(a)} gives independence of the
regular S-shaped model. Canonical width independence of
$\tau(D,\theta,M)$ follows from Proposition~\ref{prpTauComparison}\textup{(f)}.
In the geometrically finite noncompact case all comparisons use the fixed
asymptotic coefficient $\epsilon_\infty$.
        \item \emph{(homology)}
By Proposition~\ref{prpTauComparison}\textup{(d)} together with the
canonical definition of $\tau$, the invariant depends on the symplectic
data only through the relative class $[\omega,\theta]$.
        \item \emph{(concavity)} This is Proposition~\ref{prpTauConcave}.
        \item \emph{(monotonicity)} For an exact inclusion $D_1\subset D_2$, choose primitives so that after an exact change on $D_1$ one has $\theta_2|_{D_1}=\theta_1$. This exact change does not affect $\tau$ by Proposition~\ref{prpTauComparison}\textup{(d)}. Proposition~\ref{prpTauComparison}\textup{(f)}, then \textup{(c)}, then \textup{(e)}, then \textup{(f)} again yield
        \[
        \tau(D_1)\leq\tau(D_2),
        \]
        using that the existing Viterbo/intrinsic identification identifies
        $SH^*_{\widehat{D}_2}(D_2)\to SH^*_{\widehat{D}_2}(D_1)$ with the
        intrinsic Viterbo restriction $SH^*(D_2)\to SH^*(D_1)$, so the
        injectivity hypothesis of Theorem~\ref{mainTmB} supplies the
        hypothesis of clause~\textup{(e)}.
        \item \emph{(constancy)} This is Corollary \ref{cyTauSkeleton} applied to Liouville shrinkings $tD$, which have the same skeleton.
    \end{enumerate}
\end{proof}

\begin{rem} 
    There is also a version of monotonicity for non-exact inclusions. The point of the exactness assumption is that Theorem \ref{mainTmB} talks about the invariant $\tau(D,\theta,M)$ for $\theta$ a Liouville form. 
\end{rem}

% Flux section 
\subsection{Embedding torsion and flux}\label{SecFlux}
%\subsection{Torsion on families of embeddings}
\begin{df}\label{dfSymplecticEmbedding}
Let $(D,\theta)$ be a Liouville domain. A symplectic embedding  $\iota:(D,\theta)\to (M,\omega)$ is an open embedding such that $\iota^*\omega=d\theta$. An \emph{isotopy} of symplectic embeddings is a smooth family $s\mapsto \iota_s$ of symplectic embeddings. 
\end{df}

An isotopy $\iota_s$ gives rise to an element 
\begin{equation}\label{eqFluxDefinition}
Flux_{\{\iota_s\}}\in Hom(H_1(D;\bZ),\bR)=H^1(D;\bR)
\end{equation}
via
\begin{equation}\label{eqFluxFormula}
Flux_{\{\iota_s\}}(\gamma)=\int_{tr_{\{\iota_s\}}(\gamma)}\omega
\end{equation}
where $tr_{\iota_s}(\gamma)$ is the cylinder $(s,t)\mapsto \iota_s(\gamma(t))$, oriented by $ds\wedge dt$.

\begin{lm}\label{lmFluxWellDefined}
$Flux_{\{\iota_s\}}$ depends only on the homology class of $\gamma$ in $H_1(D;\bZ)$ and on the endpoint preserving homotopy class of $\{\iota_s\}$. 
\end{lm}
\begin{proof}
Let $\gamma_1$ and $\gamma_2$ be homologous smooth loops in $D$, and let
$C$ be a smooth singular $2$-chain in $D$, or equivalently an oriented
compact surface mapped to $D$, with
\[
\partial C=\gamma_1-\gamma_2.
\]
For
\[
F:[0,1]\times D\to M,\qquad F(s,x)=\iota_s(x),
\]
apply Stokes to $F_*([0,1]\times C)$. Since
\[
\iota_0^*\omega=\iota_1^*\omega=d\theta,
\]
the endpoint integrals over $C$ cancel, giving equality of the flux
integrals of $\gamma_1$ and $\gamma_2$. The second part follows by a
homotopy argument: for a fixed $\gamma$, varying the isotopy with fixed
endpoints amounts to varying $tr(\gamma)$ homotopically with fixed
endpoints.
\end{proof}

\begin{df}\label{dfStraightIsotopy}
For $s\in[0,1]$, write
\[
Flux_s
:=
Flux_{\{\iota_\tau\}_{0\le\tau\le s}}.
\]
The isotopy $\iota_s$ is called \emph{straight} if
\[
s\longmapsto Flux_s
\]
is affine.
\end{df}

\begin{lm}\label{lmStraightIsotopy}
Suppose $\beta(D)<\infty$.  For any straight isotopy $\iota_t$,
\[
t\longmapsto\tau(\iota_t(D))
\]
is concave.
\end{lm}
\begin{proof}
Fix $t_0$. Choose $0<r<1$ sufficiently close to $1$, and then $\epsilon>0$, such that for $|t-t_0|<\epsilon$ the compact Liouville subdomain
\[
K_t:=\iota_t(rD)
\]
satisfies $K_t\Subset D_0:=\iota_{t_0}(D)$. Let $\lambda_t:=(\iota_t^{-1})^*\theta$ be the natural primitive on $\iota_t(D)$, hence on $K_t$.

For $t$ sufficiently close to $t_0$, the inclusion $K_t\hookrightarrow D_0$ is a homotopy equivalence: $K_t$ is a small ambient isotopic perturbation of the standard shrinking $\iota_{t_0}(rD)$, which is a deformation retract of $D_0$.

Choose a smooth family of primitives $\theta_t$ of $\omega$ on $D_0$ such that $\theta_t|_{K_t}=\lambda_t$. To see that such a family exists, start from $\theta_{t_0}$; the closed form $\lambda_t-\theta_{t_0}|_{K_t}$ determines a class in $H^1(K_t)$; use the isomorphism $H^1(D_0)\cong H^1(K_t)$ to extend that class by a closed one-form on $D_0$; correct the remaining exact difference by extending a primitive function; and choose these data smoothly for $t$ in the small interval.

The extensions may be chosen so that
\[
\theta_t\to\theta_{t_0}
\]
in $C^1$ as $t\to t_0$.  After decreasing $\epsilon$, the Liouville
vector field of $\theta_t$ therefore remains transverse to the compact
cobordism
\[
D_0\setminus\operatorname{int}K_t
\]
with the same sign as at $t=t_0$, and has no trapped trajectories
there. Consequently
\[
\operatorname{Skel}(D_0,\theta_t)
\subset\operatorname{int}K_t.
\]

Then $K_t$ is Liouville for $\theta_t$. By
Proposition~\ref{prpTauComparison}\textup{(c)} and~\textup{(f)},
\[
\tau(\iota_t(D),\lambda_t)
=
\tau(D_0,\theta_t,M).
\]

Now relate $[\theta_t-\theta_{t_0}]$ to flux. Use loops contained in the shrunk copy $rD$. The trace cylinder of such a loop lies in $D_0$ for $t$ close to $t_0$. Stokes' theorem gives, with the orientation convention of \eqref{eqFluxFormula},
\[
[\theta_t-\theta_{t_0}]
=
-\operatorname{Flux}_{t_0\to t}
\]
under the natural identification $H^1(D_0)\cong H^1(D)$. Since the isotopy is straight, this cohomology class depends affinely on $t$. Theorem~\ref{mainTmB}\textup{(3)} therefore makes $t\mapsto\tau(\iota_t(D))$ locally concave. Since $t_0$ is arbitrary, it is concave on the whole parameter interval.
\end{proof}

\subsection{The star shape and the locality spectral sequence}\label{SubSecStarShape}
Following~\cite{ShelukhinEtAl}, define the \emph{star shape}
$Sh^*(D,M)\subset H^1(D;\bR)$ to be the set of fluxes of straight
isotopies in the sense of Definition~\ref{dfStraightIsotopy}. Let
\[
\delta:H^1(D;\bR)\longrightarrow H^2(M,D;\bR)
\]
be the coboundary map, and define the \emph{dual cone}
\[
C^*(D,M)
:=
\bigl\{
\beta\in H_2(M,D;\bZ)
:
\langle[\omega,\theta]+\delta(v),\beta\rangle\geq 0
\text{ for every }v\in Sh^*(D,M)
\bigr\}.
\]

\begin{prp}[Star-shape criterion]\label{tmTauConstant}
Assume $\beta(D)<\infty$.  Let
\[
\partial:H_2(M,D;\bZ)\longrightarrow H_1(D;\bZ)
\]
be the boundary map, and write
\[
\Gamma:=H_1(D;\bZ)/\operatorname{Tor}.
\]
If
\[
\partial C^*(D,M)\subset\operatorname{Tor}H_1(D;\bZ),
\]
then for every sufficiently fine discretization of the relative
filtration, the first nonzero higher differential in the locality
spectral sequence \eqref{eqLocalitySpectralSequence}, if it exists, preserves the
$\Gamma$-grading.
Moreover, if $M$ is exact, the spectral sequence degenerates at $E_1$.
\end{prp}
In particular, if $H_1(D;\bZ)$ is torsion-free, this is preservation of the
full $H_1(D;\bZ)$-grading.

The hypothesis
$\partial C^*(D,M)\subset\operatorname{Tor}H_1(D;\bZ)$
holds whenever the star shape contains either a line whose direction has
trivial annihilator in $\Gamma$, or rays whose nonnegative span is all of
$H^1(D;\bR)$.

\begin{proof}[Proof of Proposition~\ref{tmTauConstant}]
Near $0$, choose a local section of primitives so that
\[
[\sigma(v)-\theta]=-v.
\]
Then
\[
[\omega,\sigma(v)]
=
[\omega,\theta]+\delta(v),
\]
and hence
\[
\langle[\omega,\sigma(v)],\alpha\rangle
=
\langle[\omega,\theta]+\delta(v),\alpha\rangle.
\]
By Proposition~\ref{prpTauComparison}\textup{(d)}, $\tau$ descends to a neighborhood of $0$ in
$H^1(D;\bR)$.

Set $\tau_0:=\tau(D,\theta,M)$.  If $\tau_0=+\infty$, every higher
differential vanishes and there is nothing to prove; assume
$\tau_0<+\infty$.

Choose $K=tD$ sufficiently small that Lemma~\ref{lmFamilySDR} applies, and
identify $\Gamma_K\cong\Gamma_D=:\Gamma$. Let
\[
d_{\mathrm{def}}^P
\]
be the transferred differential on $H_K\widehat\otimes R(P)$ furnished by
Lemma~\ref{lmFamilySDR}, over a small convex polytope $P\ni0$. By
Lemma~\ref{lmFamilyOrderZero}, $H_K$ carries a $\Gamma$-grading for which
the order-zero differential and special deformation-retract data are
grading-preserving.  Accordingly, write
\[
d_{\mathrm{def}}^P
=
\sum_{s\in\Gamma}d_s^P,
\qquad
c(s)
:=
\val_\theta\bigl(\operatorname{sp}_0(d_s^P)\bigr).
\]

For every relative class $\alpha$ occurring in the block $d_s^P$ one has
$\overline{\partial\alpha}=s\in\Gamma$, where the bar denotes passage
modulo torsion.  Variation of the primitive therefore gives
\[
\operatorname{sp}_v(d_s^P)
=
T^{\langle v,s\rangle}
\operatorname{sp}_0(d_s^P),
\]
and locally
\[
\tau(v)
=
\inf_{s\in\Gamma}\bigl(c(s)+\langle v,s\rangle\bigr).
\]

Since  $P$ contains a dual ball of radius
$\rho>0$ and since $\tau(v)\ge0$, this formula yields $c(s)\ge\rho\|s\|$.
Since $\Gamma$ is discrete, it follows that $c$ has finite sublevel sets and the infimum is attained. In particular,
\[
\tau_0=\min_{s\in\Gamma}c(s).
\]

Let $s\in\Gamma$ satisfy $c(s)=\tau_0$.  Near $0$ one has
$\tau(v)\le\tau_0+\langle v,s\rangle$.  For $v\in Sh^*(D,M)$, choose a
straight isotopy with flux $v$.  By Lemma~\ref{lmStraightIsotopy}, $\tau$
is concave along this isotopy, so the supporting-affine inequality
propagates from a neighborhood of $0$ along the entire segment:
\[
\tau_0+\langle v,s\rangle
\ge
\tau(v)
\ge0.
\]
Suppose $\alpha$ is a relative class occurring with nonzero coefficient in the
$s$-block.  Then $\overline{\partial\alpha}=s$ and
$\langle[\omega,\theta],\alpha\rangle\ge\tau_0$.  For every
$v\in Sh^*(D,M)$,
\[
\begin{aligned}
\langle[\omega,\theta]+\delta(v),\alpha\rangle
&=
\langle[\omega,\theta],\alpha\rangle+\langle v,s\rangle\\
&\ge
\tau_0+\langle v,s\rangle\\
&\ge0.
\end{aligned}
\]
Hence $\alpha\in C^*(D,M)$.  The hypothesis forces $\partial\alpha$ to be
torsion, so
\[
s=\overline{\partial\alpha}=0\in\Gamma.
\]

Discreteness of $\Gamma$ now supplies $\Delta>0$ with
$s\neq0\Rightarrow c(s)\ge\tau_0+\Delta$.  Choose
$0<\hbar_{\mathrm{disc}}<\Delta$ sufficiently small for the locality
spectral sequence.  Let $d_r$ be the first nonzero higher differential.
Then
\[
r\hbar_{\mathrm{disc}}
\le
\tau_0
<
(r+1)\hbar_{\mathrm{disc}}
<
\tau_0+\Delta.
\]
Therefore every term contributing to $d_r$ has $\Gamma$-shift zero, and
$d_r$ preserves the $\Gamma$-grading.
As in Proposition~\ref{prpTauComparison}\textup{(f)}, the same-skeleton restriction
$tD\subset D$ identifies the locality spectral sequences from $E_1$
onward and preserves residual valuation and the $\Gamma$-grading, so the
conclusion holds for $D$.

If $M$ is exact and $\tau_0<\infty$, choose a minimizing shift.  The
preceding argument gives $s=0\in\Gamma$.  Choose a relative class
$\alpha$ in the corresponding leading block.  Then $\partial\alpha$ is
torsion; pick $N>0$ with $N\partial\alpha=0$.  The class $N\alpha$ lies in
the kernel of
$\partial:H_2(M,D;\bZ)\to H_1(D;\bZ)$, and by the long exact sequence of
the pair lifts to some $A\in H_2(M;\bZ)$.  Exactness yields
\[
N\langle[\omega,\theta],\alpha\rangle
=
\int_A\omega
=
0,
\]
so $\langle[\omega,\theta],\alpha\rangle=0$.  This contradicts the
positive leading-order estimate
$\langle[\omega,\theta],\alpha\rangle\ge\tau_0>0$
(since $\tau_0$ is the valuation of a nonzero higher differential, hence
the first positive filtration order).  Thus $\tau_0=\infty$ and all
higher differentials vanish.
\end{proof}

\appendix

\section{$C^0$ estimates for $S$-shaped Hamiltonians. }\label{SecEstimates}

In this section we develop the packing and topological estimates used in the proof of Proposition~\ref{PrpRelHbar}.

Fix a geometrically bounded reference metric $g_0$ on $M$.
For a bounded $C^0$-set $B$ of compatible almost complex
structures, there is a constant $C_B\geq 1$ such that
\[
C_B^{-1}g_0\leq g_J\leq C_Bg_0
\]
for every $J\in B$.
All monotonicity estimates below are applied using $g_0$ as
the bounded-geometry reference metric and the second clause of
Lemma~\ref{lmMonEst++}. Consequently, all constants are uniform
for $J\in B$.

\subsection{Gromov's trick for Floer trajectories}
Let $\Sigma$ be a Riemann surface with $n$ input punctures, one output puncture, all endowed with cylindrical ends. We denote the complex structure on $\Sigma$ by $j_{\Sigma}$. For $i=1,2$ let $\pi_i$ be the projections to $\Sigma$ and $M$ respectively.

Let $H:\Sigma\times M\to\bR$ be a function which is $s$ independent on the ends. Let $z\mapsto J_z$ be a $\Sigma$-dependent family of almost complex structures on $M$ which is $s$ independent on the ends. We introduce the notation $d_{\Sigma}H$ for the differential of $H$ in the $\Sigma$ directions. In local coordinates on $\Sigma$ it is given as $d_{\Sigma}H=\frac{\partial H}{\partial s}\pi_1^*ds+\frac{\partial H}{\partial t}\pi_1^*dt$. The $C^{\infty}$ norm $\|d_{\Sigma}H\|_{\infty}$ measures the maximum magnitude of the derivatives of $H$ with respect to the metric $g_{\Sigma}$. We fix a closed $1$-form $\alpha$ on $\Sigma$ such that $\alpha$ is a multiple
of $dt$ on the cylindrical ends and
\[
d_\Sigma H\wedge\alpha\geq 0.
\]

To this datum we associate an almost complex structure $J_{H}$ on $\Sigma\times M$ defined by
\begin{equation}\label{eqGrTrick}
J_{H}:=J_M+j_{\Sigma}+X_H\otimes \alpha\circ j_{\Sigma}- JX_H\otimes \alpha.
\end{equation}

Fix an area form $\omega_\Sigma$ on $\Sigma$ which is compatible with $j_\Sigma$, and which coincides with the form $ds\wedge dt$ on the ends. The closed form
\begin{equation}\label{eqGrTrick2}
\omega_{H}:=\pi_1^*\omega_{\Sigma}+\pi_2^*\omega+ dH\wedge \alpha
\end{equation}
on $\Sigma\times M$ can be shown to be symplectic and  $J_{H}$ is compatible with it. We denote the induced metric on $\Sigma\times M$ by $g_{J_{H}}$. We refer to the metric $g_{J_H}$ as the \emph{Gromov metric}. We stress that \emph{the Gromov metric depends on the choice of area form on $\Sigma$}. Set
\[
A_\alpha:=\max\{1,\|\alpha\|_\infty\}.
\]
We will also consider the form $\omega_{H,\tau}:=\tau^2\pi_1^*\omega_{\Sigma}+\pi_2^*\omega+ dH\wedge \alpha$ for $\tau>0$. 

Given a pair of Riemannian metrics $g_1,g_2$ on a smooth manifold, we say they are \emph{$C$-equivalent} for some $C>1$ if $\frac1{C}\|v\|^2_{g_2}<\|v\|^2_{g_1}<C\|v\|^2_{g_2}$. 

\begin{lm}\label{lmProdCEquivalence}
There is a continuous function $f:\bR_+^2\to\bR_+$ converging to $1$ at $(0,0)$ such that writing
\[
C_{\tau}
:=
f\!\left(
\frac{A_\alpha\|d_{\Sigma}H\|_{\infty}}{\tau^2},
\frac{A_\alpha\|X_H\|_{\infty}}{\tau}
\right),
\]
the metric $g_{J_H}$ determined by $J_H$ and $\omega_{\Sigma,\tau}$ is $C_{\tau}$-equivalent to the product metric $\tau^2\pi_1^*g_\Sigma+\pi_2^*g_J$.  
\end{lm}
\begin{proof}
 A tedious but straightforward calculation shows that the metric $g_{J_H,\tau}$ determined by $J_H$ and $\omega_{H,\tau}$ is given by the formula 
\begin{equation}\label{eqGromovMetricFormula}
g_{J_H,\tau}=\tau^2\pi_1^*g_{\Sigma}+\pi_1^*(d_{\Sigma}H\wedge\alpha)(\cdot,j_{\Sigma}\cdot)+\|X_H\|^2_{g_J}(\pi_1^*\alpha)^2-g_J(X_H,\cdot)\alpha +\pi_2^*g_J.
\end{equation}
The term $(d_{\Sigma}H\wedge\alpha)(\cdot,j_{\Sigma}\cdot)$ can be written as $\kappa_Hg_{\Sigma}$ for an appropriate function $\kappa_H$ determined by $d_{\Sigma}H$ and $\alpha$. Importantly,
\[
0\leq\kappa_H\leq A_\alpha\|d_{\Sigma}H\|_{\infty},
\qquad
|\mu|\leq A_\alpha,
\]
where $\mu$ is as below.

To relate the metric $g_{J_H,\tau}$ with the product metric we split $$T(\Sigma\times M)=\left(T\Sigma\oplus\bR X_H\right)\oplus (\bR X_H)^{\perp},$$ 
where  $(\bR X_H)^{\perp}$ denote the orthogonal complement in $TM$ with respect to $g_{J_H}$. Note this is also an orthogonal splitting with respect to $g_J$. Let $p_1, p_2$ be the orthogonal projections associated with the splitting. Note that our claim is a pointwise property. We thus pick a pointwise $g_{\Sigma}$-orthonormal basis $\frac{\partial}{\partial s},\frac{\partial}{\partial t}$ for $T_z\Sigma$ so that $\alpha(\frac{\partial}{\partial s})=0$. Let $\mu:=\alpha(\frac{\partial}{\partial t})$ be the evaluation of $\alpha$ on $\frac{\partial}{\partial t}$. If $X_H=0$ at $z$, then
\[
g_{J_H,\tau}
=
(\tau^2+\kappa_H)g_\Sigma\oplus g_J,
\]
so the comparison with $\tau^2g_\Sigma\oplus g_J$ is controlled by
$1+\kappa_H/\tau^2$, hence by the same function $f$. Thus one may assume
$X_H\neq 0$ at $z$. Consider on the first summand  the basis $\left(\frac1{\tau}\frac{\partial}{\partial s},\frac1{\tau}\frac{\partial}{\partial t}, \frac{X_H}{\|X_H\|}\right),$ which is orthonormal with respect to $\tau^2g_{\Sigma}+g_J$. Let $A_{\tau}$ be the Gram matrix of the restriction of $g_{J_H,\tau}$ to the first summand. It is   the matrix
$$
\begin{pmatrix}1+\frac{\kappa_H}{\tau^2}& 0& 0\\
0& 1+\frac{\kappa_H+\mu^2\|X_H\|^2}{\tau^2}& -\mu\frac{\|X_H\|}{\tau}\\
0& -\mu\frac{\|X_H\|}{\tau}& 1
\end{pmatrix}
$$
Then $g_{J_{H},\tau}$ is represented with respect to the product metric by the linear map $A\circ p_1+p_2$. For $x,y\geq 0$ consider the matrix $A(x,y)$
$$
\begin{pmatrix}1+x& 0& 0\\
0& 1+x+y^2& -y\\
0& -y& 1
\end{pmatrix}
$$
and let $h(x,y)=\max\left\{\lambda_{max},\frac 1{\lambda_{min}}\right\} $ where we refer to the maximal and minimal eigenvalue of the positive definite symmetric matrix $A(x,y)$. Let $f(s,t):=\sup\{h(x,y):x,y\geq 0, x\leq s,y\leq t\}$. Then $f(s,t)\to 1$ as $(s,t)\to (0,0).$

Moreover, since $0\leq \kappa_H\leq A_\alpha\|d_{\Sigma}H\|_{\infty}$ and $|\mu|\leq A_\alpha$, and since the Gram matrix is controlled by $(x,y)=(\kappa_H/\tau^2,\mu\|X_H\|/\tau)$, we have that for $C_{\tau}$ as in the statement of the lemma the metric $g_{J_H}$ is $C_{\tau}$-equivalent to the corresponding product metric.

\end{proof}

An observation known as \emph{Gromov's trick} is that $u$ is a solution to Floer's equation
\begin{equation}\label{eqGromovTrickFloer}
(du-X_H\otimes\alpha)^{0,1}=0,
\end{equation}
if and only if its graph $\tilde{u}$ satisfies the Cauchy Riemann equation
\begin{equation}
\overline{\partial}_{J_{H}}\tilde{u}=0.
\end{equation}
Thus Floer trajectories can be considered as $J_{H}$-holomorphic sections of $\Sigma\times M\to\Sigma$. 
To a Floer solution $u:\Sigma\to M$ and a subset $S\subset\Sigma$ we can now associate three different non-negative real numbers:
\begin{itemize}
\item The \emph{geometric energy} $E_{\mathrm{geo}}(u;S):=\frac12\int_S\|du-X_H\otimes\alpha\|^2$.
\item The \emph{topological energy} $E_{\mathrm{top}}(u;S):=\int_S u^*\omega+\int_S d H\circ\tilde{u}\wedge\alpha$.
\item The \emph{symplectic energy} $E(\tilde{u};S):=\int_S\tilde{u}^*\omega_{H,\tau}$. 
\end{itemize}
For a solution of Floer's equation one has
\[
E_{\mathrm{top}}(u;S)
=
E_{\mathrm{geo}}(u;S)
+
\int_S d_\Sigma H\wedge\alpha .
\]
Consequently,
\[
\int_S\widetilde u^*\omega_{H,\tau}
=
E_{\mathrm{geo}}(u;S)
+
\tau^2\operatorname{Area}(S)
+
\int_S d_\Sigma H\wedge\alpha .
\]
In particular, if $d_\Sigma H\wedge\alpha\geq0$, then
\[
E_{\mathrm{geo}}(u;S)\leq E_{\mathrm{top}}(u;S).
\]

 The key to obtaining $C^0$ estimates is the monotonicity lemma  \cite{Sikorav94} .  For a $J$-holomorphic map $u:\Sigma\to M$ and for a measurable subset $U\subset \Sigma$ write
\begin{equation}
E(u;U):=\int_Uu^*\omega.
\end{equation}
\begin{lm}\label{lmMonEst++}[\textbf{Monotonicity} \cite{Sikorav94} ]
Fix a compatible almost complex structure $J$ on $M$. Let $a$ be a constant and let $p\in M$ be a point at which $|Sec|\leq a^2$ on the ball $B_{1/a}(p)$ and such that at $p$ the injectivity radius is $\geq\frac1{a}$. 
 Let $S$ be a compact Riemann surface with boundary and let $u:S\to M$ be a nonconstant $J$-holomorphic map such that $p$ is in the image of $u$ and such that
\[
u(\partial S)\cap B_{1/a}(p)=\emptyset.
\]
Then there is a universal constant $c$ such that
\begin{equation}
E\left(u;u^{-1}(B_{1/a}(p))\right)\geq\frac{c}{a^2}.
\end{equation}
If, instead, we only require that there exists a constant $C>1$ and a Riemannian metric $h$ satisfying the above bounds on the sectional curvature and injectivity radius  $g_J$ and such that 
$$
\frac1{C}g_J(v,v)\leq h(v,v)\leq Cg_J(v,v),
$$
we get the inequality
\begin{equation}
E\left(u;u^{-1}(B_{1/a}(p))\right)\geq\frac{c}{C^3a^2}.
\end{equation}
 \end{lm}

\section{Virtual counts and locality}
\label{SecVirtualChains}

Fix a field $\bK$ of characteristic zero and a theory of virtual
counts in the sense of \cite[Definition~1.5]{Abouzaid2022},
satisfying the standard normalization of
\cite[Definition~1.7]{Abouzaid2022}. Such normalized symmetric
theories exist over $\bK$ by \cite[Theorem~1.9]{Abouzaid2022}.
Abouzaid constructs a multifunctor
\[
\mathbb X:\mathcal F\longrightarrow \mathrm{Flow}
\]
from Hamiltonian Floer data to Kuranishi flow categories and flow
multimodules \cite[Theorem~10.3]{Abouzaid2022}. In the geometrically
bounded noncompact case we use dissipative Floer data, as in
\cite[Remark~10.2]{Abouzaid2022}. A theory of virtual counts then
determines a multifunctor $\mathrm{Flow}\to\mathrm{Ch}_{\bK}$, and
hence a Floer multifunctor.

For fixed asymptotic orbits and energy label $\lambda$, write
\[
\overline{\mathcal M}_{\lambda}
(\mathfrak d;\gamma_1,\ldots,\gamma_k;\gamma_0)
\]
for the corresponding compactified moduli space. Its points are stable
Floer configurations: they may contain broken Floer levels and attached
sphere components. The energy label is the total topological energy,
including the energies of attached sphere bubbles; compare
\cite[\S11.2]{Abouzaid2022}. The image of a stable configuration means
the union of the images of its Floer components and its nonconstant
sphere components. Constant ghost components are ignored. The colimit
of the footprints of the associated Kuranishi presentation is naturally
identified with this compactified moduli space
\cite[Proposition~11.4]{Abouzaid2022}.

The cubical and multi-input constructions are compatible with faces and
multicomposition \cite[Proposition~12.9]{Abouzaid2022}. Whenever an
operation is parametrized by a compact parameter cycle, we fix a finite
cubical chain representing that cycle and take the corresponding finite
sum of cube operations. In particular, the $S^1$-family defining the BV
operator is represented by a finite cubical $1$-cycle.

\begin{lm}[Support of virtual counts]
\label{lmVirtualSupport}
Let $\mathfrak d$ be a labeled Floer multimorphism, with fixed
asymptotic orbits and energy label $\lambda$. If
\[
\overline{\mathcal M}_{\lambda}
(\mathfrak d;\gamma_1,\ldots,\gamma_k;\gamma_0)
=\varnothing,
\]
then the corresponding virtual operation vanishes.
\end{lm}

\begin{proof}
This follows from the definition of a theory of virtual counts in
\cite{Abouzaid2022}. The stratum complex is generated by the nonempty
strata of the corresponding Kuranishi presentation
\cite[\S2.5]{Abouzaid2022}, while
\cite[Proposition~11.4]{Abouzaid2022} identifies the colimit of its
footprints with the compactified moduli space of actual solutions.
Thus, for the fixed asymptotics and energy label above, an empty
compactified moduli space contributes no stratum to the corresponding
summand of the stratum complex, and its virtual count is zero.
\end{proof}

\begin{df}[Floer-local theory of virtual counts]
\label{dfLocalVirtualCounts}
A theory of virtual counts is called \emph{Floer-local} if the virtual
operations obtained from the Hamiltonian Floer multifunctor satisfy
the following germ-invariance property.

Suppose $\mathfrak d_i$ are two labeled Floer multimorphisms on
symplectic manifolds $M_i$. Assume that there are open sets
$U_i\subset M_i$ containing the images of all stable configurations in
the corresponding compactified moduli spaces and a symplectomorphism
\[
\phi:U_1\longrightarrow U_2
\]
which identifies the asymptotic orbits and the restrictions of the
Floer data, and induces an identification of the compactified moduli
spaces compatible with orientations, energy labels, boundary strata,
gluing maps, and cubical face maps. Then the corresponding virtual
operations are identified by the induced maps of orientation lines.

The same condition is imposed on flow multimodules and is required
cube-by-cube for finite cubical chains, compatibly with faces and
multicomposition.
\end{df}

\begin{rem}[Status of germ invariance]
    \label{remVirtualGermStatus}
    Germ invariance is an additional hypothesis; it is not one of the
    axioms for a theory of virtual counts stated in
    \cite{Abouzaid2022}, and we do not prove it here.
    
    The Hamiltonian Floer construction of \cite[\S\S11--12]{Abouzaid2022}
    nevertheless suggests a version of this property that depends on a neighborhood.
    Given an open neighborhood $U$ containing the images of the
    relevant stable solutions, the finite-dimensional inhomogeneous
    perturbations used in the Kuranishi thickenings can be localized in
    the target by cutoffs which are equal to one near these images.  The
    thickened charts may then be restricted to perturbed maps whose images
    remain in $U$.  Thus, when two Floer problems agree on $U$, one can
    make the same local thickening choices for the two problems; at this
    level the resulting local Kuranishi presentations are identified, and
    hence give the same virtual counts.
    
    This is a neighborhood-dependent, or $\varepsilon$-local, version of
    the germ-invariance property above.  To obtain germ invariance as
    stated one must in addition show that the resulting counts are
    independent, in a manner compatible with the Floer multifunctor, of
    the chosen localization neighborhood and the corresponding local
    thickening data.  Such a comparison is suggested both by
    parameterized versions of the Floer construction and by the excision
    properties of the underlying virtual-cochain construction, but we do
    not establish it here.  We therefore retain germ invariance as an
    explicit hypothesis.
    \end{rem}

\begin{prp}[Virtual locality principle]
\label{prpVirtualLocality}
Let $\mathfrak d$ define a differential, continuation map, continuation
homotopy, or fixed-arity Floer operation, possibly over a finite cubical
chain of parameterized Floer data.
\begin{enumerate}
\item
A nonzero virtual coefficient with fixed asymptotics and energy label
has a nonempty compactified moduli space with the same labels.

\item
Suppose every stable configuration with fixed asymptotics and energy
label has image contained in an open set $U$, and suppose that on $U$
the Floer problem is identified with a second Floer problem in the sense
of Definition~\ref{dfLocalVirtualCounts}. Then the corresponding virtual
coefficients agree.
\end{enumerate}
The assertions apply termwise to finite cubical chains and to
multi-input operations.
\end{prp}

\begin{proof}
The first assertion is Lemma~\ref{lmVirtualSupport}. The second is germ
invariance. For a finite cubical chain the operation is the finite sum
of the operations associated to its cubes, so both statements apply
cube-by-cube. The multi-input statement follows from the corresponding
multimodule version of germ invariance.
\end{proof}

\bibliographystyle{plain}
\bibliography{sections/ref}
\end{document}